\newif\ifarxiv\arxivfalse

\RequirePackage[dvipsnames]{xcolor}
\ifarxiv
	\documentclass[10pt, twoside, reqno]{amsart}
\else
	\RequirePackage{amsmath}
	\documentclass[smallextended]{svjour3}
	\usepackage[11pt]{extsizes}
	\usepackage[a4paper, total={16cm, 23cm},left=2cm,right=2cm,bottom=3cm,top=2cm]{geometry}
\fi
	\PassOptionsToPackage{noend}{algpseudocode}
	\PassOptionsToPackage{bookmarksdepth=4}{hyperref}
	
	\ifarxiv\else
		\makeatletter
			\let\cl@chapter\undefined
		\makeatother

		\smartqed
	\fi

	\usepackage{todonotes}
	\usepackage{bm}
	\usepackage{booktabs}
	\usepackage{multirow}
\usepackage[linesnumbered,ruled,vlined]{algorithm2e}
\RestyleAlgo{boxruled}
\SetInd{0.5em}{1em} 
    \usepackage[%
		eqreset=section,
		thmreset=section,
	]{myPreamble}

	\makeatletter
		\newcommand\mynum[1]{$ˆ{\@fnsymbol{#1}}$}
	\makeatother
	
	\CreateNewTheorem[Fact]{fact}[dummythm]{Fact}
	\ifarxiv\else
	    
	\fi
	
	\makeatletter
		\newcommand\D[1][h]{%
			\@ifnextchar_{\operatorname D}{%
				\ifstrempty{#1}{\operatorname D}{\operatorname D_{#1}}%
			}
		}
	\makeatother
	
\Crefname{equation}{}{}

\Crefname{item}{}{}{}

\newcommand{\TheKeywords}{%
	Nonsmooth nonconvex optimization,
        Paraconvex functions, 
        Projected subgradient methods,
        H\"{o}lderian error bound,
        Linear convergence,
        Robust low-rank matrix recovery.%
}
\newcommand{\TheSubjclass}{%
	90C06, 
	90C25, 
	90C26, 
	49J52, 
	49J53.
}
\newcommand{\TheTitle}{%
	Projected subgradient methods for paraconvex optimization: Application to robust low-rank matrix recovery%
}
\newcommand{\TheShortTitle}{%
	Projected subgradient methods for paraconvex optimization%
}
\newcommand{\TheShortAuthor}{%
	M. Rahimi, S. Ghaderi, Y. Moreau, and M. Ahookhosh%
}
\newcommand{\TheFunding}{%
	MR and MA acknowledge the support by the \emph{Research Foundation Flanders (FWO)} research project G081222N and
	\emph{UA BOF DocPRO4 projects} with ID 46929 and 48996;
	SG also acknowledges the support of the  \emph{Research Foundation Flanders (FWO)} research project 12AK924N.
}
\newcommand{\TheAddressUA}{%
		Department of Mathematics, University of Antwerp, Middelheimlaan 1, B-2020 Antwerp, Belgium}
\newcommand{\TheAbstract}{%

    This paper is devoted to the class of paraconvex functions and presents some of its fundamental properties, characterization, and examples that can be used for their recognition and optimization. Next, the convergence analysis of the projected subgradient methods with several step-sizes (i.e., constant, nonsummable, square-summable but not
    summable, geometrically decaying, and Scaled Polyak's step-sizes) to global minima for this class of functions is studied. In particular, the convergence rate of the proposed methods is investigated under paraconvexity and the H\"{o}lderian error bound condition, where the latter is an extension of the classical error bound condition.
    The preliminary numerical experiments on several robust low-rank matrix recovery problems (i.e., robust matrix completion, image inpainting, robust nonnegative matrix factorization, robust matrix compression, and robust image deblurring) indicate promising behavior for these projected subgradient methods, validating our theoretical foundations.
}

\begin{document}

\ifarxiv
	\title[\TheShortTitle]{\TheTitle}
	\author[\TheShortAuthor]{%
		Morteza Rahimi\textsuperscript{1},\ 
		Susan Ghaderi\textsuperscript{2}.\ 
		Masoud Ahookhosh\textsuperscript{2},\ and\ 

	}
	\thanks{\textsuperscript{1}\TheAddressKU}
	\thanks{\textsuperscript{2}\TheAddressMons}
	\thanks{\TheFunding}
	\subjclass{\TheSubjclass}


		\begin{abstract}
			\TheAbstract
		\end{abstract}

		\maketitle
        
\else

	\journalname{...}

	\title{\TheTitle\thanks{\TheFunding}}
	\titlerunning{\TheShortTitle}

	\author{%
            Morteza Rahimi\and
            Susan Ghaderi\and
            Yves Moreau\and
		Masoud Ahookhosh%
	}
	\authorrunning{\TheShortAuthor}
    
    \institute{%
        Corresponding author: Morteza Rahimi\\
		M. Rahimi and M. Ahookhosh
		\at
			\TheAddressUA.\\
			{\tt
					E-mail: \href{mailto:morteza.rahimi@uantwerpen.be}{morteza.rahimi@uantwerpen.be},%
					\href{mailto:masoud.ahookhosh@uantwerpen.be}{masoud.ahookhosh@uantwerpen.be}%
			}%
		\and
		S. Ghaderi, Y. Moreau%
		\at
			\TheAddressKU.\\
			{\tt
				    E-mail: \href{mailto:susan.ghaderi@esat.kuleuven.be}{susan.ghaderi@esat.kuleuven.be},%
                        \href{mailto:Yves.Moreau@kuleuven.be}{Yves.Moreau@kuleuven.be}%
			}%
	}%

	\maketitle

	\begin{abstract}
		\TheAbstract
		\subclass{\TheSubjclass}%
	\end{abstract}
\fi

\section{Introduction}\label{sec:Introduction}

Subgradient methods and their projected variants were originally introduced for solving convex nonsmooth problems in the 1960s; see, e.g., \cite{ermol1966methods,polyak1967general,polyak1987introduction,shor2012minimization}. Due to their simple structure and low memory requirement, these methods have been widely employed \cite{nesterov2009primal,nesterov2014subgradient,nesterov2024primal}, especially when the objective function does not admit a simple structure for splitting algorithms; cf. \cite{ahookhosh2018optimal,ahookhosh2017optimal,ahookhosh2018solving,neumaier2016osga}. Recent trends in application domains such as signal and image processing, machine learning, data science, statistics, and sparse optimization often involve nonsmooth and nonconvex functions. Representative examples include image deblurring \cite{liu2021efficient}, image inpainting \cite{yu2019new}, compressed sensing \cite{chartrand2008restricted}, sparse recovery \cite{wen2018survey}, matrix factorization \cite{chi2019nonconvex}, nonnegative matrix factorization \cite{gillis2020nonnegative}, matrix completion \cite{sun2016guaranteed}, tensor factorization \cite{welling2001positive}, phase retrieval \cite{eldar2014phase}, and deep neural networks \cite{jain2017non}. This motivates the development of efficient subgradient methods for large-scale nonconvex problems, which is the primary goal of this work.

In the current paper, we are interested in solving the nonsmooth and nonconvex optimization problem
\begin{equation}\label{prb2}
\min_{x\in X} f(x),
\end{equation}
where $f:\R^n\to \overline{\R}$ is proper, nonsmooth, and nonconvex, and $X\subseteq \dom(f)$ is nonempty, closed, and convex. We assume the set of optimal solutions $\mathcal{X}^*$ is nonempty and $f^*>-\infty$ is the optimal value. Compared to convex problems, nonsmooth and nonconvex problems are challenging due to (i) possible nondifferentiability of $f$, requiring nonempty Fr\'{e}chet/Mordukhovich/Clarke subdifferentials \cite{clarke1990optimization}, (ii) complex landscapes with minima, maxima, and saddle points, and (iii) local minimizers not necessarily being global.

First-order methods are suitable for problem \cref{prb2}, as they rely only on function values and (sub)gradients \cite{bertsekas2003convex,nesterov2018lectures,ruder2016overview}. Projected Subgradient Methods (PSMs), in particular, iterate as
\[
    \left\{ \begin{array}{l}
        \text{Select}~\zeta_k \in \partial f(x_k), \\
        x_{k+1} = \operatorname{proj}_X\left(x_k + \alpha_k \frac{\zeta_k}{\|\zeta_k\|}\right),
    \end{array} \right.
\]
where $\alpha_k>0$ is a predetermined step-size and $\operatorname{proj}_X$ is the projection onto $X$. This approach avoids function evaluations and has the step-size predefined before its implementation, making it suitable for large-to huge-scale problems. In the nonconvex setting, most studies of subgradient methods have only been carried out for limited classes of functions, such as quasiconvex functions \cite{kiwiel2001convergence,hu2020convergence,hu2015inexact,quiroz2015inexact} and weakly convex functions \cite{davis2018subgradient,li2021weakly}, along with references therein. This is our primary motivation to explore the convergence analysis of subgradient methods for larger classes of functions.

We here focus on the class of $\nu$-paraconvex functions with $\nu\in(0,1]$. A function $\func{h}{\R^n}{\Rinf}$ is said to be $\nu$-paraconvex if there exists $\rho\ge 0$ such that
\begin{equation}\label{eq:para000}
h(\lambda x+(1-\lambda)y) \le \lambda h(x) + (1-\lambda) h(y) + \rho \min\{\lambda,1-\lambda\}\|x-y\|^{1+\nu},
\end{equation}
for each $x,y\in X$ and $\lambda\in [0,1]$.
The class of $\nu$-paraconvex functions generalizes weakly convex functions (\(1\)-paraconvex functions) \cite{rolewicz1979paraconvex,rolewicz1979gamaparaconvex,vial1983strong}, with additional examples not satisfying weak convexity. The class of $1$-paraconvex functions arises in applications including optimal control, image processing, and machine learning \cite{cannarsa2004semiconcave,li2021weakly,davis2018subgradient}.
In particular, Rolewicz \cite{rolewicz1979paraconvex} demonstrated that \(1\)-paraconvexity of $h$ is equivalent to the convexity of the perturbed function \(x \mapsto h(x) + \rho \|x\|^2\) for some \(\rho \geq 0\), while this property does not hold for \(\nu \in (0,1)\), providing a counterexample in \cite{rolewicz1979gamaparaconvex}. Nevertheless, we establish that if the function \(x \mapsto f(x) + \rho \|x\|^{1+\nu}\) is convex for some \(\rho \geq 0\), then \(f\) qualifies as \(\nu\)-paraconvex.
These functions admit useful properties such as local Lipschitzness, subdifferentiability, and subdifferential monotonicity  \cite{jourani1996open,jourani1996subdifferentiability,van2008paraconvex}.

In absence of the strong convexity of the objective function, studying convergence under {\it error bound} conditions has a long history \cite{bolte2007lojasiewicz,bolte2017error,burke2002weak,burke1993weak,davis2018subgradient,ferris1991finite,johnstone2020faster,li2013global,luo1993error,pang1997error,xu2016accelerate,zhang2020new,zhou2017unified}. In particular, the {\it H\"olderian error bound} with order $\delta\in (0,1]$ provides an upper bound on the distance to the solution set in terms of the residual, with special cases including sharpness error bound (\(\delta=1\)), \cite{burke2002weak,burke1993weak,davis2018subgradient,polyak1979sharp}, and quadratic growth condition (\(\delta=\tfrac{1}{2}\)), \cite{karimi2016linear,zhang2020new}. Under these assumptions, PSMs with carefully chosen step-sizes can achieve linear convergence in certain settings \cite{davis2018subgradient}. Our work studies convergence rates of PSMs (with several different step-sizes) for $\nu$-paraconvex functions under H\"olderian error bounds.

\vspace{-3mm}
\subsection{{\bf Contribution}}
Our primary goal is to extend the applicability of PSMs to a broader class of nonconvex, nonsmooth optimization problems, particularly focusing on paraconvex functions.
We aim to provide comprehensive theoretical foundations to investigate convergence rates of the PSMs for such a class of functions and practical insights into this area. Our contribution has threefold, which we describe as follows:

\begin{itemize}
    \item[(i)] 
        {\bf Characterization of the class of paraconvex functions.} We first describe the class of paraconvex functions as a generalization of the class of weakly convex functions.
        Then, we verify that while preserving the concept of paraconvexity, the term 
        \(\min\{\lambda,1-\lambda\}\) can be removed from \cref{eq:para000}.
        Next, it is shown that the continuous midpoint $\nu$-paraconvexity implies $\nu$-paraconvexity, facilitating the assessment of paraconvexity. We particularly demonstrate a connection to a local version of $\nu$-paraconvex functions, known as $\nu$-weakly convex functions (see \cite{daniilidis2005filling}), coincide with $\nu$-paraconvex functions on compact convex sets.
        In addition, we establish the local Lipschitzness of such functions, ensuring the nonemptiness of their Clarke subdifferential.
        As a wider class than weakly convex functions, they accept saddle points. As such, under H\"{o}lderian error bound, we identify a region around the optimal set containing no saddle points.
    \item[(ii)] 
        {\bf Convergence rate of PSMs with several step-sizes.} We describe a generic framework for PSMs for paraconvex functions under the H{\"o}lderian error bound condition. In this setting, we study the convergence behavior of this subgradient method for several different step-sizes, such as constant, nonsummable diminishing, square-summable but not summable, geometrically decaying, and Scaled Polyak's step-sizes. We demonstrate recurrences to establish upper bounds for the sequence of iterations (i.e., $\{\dist(x^k,\mathcal{X}^*)\}_{k\in\N}$) and function values (i.e., $\{f(x^k)-f^*\}_{k\in\N}$) residuals. It is shown that: (i) for the PSM with {\it constant step-size}, $\{\dist(x^k,\mathcal{X}^*)\}_{k\in\N}$ is linearly convergent up to a specified tolerance; (ii) for the PSMs with nonsummable diminishing and square-summable but not summable step-sizes, the subsequential and global convergence are established, respectively; (iii) for the PSMs with geometrically decaying and Scaled Polyak's step-sizes, $\{\dist(x^k,\mathcal{X}^*)\}_{k\in\N}$ is linearly convergent.
    \item[(iii)] 
        {\bf Applications to low-rank robust matrix recovery.} We introduce a class of robust low-rank matrix recovery appearing in several applications such as robust matrix completion, image inpainting, robust nonnegative matrix factorization, robust matrix compression, and robust image deblurring. We apply the proposed PSM to these applications, and the results indicate the promising behavior of these algorithms.  To the best of our knowledge, this is the first theoretical and numerical appearance of the PSM with Scaled Polyak's step-size in the literature.
\end{itemize}

\vspace{-6mm}
\subsection{{\bf Organization}}

The remainder of the paper is organized as follows.
\Cref{sec:Preliminaries} introduces the notations and reviews foundational concepts.
In \Cref{sec:paraconvex}, we analyze the class of $\nu$-paraconvex functions and explore their fundamental properties.
\Cref{sec:subGradMethod} focuses on PSMs for $\nu$-paraconvex functions satisfying the H{\"o}lderian error bound condition, deriving key inequalities.
In \Cref{sec:subGradMethodWithConstantstep-size}, we investigate PSM with constant step-size.
\Cref{sec:subGradMethodWithDiminishingstep-size} discusses PSMs with diminishing step-size strategies.
\Cref{sec:Polyak subgradient method2} examines PSMs using Scaled Polyak's step-size.
In \Cref{sec:NumRes}, we present preliminary numerical results on robust low-rank matrix recovery validating our theoretical foundations.
Finally,  \Cref{sec:Conclusion} delivers our concluding remarks.


\section{Preliminaries}\label{sec:Preliminaries}

\subsection{{\bf Notation}}

In this paper, we denote the standard {\it inner product} and the {\it Euclidean norm} in $n$-dimensional real {\it Euclidean space}
$\R^n$ by $\langle \cdot,\cdot\rangle$ and 
$\|\cdot\|=\sqrt{\langle \cdot, \cdot\rangle}$, respectively.
The set of {\it natural numbers} is denoted by $\N$. The notion $x^T$ represents the {\it transpose} of a vector $x\in \mathbb{R}^n$. For a real $m \times n$ matrix $C = [c_{ij}] \in \mathbb{R}^{m\times n}$ the {\it Frobenius norm} is $\Vert C\Vert =\sqrt{\sum_{i,j} \vert c_{ij}\vert^2}$.
The open ball with center $x\in\R^n$ and radius $r>0$ is expressed as $\mathbb{B}(x;r)$.
The {\it interior} and {\it closure} of a set $S\subseteq \mathbb{R}^n$ are indicated by 
$\interior S$ and $cl\, S$, respectively.
The {\it Euclidean distance} from a point $x\in\R^n$ to a nonempty set $S\subseteq \mathbb{R}^n$ is defined as $\dist(x;S)=\inf_{z\in S} \|z-x\|$. Moreover, the {\it Euclidean projection} of the point $x$ onto $S$ is given by
$\operatorname{proj}_S(x):=\argmin_{z\in S} \|z-x\|$.
For a given function $\func{h}{\R^n}{\Rinf=\R\cup\{+\infty\}}$, the {\it effective domain} of
$h$ is defined as $\dom(h) := \{x \in \R^n : h(x) < +\infty\}$.
The function $h$ is said to be {\it proper} if $\dom(h) \neq \emptyset$.
The {\it indicator} function of a nonempty set $S\subseteq \mathbb{R}^n$, $\func{\delta_S}{\R^n}{\Rinf}$, is defined as $\delta_{S}(x)=0$ if $x\in S$, and $\delta_{S}(x)=+\infty$ otherwise.

Let $\func{h}{\R^n}{\Rinf}$ be a proper function that is locally Lipschitz at $\bar x\in \interior\dom(h)$.
The {\it generalized directional derivative} of $h$ at $\bar{x}$ in the direction $d$ is given by
$h^\circ(\bar{x};d):=\limsup_{\stackrel{x\rightarrow \bar{x}}{t\downarrow 0}}
\frac{h(x+td)-h(x)}{t}$.
Moreover, the {\it Clarke's subdifferential} of $h$ at $\bar{x}$ \cite{clarke1990optimization} is defined as
$\partial h(\bar{x}) := \left\{\zeta \in \R^n : \langle \zeta,d\rangle \leq h^\circ(\bar{x};d), \forall d\in \R^n\right\}$.
Given a nonempty set $S\subseteq \interior\dom(h)$, a point $\bar{x}\in S$ is said to be a {\it stationary} point of the problem $\displaystyle\min_{x\in S} h(x)$ if
$0\in \partial (h+\delta_S)(\bar{x})$.

To derive the convergence rate bounds of our algorithms, we use the following fact from \cite{aragon2018accelerating}.

\begin{fact}\label{lem-sequence}
    Let ${s_{k}}$ be a nonnegative real sequence and let $\varsigma,\varpi$ be some positive constants.
    Suppose that $s_{k}\to 0$ and that the sequence satisfies
    $s^{\varsigma}_{k}\leq \varpi(s_{k} - s_{k+1}),$
    for all $k$ sufficiently large. Then, the following assertions hold:
    \vspace{-4mm}
    \begin{enumerateq}
        \item
            If $\varsigma = 0$, the sequence ${s_{k}}$ converges to 0 in a finite number of steps;
        \item
            If $\varsigma \in (0, 1]$, the sequence ${s_{k}}$ converges linearly to 0 with rate $1-\nicefrac{1}{\varpi}$;
        \item
            If $\varsigma > 1$, there exists $\eta>0$ such that $s_{k} \leq \eta k^{-\tfrac{1}{\varsigma-1}}$, for all $k\in \N$ sufficiently large.
    \end{enumerateq}
\end{fact}

\vspace{-4mm}
\subsection{{\bf H\"{o}lderian error bound condition}}

An error bound is a crucial tool in optimization to achieve faster convergence rates in algorithms. In this study, we focus on a specific type of error bound known as the H\"{o}lderian error bound.

\begin{defin}\label{holerrbound}
    A proper function $\func{h}{\R^n}{\Rinf}$ is said to admit a H\"{o}lderian error bound (HEB) of order $\delta \in (0,1]$ with constant $\mu > 0$ on a nonempty set $S \subseteq \dom(h)$ if
    \begin{equation}\label{eq-holerrbound}
        \mu \dist^{\nicefrac{1}{\delta}}(x; \mathcal{X}^*) \leq h(x) - h^{*}, \quad \forall x \in S,
    \end{equation}
    where $h^* := \inf_{x \in S} h(x)$ and the set of minimizers $\mathcal{X}^* := \{x \in S \,:\, h(x) = h^{*}\}$ is nonempty.
    \end{defin}

In the literature, this property is also known by other names, such as the H\"{o}lderian growth condition, the H\"{o}lder growth bound, the H\"{o}lder error bound and the global H\"{o}lderian error bound; see, e.g., \cite{johnstone2020faster,grimmer2023general}. The HEB is instrumental in ensuring that PSMs with appropriate step-sizes achieve linear convergence, even for nonconvex problems \cite{davis2018subgradient}.
Two notable cases of the HEB arise when $\delta = 1$ and $\delta = \tfrac{1}{2}$.
For $\delta = 1$, the condition is known as the sharpness error bound or the weak sharpness condition \cite{burke1993weak,polyak1979sharp}, while for $\delta = \tfrac{1}{2}$, it is known as the quadratic growth condition \cite{karimi2016linear,zhang2020new}.

Beyond the HEB (including the sharpness error bound and quadratic growth condition), other prominent properties have been studied, such as the {\L}ojasiewicz inequality \cite{bolte2007lojasiewicz}, the Kurdyka-{\L}ojasiewicz (KL) inequality \cite{bolte2017error}, and (H\"{o}lder) metric subregularity/regularity \cite{kruger2015error}.
The {\L}ojasiewicz inequality can be considered a version of the HEB, where the inequality \cref{eq-holerrbound} holds on each compact subset of an open set $S$.
Furthermore, it has been established that the HEB and the KL inequality are equivalent for convex, closed, and proper functions \cite{bolte2007lojasiewicz,bolte2017error}.
On the other hand, H\"{o}lder metric subregularity represents a localized version of the HEB. Additionally, \cite[Proposition 3.3]{garrigos2023convergence} showed that the HEB of order $\delta$ implies H\"{o}lder metric subregularity of $\partial f$ in the convex setting.

\vspace{-2mm}
\section{Fundamental properties of paraconvex functions and optimization}\label{sec:paraconvex}
This section introduces the class of $\nu$-paraconvex functions and describes some of their key properties, which are necessary for the coming sections.

\vspace{-2mm} 

\subsection{{\bf Fundamental properties of paraconvex functions}}\label{sec:paraconvexFunc}
Let us begin with the definition of $\nu$-paraconvexity.

\begin{defin}\label{def:paraconvex}
    A proper function $\func{h}{\R^n}{\Rinf}$ is said to be $\nu$-paraconvex function on a convex set $S\subseteq \dom(h)$ for some $0<\nu\leq 1$ if there exists $\rho > 0$ such that for any $x,y \in S$ and $\lambda \in [0,1]$,
    \begin{equation}\label{eq:para1}
        h(\lambda x+(1-\lambda)y)\leq \lambda h(x) +(1-\lambda) h(y) + \rho \min\{\lambda , 1-\lambda\} \Vert x-y\Vert ^{1+\nu}.
    \end{equation}
\end{defin}

The following result shows that the notion of $\nu$-paraconvexity is unchanged if the factor $\min\{\lambda,1-\lambda\}$ is omitted from the inequality in \cref{eq:para1}. Nevertheless, this factor plays a useful role in \cref{eq:para1}, as it ensures equality at the endpoints $\lambda=0$ and $\lambda=1$.

\begin{prop}[\textbf{Equivalent definitions of $\nu$-paraconvexity}]\label{pro:paraequivalent}
    Let $\func{h}{\R^n}{\Rinf}$ be a proper function, $S\subseteq \dom(h)$ be a convex set, and $\nu\in(0,1]$. The following assertions are equivalent:
    \begin{enumerate}[label=(\alph*)]
        \item\label{pro:paraequivalent-1}
            There exists a constant $C> 0$ such that for any $x, y \in S$ and $\lambda \in [0,1]$,
            \begin{equation}\label{eq:para2}
            h(\lambda x+(1-\lambda)y) \leq \lambda h(x) + (1-\lambda) h(y) + C \lambda(1-\lambda)\Vert x-y\Vert ^{1+\nu};
            \end{equation}
        \item\label{pro:paraequivalent-2}
            The function $h$ is $\nu$-paraconvex on $S$;
        \item\label{pro:paraequivalent-3}
            There exists a constant $C> 0$ such that for any $x, y \in S$ and $\lambda \in [0,1]$,
            \begin{equation}\label{eq:para3}
            h(\lambda x+(1-\lambda)y) \leq \lambda h(x) + (1-\lambda) h(y) + C \Vert x-y\Vert ^{1+\nu}.
            \end{equation}
    \end{enumerate}
\end{prop}
\begin{proof}
    The implication \ref{pro:paraequivalent-1} $\Rightarrow$ \ref{pro:paraequivalent-2} $\Rightarrow$ \ref{pro:paraequivalent-3}
    follows directly from the definition of $\nu$-paraconvexity and the fact that $\lambda(1-\lambda) \le \min\{\lambda, 1-\lambda\}\le 1$. We next show that \ref{pro:paraequivalent-3} implies \ref{pro:paraequivalent-1}. Assume that \ref{pro:paraequivalent-3} holds.
    Fix arbitrary points $x, y \in S$ and $\lambda \in [0,1]$.
    Without loss of generality, we suppose that $0\le \lambda\leq \tfrac{1}{2}$.
    Let us define the function $\func{\varphi}{[0,1]}{\R}$ by
    \(\varphi(t):=h(z_t)-th(x)-(1-t)h(y)\) in which $z_t=tx+(1-t)y$. Then, the argument reduces to showing that there exists some $\overline{C}> 0$ such that
    \(\varphi(\lambda)\leq  \overline{C}\lambda(1-\lambda)\|x-y\|^{1+\nu}\).
    By \cref{eq:para3}, $\varphi(0)=0$ and \(\varphi(\tfrac{1}{2})\leq C\|x-y\|^{1+\nu} \le \tfrac{\overline{C}}{4}\|x-y\|^{1+\nu}\) for some $4C\le \overline{C}$. Assume that $\lambda\in (0,\tfrac{1}{2}]$. Then, there exists \(s\in \N\) such that $\lambda\in [2^{-s},2^{-s+1})$ having the expression \(\lambda=\mu 2^{-s} + (1-\mu) 2^{-s+1}\) for some $\mu\in (0,1]$.
    We claim that, for all $k\in\N$,
    \begin{align}\label{pro:eq:paraequivalent}
        \varphi(\tfrac{1}{2^k})\leq \tfrac{2^{1+\nu}}{2^{k}}C\|x-y\|^{1+\nu} \sum_{i=1}^k (\tfrac{1}{2^{\nu}})^{i}.
    \end{align}
    The above inequality is valid for $k=1$. By induction, we assume that it holds for $k\in \N$. Using \cref{eq:para3}, for each $t\in [0,1]$, we get
    \[
    \varphi(\tfrac{t}{2}) = h\left(\tfrac{z_{t}+y}{2}\right)-
    \tfrac{t}{2}h(x) -(1-\tfrac{t}{2})h(y)
    \leq \tfrac{1}{2} h(z_t) + \tfrac{1}{2} h(y) +Ct^{1+\nu}\|x-y\|^{1+\nu} - \tfrac{t}{2}h(x) -(1-\tfrac{t}{2})h(y)
    = \tfrac{1}{2} \varphi(t)  + Ct^{1+\nu}\|x-y\|^{1+\nu}.
    \]
    Applying this inequality with $t=2^{-k}$ and invoking the induction hypothesis yields
    \[
    \varphi(\tfrac{1}{2^{k+1}}) = \varphi\Big(\tfrac{\tfrac{1}{2^k}}{2}\Big)
    \le \tfrac{1}{2}\varphi(\tfrac{1}{2^{k}}) + \frac{C}{2^{k(1+\nu)}}\|x-y\|^{1+\nu}
    \le \bigg(\tfrac{2^{1+\nu}}{2^{k+1}} \sum_{i=1}^k (\tfrac{1}{2^{\nu}})^{i}
    + \frac{1}{2^{k(1+\nu)}}\bigg)C\|x-y\|^{1+\nu} =
    \tfrac{2^{1+\nu}}{2^{k+1}}C\|x-y\|^{1+\nu} \sum_{i=1}^{k+1} (\tfrac{1}{2^{\nu}})^{i},
    \]
    which completes the induction.\\
    Finally, using the inequality \cref{eq:para3} and the above bound, we obtain
    \begin{align*}
        \varphi(\lambda) &= \varphi\big(\mu \tfrac{1}{2^s} + (1-\mu)\tfrac{1}{2^{s-1}}\big) = h\big(\mu z_{2^{-s}} + (1-\mu) z_{2^{1-s}}\big) -\lambda h(x) - (1-\lambda) h(y)\\
        &\le \mu h(z_{2^{-s}}) + (1-\mu) h(z_{2^{1-s}}) + C\| z_{2^{-s}}- z_{2^{1-s}}\|^{1+\nu}  -\lambda h(x) - (1-\lambda) h(y)\\
        &= \mu \varphi(\tfrac{1}{2^s}) + (1-\mu) \varphi(\tfrac{1}{2^{s-1}}) + \tfrac{C}{2^{s(1+\nu)}}\|x-y \|^{1+\nu}
        \leq
        \bigg(\mu\tfrac{2^{1+\nu}}{2^{s}} \sum_{i=1}^s (\tfrac{1}{2^{\nu}})^{i} + (1-\mu)
        \tfrac{2^{1+\nu}}{2^{s-1}} \sum_{i=1}^{s-1} (\tfrac{1}{2^{\nu}})^{i} + \tfrac{1}{2^{s(1+\nu)}}\bigg)C\|x-y \|^{1+\nu}\\
        &\leq \bigg( \mu 2^{1+\nu} \sum_{i=1}^s (\tfrac{1}{2^{\nu}})^{i} + (1-\mu)
        2^{2+\nu} \sum_{i=1}^{s-1} (\tfrac{1}{2^{\nu}})^{i} + \tfrac{1}{2^{s\nu}}\bigg)\lambda C\|x-y \|^{1+\nu} \leq \overline{C} \lambda(1-\lambda) \|x-y \|^{1+\nu},
    \end{align*}
    where $\overline{C}:=\Big(\tfrac{2^{3+\nu}}{2^\nu -1}+2^{1-\nu}\Big)C$.
    This verifies \cref{eq:para2} and completes the proof.
\end{proof}

\begin{rem}
    To the best of our knowledge, Rolewicz \cite{rolewicz1979gamaparaconvex,rolewicz1979paraconvex} pioneered the study of paraconvexity, introducing this concept for the class of multi-valued mappings.
    This property reduces to convexity when $\nu>1$, due to \cite[Proposition 7]{rolewicz2000alpha}.
    Furthermore, several important properties and characterizations, discussed later in this section, hold when $0<\nu\le 1$.
\end{rem}

The next proposition states that continuous midpoint $\nu$-paraconvex functions are $\nu$-paraconvex, which result simplifies the verification of paraconvexity.

\begin{prop}[{\bf Midpoint paraconvexity}]\label{pro:midpoint}
    Let $\func{h}{\R^n}{\Rinf}$ be a proper continuous function on a nonempty convex set $S\subseteq \dom(h)$.
    Given some $\nu\in (0,1]$, the function $h$ is $\nu$-paraconvex on $S$ if and only if it is midpoint $\nu$-paraconvex on $S$, that is, there exists the constant $\rho> 0$ such that
    \begin{equation}\label{eq:midpoint:para1}
        h\bigg(\frac{x+y}{2}\bigg) \le \frac{1}{2}h(x) +\frac{1}{2}h(y) + \frac{\rho}{2} \|x-y\|^{1+\nu}, \quad\quad\forall x,y \in S.
    \end{equation}
\end{prop}

\begin{proof}
    The necessity follows directly from \cref{def:paraconvex}.
    To justify sufficiency, we assume a constant $\rho> 0$ exists such that  \cref{eq:midpoint:para1} holds.
    Let $x,y\in S$.
    It is clear that the set $\Lambda:=\big\{\frac{p}{2^k} : k,p\in \N, p\le 2^k\big\}$ is dense in $[0,1]$.
    Thus, it suffices to verify that the inequality \cref{eq:para1} holds for each $\lambda\in \Lambda$.
 In particular, we show by induction that for each $k\in \N$, it holds that
    \begin{equation}\label{eq:midpoint:para2}
        h\bigg(\frac{px+qy}{2^k}\bigg) \le \frac{p}{2^k}h(x) +\frac{q}{2^k}h(y) + \rho\bigg(\sum_{i=1}^k \frac{1}{2^i}\bigg) \|x-y\|^{1+\nu},
    \end{equation}
    for all $p,q\in \N$ such that $p+q=2^k$.
    For $k=1$, this inequality holds thanks to \cref{eq:midpoint:para1}.
    Suppose \cref{eq:midpoint:para2} is satisfied for some $k\in \mathbb{N}$.
    We show it for $k+1$.
    Let $p,q\in \mathbb{N}$ such that $p+q=2^{k+1}$.
    If $p=q=2^k$, then there is nothing to prove.
    Without loss of generality, assume that $p>2^k>q$, i.e.,
    \begin{align*}
         h\bigg(\frac{px+qy}{2^{k+1}}\bigg) 
         &= h\Bigg(\frac{\frac{(p-2^k)x+qy}{2^{k}}+x}{2}\Bigg)
         \le \frac{1}{2}h\bigg(\frac{(p-2^k)x+qy}{2^{k}}\bigg) +\frac{1}{2}h(x) + \frac{\rho}{2} \left\|\frac{(p-2^k)x+qy}{2^{k}}-x\right\|^{1+\nu}\\
         &\leq \frac{p}{2^{k+1}}h(x) +\frac{q}{2^{k+1}}h(y) +
         \rho\bigg(\sum_{i=1}^{k+1} \frac{1}{2^i}\bigg) \|x-y\|^{1+\nu},
    \end{align*}
    which comes from the inductive assumption, completing the inductive step.
    Now, on the basis of the facts that $\Lambda$ is dense in $[0,1]$, $\sum_{i=1}^{\infty} \frac{1}{2^i}=1$, and $h$ is a continuous function, it is concluded that
    $$h(\lambda x+(1-\lambda)y) \le \lambda h(x) +(1-\lambda)h(y) + \rho \|x-y\|^{1+\nu},$$
    for each $\lambda\in [0,1]$. Thus, the function $h$ is $\nu$-paraconvex on $S$ by \Cref{pro:paraequivalent}, completing the proof.
\end{proof}

In the following, we provide some notable examples of the class of $\nu$-paraconvex functions.

\begin{exa}[{\bf Known paraconvex functions}]\label{exa:para} We next mention several classes of paraconvex functions. 
    \begin{enumerate}[label=(\alph*)]
        \item\label{exa:para-1}
            One notable example of $\nu$-paraconvex functions lies in the class of composite optimization problems:
            \begin{equation}\label{eq:compositeporblem}
                \min_{x\in S} h(x):=\varphi(g(x)),
            \end{equation}
            where $S$ is a convex set, $\varphi:\R^n\to \R$ is a convex and Lipschitz function with constant $L_{0}>0$ on $S$, and the Jacobian of $g:\R^m\to \R^n$, $J(x)$, is H\"{o}lder continuous of order $\nu\in (0,1]$ with constant $L_{1}>0$ on $S$:
            $$\|J(x)- J(y) \|\leq L \|x-y\|^{\nu},~~~\forall x,y\in S.$$
            It can be shown that the function $h=\varphi \circ g$ is $\nu$-paraconvex with constant $\tfrac{2^{1-\nu}L_{0}L_{1}}{1+\nu}$ on $S$; see \cref{pro:para-comp}.
        \item\label{exa:para-2}
            Any smooth function $\func{h}{\R^n}{\R}$ whose gradient is H\"{o}lder continuous of order $\nu\in (0,1]$ with constant $L>0$ on a convex set $S$, is $\nu$-paraconvex function with constant $\tfrac{2^{1-\nu}L}{1+\nu}$ on $S$.
        \item\label{exa:para-3}
            In robust principal component analysis (PCA), the objective is to identify sparse corruptions in a low-rank matrix. Given an $m\times n$ matrix $A$, the goal is to decompose it as $A=UV^T$ where $U\in \R^{m\times r}$ and $V\in \R^{n\times r}$ are low-rank factor matrices with $r$ denoting the desired rank. This problem can be formulated as:
            $$\min \|UV^T - A\|_1 ~~\text{s.t.}~~ U\in \R^{m\times r}, V\in \R^{n\times r}.$$
            The Jacobian of the function $(U,V)\mapsto UV^T - A$ is Lipschitz continuous with constant $L_1=1$.
            Therefore, the objective function, which fits into the composite form \cref{eq:compositeporblem}, is $1$-paraconvex by \Cref{pro:para-comp}; see also \cite{drusvyatskiy2017proximal}. Moreover, since the function is semialgebraic, it satisfies the HEB locally; see \cite[Theorem 6.5]{bierstone1988semianalytic}. In addition, by incorporating a suitable regularization term of the form \(\|U^TU-V^TV\|_F\), the resulting objective function \(\tfrac{1}{m}\|UV^T - A\|_1+\lambda\|U^TU-V^TV\|_F\) satisfies the HEB of order $\delta=1$ and remains $1$-paraconvex function; see \cite[Propositions 5 and 6]{li2020nonconvex}.
        \item\label{exa:para-4}
            As another example, consider the Robust Phase Retrieval problem; see \cite{davis2020nonsmooth,duchi2019solvingg,eldar2014phase}. While the classical formulation targets squared magnitude conditions, we extend this to the more flexible $p$-power regime for $p \in (1, 2]$ by seeking a solution $x$ that satisfies $|\langle a_i, x \rangle|^p \approx b_i$ for $i=1, \dots, m$, where $\{a_i\}_{i=1}^{m}\subseteq\R^n$ and $\{b_i\}_{i=1}^{m}\subseteq\R$. This leads to the Robust generalized Phase Retrieval problem as follows
            \[
            \min_{x\in\mathbb{R}^n} \; \frac{1}{m}\sum_{i=1}^m \Big||\langle a_i,x\rangle|^p - b_i\Big|,
            \]
            a formulation that serves as a robust alternative to the standard least-squares approach. By choosing $1<p \le 2$, the model might become resilient to impulsive noise and heavy-tailed outliers often encountered in real-world sensor data.
            From an optimization perspective, this objective is of composite form \cref{eq:compositeporblem}, is $(p-1)$-paraconvex and semialgebraic, satisfying a H\"{o}lderian error bound.
        \item\label{exa:para-5} 
            The deployment of regularizers, ranging from $\ell_1$- and $\ell_2$- norms to nonconvex and weakly convex penalties, has attracted significant attention in optimization. In applications such as signal processing and machine learning, these penalties are essential for promoting sparsity, ensuring coercivity, and providing robustness against noise. In some cases, it has been known that nonconvex regularizers like the Minimax Concave Penalty (MCP) and SCAD have proven effective in achieving nearly unbiased sparse recovery \cite{chen2014convergence,liu2018robust,shen2016square,wen2017efficient,yang2019weakly}. In this work, we introduce a specialized penalty, the Generalized MCP (GMCP), defined as
            $$r(t) = 
            \begin{cases} 
            (1+\nu)\omega |t| - \omega^{1+\nu} |t|^{1+\nu}, \hspace*{1cm}& \text{if }\, |t| \le \frac{1}{\omega}, \\[2mm]
            \nu, & \text{if }\, |t| > \frac{1}{\omega},
            \end{cases}$$
            where $\nu \in (0,1]$ and the scaling parameter $\omega > 0$; see \Cref{fig:MPC}. From a theoretical standpoint, this function  is $\nu$-paraconvex with a constant $\rho = \omega^{1+\nu}$ by \cref{thm:S-chara}~\ref{thm:S-chara-2}. Consequently, due to \cref{pro:relWeakConvCal}~\ref{pro:relWeakConvCal-1} for any convex loss function $h: \mathbb{R}^n \to \mathbb{R}$ and any regularization parameter $\lambda > 0$, the composite objective $x \mapsto h(x) + \lambda \sum_{i=1}^{n} r(x_i)$ is globally $\nu$-paraconvex.
            \begin{figure}[!htbp]
                \centering
                \includegraphics[width=0.75\textwidth]{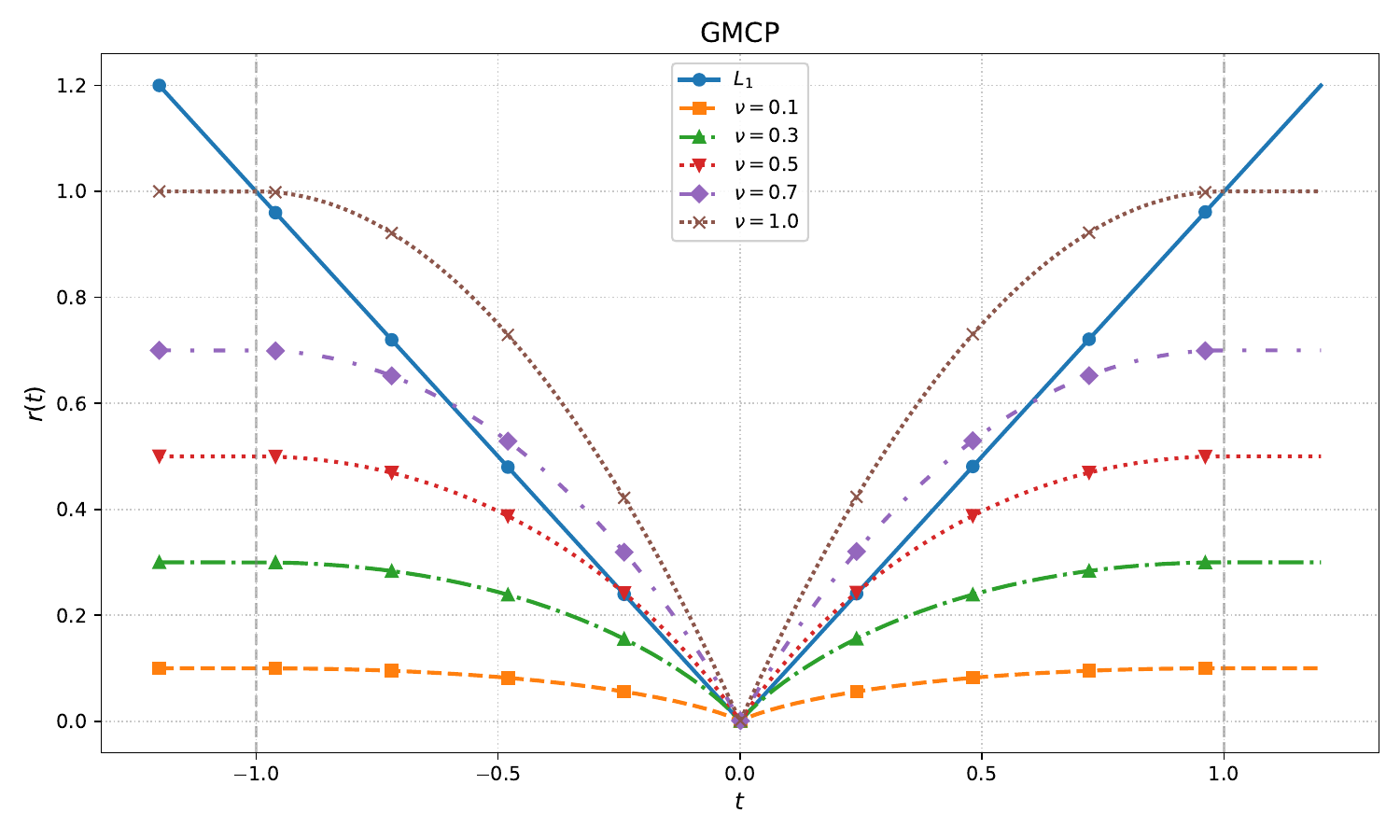}
                \caption{ The $\ell_1$-norm regularization and the GMCP  with $\nu\in\{0.1, 0.3, 0.5, 0.7, 1\}$ and $\omega=1$.}
                \label{fig:MPC}
            \end{figure}
            \qed
    \end{enumerate}
\end{exa}

A further extension of convexity in the nonconvex setting is $\nu$-weak convexity \cite{daniilidis2005filling}.
Given a nonempty open set $\mathcal{U}\subseteq \R^n$ and $0<\nu\leq 1$, a locally Lipschitz function $\func{\varphi}{\mathcal{U}}{\R}$ is said to be $\nu$-weakly convex, if for any $x_0 \in \mathcal{U}$, there exist $\delta>0$ and $C> 0$ such that for all $x,y \in \mathbb{B}(x_0;\delta)$ and $\lambda\in [0,1]$,
\begin{equation*}
    \varphi(\lambda x+(1-\lambda)y)\leq \lambda \varphi(x) +(1-\lambda) \varphi(y) + C \lambda(1-\lambda) \Vert x-y\Vert ^{1+\nu}.
\end{equation*}
This concept is a local version of $\nu$-paraconvexity.
Hence, $\nu$-paraconvexity implies $\nu$-weak convexity.
Furthermore, \cref{thm:weaktopara} demonstrates that the converse holds on compact convex sets.

\begin{lem}\label{lem:parapara}
    Let $\mathcal{I}$ be an interval in $\R$ and $a,b \in I$ such that $a<b$. Let $\func{\varphi}{\mathcal{I}}{\R}$ be a $\nu$-paraconvex function on $\mathcal{I}\cap [a,+\infty)$ and $\mathcal{I}\cap (-\infty,b]$ for some $\nu\in (0,1]$ with constant $\rho> 0$. Then, $\varphi$ is a $\nu$-paraconvex function on $\mathcal{I}$ with constant $6\rho$.
\end{lem}
\begin{proof}
    If $a=\inf \mathcal{I}$ or $b=\sup \mathcal{I}$, then there is nothing to prove.
    Therefore, without loss of generality, assume that $a,b \in \interior \mathcal{I}$.
    Let $x,y \in \mathcal{I}$ with $x<y$, and let $0< \lambda< 1$.
    Set $z=\lambda x+(1-\lambda)y$.
    We aim to show that
    \begin{equation}\label{eq-lem:para4}
        \varphi(z) \leq \lambda \varphi(x) + (1-\lambda) \varphi(y) + 6\rho \min\{\lambda,1-\lambda\} (y-x)^{1+\nu}.
    \end{equation}
    If $a<x$ or $y<b$, the inequality \cref{eq-lem:para4} is trivially valid due to the $\nu$-paraconvexity of $h$ on $\mathcal{I}\cap [a,+\infty)$ and $\mathcal{I}\cap (-\infty,b]$.
    Now, assume that $x<a<b<y$. Three possible cases arise: (i) $z\in (x,a)$; (ii) $z\in [a,b]$; (iii) $z\in (b,y)$.

    In Case~(i), $z\in (x,a)$, there exists $r,s,t \in (0,1)$ such that
    $z=ra+(1-r)x$, $a=sb+(1-s)x$, and $b=ty+(1-t)a$.
    By substituting these relations, we get
    \(a = sty + s(1-t)a + (1-s)x\)
    and \(b = ty + (1-t)sb +(1-t) (1-s) x,\)
    leading to
    \begin{equation}\label{eq:inn}
        a = \tfrac{st}{1-s(1-t)}y + \tfrac{1-s}{1-s(1-t)}x, \quad
        b = \tfrac{t}{1-s(1-t)} y + \tfrac{(1-t)(1-s)}{1-s(1-t)}x, \quad
        z = \left( (1-r) + \tfrac{r(1-s)}{1-s(1-t)} \right)x + \tfrac{rst}{1-s(1-t)}y.
    \end{equation}
    Thus, \(\lambda = (1-r) + \tfrac{r(1-s)}{1-s(1-t)}\).
    Furthermore, the $\nu$-paraconvexity of $\varphi$ on $\mathcal{I}\cap [a,+\infty)$ and $\mathcal{I}\cap (-\infty,b]$ with constant $\rho$ together with $\min\{\lambda,1-\lambda\}\le 2\lambda(1-\lambda)$ verifies that
    \begin{align}
        &\,\,\varphi(b)\leq t\varphi(y)+(1-t)\varphi(a) + 2\rho t(1-t) (y-a)^{1+\nu},\label{eq:in1}\\[2mm]
        &\varphi(a)\leq s\varphi(b)+(1-s)\varphi(x) +2\rho s(1-s) (b-x)^{1+\nu},\label{eq:in2}\\[2mm]
        &\,\varphi(z) \leq r\varphi(a)+(1-r)\varphi(x) +2\rho r(1-r) (a-x)^{1+\nu}.\label{eq:in3}
    \end{align}
    Combining \cref{eq:in1} and \cref{eq:in2} ensures
    $$
    \varphi(a)\leq s\left(t\varphi(y)+(1-t)\varphi(a) + 2\rho t(1-t) (y-a)^{1+\nu}\right)+(1-s)\varphi(x) + 2\rho s(1-s) (b-x)^{1+\nu},
    $$
    and consequently,
    \[
    \varphi(a)\leq \tfrac{st}{1-s(1-t)} \varphi(y) + \tfrac{1-s}{1-s(1-t)} \varphi(x)
        + 2\rho \left( \tfrac{st(1-t)}{1-s(1-t)} (y-a)^{1+\nu} + \tfrac{s(1-s)}{1-s(1-t)} (b-x)^{1+\nu}\right).
    \]
    Now, the inequality \cref{eq-lem:para4} is confirmed by \cref{eq:in3}, i.e.,
    \begin{align*}
        \varphi(z) &\leq \left( (1-r) + \tfrac{r(1-s)}{1-s(1-t)}\right) \varphi(x) + \tfrac{rst}{1-s(1-t)} \varphi(y) + 2\rho\left( \tfrac{rst(1-t)}{1-s(1-t)} (y-a)^{1+\nu} + \tfrac{rs(1-s)}{1-s(1-t)} (b-x)^{1+\nu} + r(1-r) (a-x)^{1+\nu}\right)\\[2mm]
        &\leq \lambda \varphi(x) + (1-\lambda) \varphi(y) + 6\rho\lambda(1-\lambda)(y-x)^{1+\nu}\leq \lambda \varphi(x) + (1-\lambda) \varphi(y) + 6\rho\min\{\lambda,1-\lambda\}(y-x)^{1+\nu}.
    \end{align*}
    To show Case~(ii), we assume that $z\neq a$. There exist $r,s\in (0,1)$ such that $z=ra+(1-r)y$ and $a=sx+(1-s)z$
    implying
    $z = \tfrac{rs}{1-r(1-s)} x + \tfrac{1-r}{1-r(1-s)} y.$
    So, $\lambda = \tfrac{rs}{1-r(1-s)}$.
    Moreover, relying again on $\nu$-paraconvexity of $\varphi$ we deduce
    \begin{align*}
        \varphi(z) &\leq r\varphi(a)+(1-r)\varphi(y) + 2\rho r(1-r) (y-a)^{1+\nu}\\
        &\leq r\left(s\varphi(x)+(1-s)\varphi(z) + 2\rho s(1-s) (z-x)^{1+\nu} \right)+(1-r)\varphi(y) + 2\rho r(1-r) (y-a)^{1+\nu},
    \end{align*}
    which comes out as \cref{eq-lem:para4}.
    In the case where $z=a$, we have $z\neq b$, and following a similar argument yields the desired result. 
    The proof of Case~(iii) is similar to Cases~(ii)-(iii) and is therefore omitted. 
\end{proof}

Our next result indicates that $\nu$-weak convexity leads to $\nu$-paraconvexity on compact convex sets.

\begin{thm}[{\bf $\nu$-weak convexity and $\nu$-paraconvexity}]\label{thm:weaktopara}
    Let $\mathcal{U}\subseteq \R^n$ be an open set and $S\subseteq \mathcal{U}$ be a nonempty convex and compact set.
    Let $\func{h}{\mathcal{U}}{\R}$ be a $\nu$-weakly convex function for some $\nu\in (0,1]$. Then, $h$ is a $\nu$-paraconvex function on $S$. 
\end{thm}
\begin{proof}
    Let us begin by showing that the inequality of $\nu$-weak convexity holds uniformly on $S$, with a fixed constant $C> 0$ independent of $x\in S$.
    The $\nu$-weak convexity of $h$ implies that for each point $x\in S$, there exist constants
    $C_{x}> 0$ and $\delta_{x}>0$ such that for all
    $x_{0},y_{0} \in \mathbb{B}(x;\delta_{x})$ and $\lambda\in [0,1]$,
    \begin{equation}\label{eq:weak1}
            h(\lambda x_{0}+(1-\lambda)y_{0})\leq \lambda h(x_{0}) +(1-\lambda) h(y_{0}) + C_{x} \lambda(1-\lambda) \Vert x_{0}-y_{0}\Vert ^{1+\nu}.
    \end{equation}
    The collection of open balls, $\{\mathbb{B}(x;\delta_{x})\}_{x\in S}$ covers $S$, and by the compactness principle, a finite subcollection of them covers $S$.
    Thus, there exist points $x_i\in S$ and radii $\delta_i>0$, $i=1,2,\ldots,m,$ such that 
    $S\subseteq \bigcup_{i=1}^m \mathbb{B}(x_i;\delta_{i})$.
    Define $\overline{C}:=\max\{C_i : ~~i=1,2,\ldots,m\}$.
    Then, for each $x\in S$, there exists some $\delta_x>0$ such that $\mathbb{B}(x;\delta_x) \subseteq \mathbb{B}(x_i;\delta_{i})$, namely for an index $i$ such that $x\in \mathbb{B}(x_i;\delta_{i})$,
    and \cref{eq:weak1} holds on $\mathbb{B}(x;\delta_x)$ with the fixed constant $\overline{C}$.

    Let $x,y \in S$ and $\mu\in [0,1]$ be arbitrary and constant hereafter.
    Let us consider the function $\varphi:[0,1]\to \R$ defined as
    $\varphi(t):= h(z_t)$,
    where $z_t=y+t(x-y)$ and $t\in[0,1]$.
    We aim at showing that $\varphi$ is $\nu$-paraconvex on $[0,1]$ with constant $9\overline{C}\Vert y-x\Vert^{1+\nu}$.
    Once this is established, we have
    $$
         h(z_{\mu}) = \varphi(\mu) \leq \mu \varphi(1) +(1-\mu) \varphi(0) +9\overline{C} \Vert y-x\Vert^{1+\nu}\mu (1-\mu) \vert 1 - 0 \vert^{1+\nu}=\mu h(x) +(1-\mu) h(y) +9\overline{C} \mu (1-\mu) \Vert y - x \Vert^{1+\nu},
    $$
    as our desired result.
    Let $\hat{t}\in (0,1)$.
    Inasmuch as $S$ is convex, it follows that $z_{\hat{t}}\in S$.
    So, there exists $\delta>0$ such that
    for all $x_{0},y_{0} \in \mathbb{B}(z_{\hat{t}};\delta)$ and $\lambda\in [0,1]$, the inequality \cref{eq:weak1} is satisfied with constant $\overline{C}$.
    Let us consider $\epsilon :={\tfrac{1}{2}} \min\{1-\hat{t},\hat{t}, \tfrac{\delta}{\Vert y -x\Vert} \}$.
    Suppose $t_1,t_2 \in [\hat{t}-\epsilon, \hat{t}+\epsilon]$ and $\lambda\in [0,1]$.
    Then, $z_{t_{1}},z_{t_{2}} \in \mathbb{B}(z_{\hat{t}};\delta)$ and therefore
    $$
    h(\lambda z_{t_{1}} +(1-\lambda) z_{t_{2}})\leq \lambda h(z_{t_{1}}) +(1-\lambda) h(z_{t_{2}}) +\overline{C} \lambda (1-\lambda) \Vert z_{t_{1}} - z_{t_{2}} \Vert^{1+\nu},
    $$
    implying
    \[
    \varphi(\lambda t_{1} +(1-\lambda) t_{2})\leq \lambda \varphi(t_{1}) +(1-\lambda) \varphi(t_{2}) +\overline{C} \Vert y-x\Vert^{1+\nu}\lambda (1-\lambda) \vert t_{1} - t_{2} \vert^{1+\nu}.
    \]
    Hence, $\varphi$ is $\nu$-paraconvex on $[\hat{t}-\epsilon , \hat{t}+\epsilon]$ with constant $\overline{C}\Vert y-x\Vert^{1+\nu}$. Now, let us define
    $$A := \left\{t\in [0,1] :~ \varphi ~\text{is}~ \nu-\text{paraconvex on} ~[\hat{t}-\epsilon,t]~ \text{with constant}~ 3\overline{C}\Vert y-x\Vert^{1+\nu}\right\},$$
    and consider $a:= \sup A$.
    Clearly, $\hat{t}+\epsilon \in A$. So, the supremum is well-defined and $0<a\leq 1$.
    It is easy to see that $a\in A$ using continuity of $h$ (see \Cref{thm:paralip}). We claim that $a=1$.
    Otherwise, if $a\in (0,1)$,
    the argument applied for $\hat{t}$ can similarly be repeated with $a$
    to conclude that $\varphi$ is $\nu$-paraconvex on $[a-\bar{\epsilon} , a+\bar{\epsilon}]\subseteq (0,1)$ for some $\bar{\epsilon}>0$ with constant $\overline{C}\Vert y-x\Vert^{1+\nu}$.
    In this case, the function $\varphi$ is $\nu$-paraconvex on $[\hat{t}-\epsilon,a]$ and
    $[a-\bar{\epsilon} , a+\bar{\epsilon}]$ with constant $\overline{C}\Vert y-x\Vert^{1+\nu}$, and so on $[\hat{t}-\epsilon, a+\bar{\epsilon}]$
    with constant $3\overline{C}\Vert y-x\Vert^{1+\nu}$, by \cref{lem:parapara}.
    This contradicts $a=\sup A$.
    Thus, $a=1$. Hence, $\varphi$ is $\nu$-paraconvex on $[\hat{t}-\epsilon,1]$ with constant $3\overline{C}\Vert y-x\Vert^{1+\nu}$.
    Similarly, it can be shown that $\varphi$ is $\nu$-paraconvex on $[0, \hat{t}+\epsilon]$ with constant $3\overline{C}\Vert y-x\Vert^{1+\nu}$.
    Therefore, $\varphi$ is $\nu$-paraconvex on $[0,1]$ with constant $9\overline{C}\Vert y-x\Vert^{1+\nu}$, due to \cref{lem:parapara}, and the proof is completed.
\end{proof}

In what follows, we investigate several properties and characterizations of the class of $\nu$-paraconvex functions.
First, similar to the convex case, we observe that $\nu$-paraconvex functions exhibit local Lipschitz continuity.
Consequently, their Clarke's subdifferential is well defined and nonempty at interior points of their domain; see \cite[Proposition 2.1.2]{clarke1990optimization}.
Although this property was proven in \cite{jourani1996subdifferentiability,jourani1996open} under the assumptions of local boundedness or continuity, we establish here that 
$\nu$-paraconvex functions are locally bounded, thus directly satisfying the local Lipschitz continuity property.
\begin{thm}[{\bf Locally bounded and Lipschitz properties}]\label{thm:paralip}
    Let $\func{h}{\R^n}{\Rinf}$ be a proper $\nu$-paraconevex function on a convex set $S\subseteq \dom(h)$ with $\nu\in (0,1]$ and constant $\rho> 0$.
    Then the following assertions hold:
    \begin{enumerate}[label=(\alph*)]
        \item\label{thm:paralip-1} The function $h$ is locally bounded on $\interior S$;
        \item\label{thm:paralip-2} The function $h$ is locally Lipschitz on $\interior S$.
    \end{enumerate}
\end{thm}
\begin{proof}
    \ref{thm:paralip-1} Let $\bar{x} \in \interior S$.
    There exists $\epsilon>0$ such that
    $P:=\{x\in \R^n : \|x-\bar{x}\|_{\infty} \leq \epsilon\}\subseteq S.$
    Clearly, polyhedral $P$ has $2^n$ vertexes.
    Assume that $E$ is the set of all vertices of the polyhedral $P$ and 
    $M:=\max_{x\in E}h(x).$
    Let $x\in P$.
    Then, thanks to the Minkowski-Weyl representation \cite[Proposition 3.2.2]{bertsekas2003convex} and extension of Caratheodory's theorem \cite[Theorem 1.22]{bertsekas2003convex}, there exist scalars $\lambda_j\in [0,1]$ and vertices $x^j\in E$, $j=1,2\ldots,n+1$, such that
    \(x = \sum_{j=1}^{n+1} \lambda_j x^j\) and \(\sum_{j=1}^{n+1} \lambda_j =1.\)
    Using approximate Jensen’s inequality \cite[Lemma 2.1]{van2022variational}, we deduce
    \begin{align*}
        h(x) \leq \sum_{j=1}^{n+1} \lambda_j h(x^j) + \rho \sum_{j=1}^{n+1} \lambda_j(1-\lambda_j) \max_{1\le k\le n+1} \|x^k-x^j\|^{1+\nu}
        \leq M + \frac{n+1}{4}\sqrt{n}^{1+\nu}2^{1+\nu}\epsilon^{1+\nu}\rho.
    \end{align*}
    Now, assume that $x\in cl \mathbb{B}(\bar{x};\epsilon)\subseteq P$.
    Set, $\bar{\lambda}:= \frac{\epsilon}{\epsilon+\|x-\bar{x}\|}$.
    Then,
    \begin{align*}
        h(\bar{x}) \leq \bar{\lambda} h(x) + (1-\bar{\lambda}) h\left(\tfrac{1}{1-\bar{\lambda}}\bar{x}-\tfrac{\bar{\lambda}}{1-\bar{\lambda}}x\right) + \rho \min\{\bar{\lambda}, 1-\bar{\lambda}\} \left\|x-\tfrac{1}{1-\bar{\lambda}}\bar{x}+\tfrac{\bar{\lambda}}{1-\bar{\lambda}}x\right\|^{1+\nu},
    \end{align*}
    i.e.,
    \begin{align*}
        h(x) &\geq \tfrac{1}{\bar{\lambda}} h(\bar{x}) - \tfrac{1-\bar{\lambda}}{\bar{\lambda}} h\left(\tfrac{1}{1-\bar{\lambda}}\bar{x}-\tfrac{\bar{\lambda}}{1-\bar{\lambda}}x\right) - \tfrac{\rho \min\{\bar{\lambda}, (1-\bar{\lambda})\}}{\bar{\lambda}}\left(\epsilon+\|x-\bar{x}\|\right)^{1+\nu}\geq h(\bar{x}) -  M - \left(\tfrac{n+1}{4}\sqrt{n}^{1+\nu}-1\right) 2^{1+\nu}\epsilon^{1+\nu}\rho.
    \end{align*}
    Setting
    $$\overline{M}:=\max\left\{M + \tfrac{n+1}{4}\sqrt{n}^{1+\nu}2^{1+\nu}\epsilon^{1+\nu}\rho, \left|h(\bar{x}) -  M - \left(\tfrac{n+1}{4}\sqrt{n}^{1+\nu}-1\right) 2^{1+\nu}\epsilon^{1+\nu}\rho\right|\right\},$$
    we come to $|h(x)|\leq \overline{M}$ for all $x\in cl\,\mathbb{B}(\bar{x};\epsilon)$, i.e., $h$ is locally bounded on $\interior S$.
    Assertion~\ref{thm:paralip-2} follows from Assertion~\ref{thm:paralip-1} in combination with \cite[Lemma 2.5]{jourani1996open} and \cite[Proposition 2.2]{jourani1996subdifferentiability}, adjusting our desired result.
\end{proof}

In what follows, we investigate some calculus around the class of $\nu$-paraconvex functions. Since the proofs of these results are straightforward consequences of \Cref{def:paraconvex}, we omit the details of the proofs.

\begin{prop}[{\bf Paraconvexity calculus}]
\label{pro:relWeakConvCal}
    Let $\func{h}{\R^n}{\Rinf}$ be a proper and $\nu$-paraconvex function on a convex set $S\subseteq \dom(h)$ for some $\nu\in(0,1]$ with constant $\rho>0$, let $\func{h_i}{\R^n}{\Rinf}$, $i\in \mathcal{I}$ where $\mathcal{I}$ is an index set, be proper and $\nu_i$-paraconvex functions on a convex set $S_{\mathcal{I}}\subseteq \dom(h_i)$ for some $\nu_i\in(0,1]$ with constant $\rho_i>0$, and let
    $\func{g}{\R^n}{\Rinf}$ be a convex function on $S\subseteq \dom(g)$.
    Then, the following statements hold:
    \begin{enumerate}[label=(\alph*)]
        \item\label{pro:relWeakConvCal-1}
            The function $g+\varrho h$ is $\nu$-paraconvex on $S$ with constant $\varrho\rho$ for any $\varrho>0$;
        \item\label{pro:relWeakConvCal-2}
            The function $h$ is $\nu$-paraconvex on $S$ with any constant $\varrho \in (\rho, \infty)$;
        \item\label{pro:relWeakConvCal-3}
            The function $\varrho h$ is $\nu$-paraconvex on $S$ with constant $\varrho\rho$ for any $\varrho>0$;
        \item\label{pro:relWeakConvCal-4}
            If $S$ is bounded with diameter $K>0$, then $h$ is $\vartheta$-paraconvex on $S$ with constant $\rho K^{\nu-\vartheta}$ for any $0<\vartheta<\nu$;
        \item \label{pro:relWeakConvCal-5}
            If $\sum_{i} \rho_i<\infty$ and $\vartheta=\nu_{i}$, $i\in \mathcal{I},$ then $\sum_{i} h_i$ is $\vartheta$-paraconvex on $S_{\mathcal{I}}$ with constant
            $\sum_{i} \rho_i$;
        \item\label{pro:relWeakConvCal-6}
            If $S_{\mathcal{I}}$ is bounded with diameter $K>0$, $0<\vartheta=\inf_{i} \nu_{i}$, and $\sum_{i} \rho_iK^{\nu_{i}-\vartheta}<\infty$, then $\sum_{i} h_i$ is $\vartheta$-paraconvex on $S_{\mathcal{I}}$ with constant
            $\sum_{i} \rho_i K^{\nu_{i}-\vartheta};$
        \item\label{pro:relWeakConvCal-7}
            If $\sup_{i} \rho_i<\infty$,  and $\vartheta=\nu_{i}$, $i\in \mathcal{I},$ then $\sup_{i} \set{h_i(x)}$ is $\vartheta$-paraconvex on $S_{\mathcal{I}}$ with constant $\rho:=\sup_{i} \rho_i$.
    \end{enumerate}
\end{prop}

We next present characterizations of the $\nu$-paraconvexity that are the straightforward consequences of the definition.
Some of these results on the whole space have been previously investigated in the literature; see \cite{jourani1996subdifferentiability}. We include them here for completeness.

Recall that, by \cite[Proposition~2.1.2]{clarke1990optimization}, if a function is locally Lipschitz continuous at an interior point of its domain, then its Clarke's subdifferential is well defined and nonempty at that point.

\begin{thm}[{\bf First order characterization}]\label{thm:f-chara}
    Let $\func{h}{\R^n}{\Rinf}$ be a proper and locally Lipschitz function on a nonempty convex set $S\subseteq \interior\dom(h)$. Let $\nu\in (0,1]$ and $\rho\geq 0$. Then, the following statements hold:
    \begin{enumerate}[label = (\alph*)]
        \item \label{thm:f-chara_p1}
            If $S$ is open and $h$ is a $\nu$-paraconvex function on $S$ with constant $\rho$,
            then for any $x, y \in S$ and any $\zeta\in\partial h(x)$, 
                \begin{equation}\label{eq:charcterization1}
                    h(y)\geq h(x)+\innprod{\zeta}{y-x}-\rho \Vert y-x\Vert^{1+\nu}.
                \end{equation}
            Conversely, if inequality \cref{eq:charcterization1} holds for any $x, y \in S$ and any $\zeta\in\partial h(x)$, then 
            $h$ is $\nu$-paraconvex function on $S$ with constant $2^{1-\nu}\rho$;
        \item \label{thm:f-chara_p2}
            If $S$ is open and $h$ is a $\nu$-paraconvex function on $S$ with constant $\rho$, then $\partial h$ is $\nu$-paramonotone on $S$ with constant $2\rho$ , i.e., for any $x,y\in S$, any $\zeta\in\partial h(x)$, and any $\eta\in\partial h(y)$,
            \begin{equation}\label{eq:charctrization5}
                \innprod{\eta-\zeta}{y-x}\geq -2\rho \Vert x-y\Vert^{1+\nu}.
            \end{equation}
            Conversely, if $\partial h$ is $\nu$-paramonotone on $S$ with constant $\rho$, then $h$ is $\nu$-paraconvex on $S$ with constant $\tfrac{2^{1-\nu}\rho}{1+\nu}$.
    \end{enumerate}
\end{thm}
\begin{proof}
    \ref{thm:f-chara_p1}  For the first implication,
    let $x, y \in S$ and $\zeta\in\partial h(x)$. Define $d:=y-x$. Then, for any $z\in S$ sufficiently close to $x$ and any $0<t<1$, we obtain
    \[
    h(z+d)- h(z) \ge \tfrac{h(z+td) -h(x)}{t} -\rho \tfrac{\min\{t,1-t\}}{t} \|d\|^{1+\nu}.
    \]
    Taking the limit superior on both sides 
    as $t\to 0$ and $z\to x$, and using continuity of $h$ yields inequality \cref{eq:charcterization1}.
    To verify the converse implication, let $x,y \in S$, $0<\lambda<1$, and $\zeta\in\partial h(\lambda x +(1-\lambda) y)$.
    By \cref{eq:charcterization1}, we obtain
    \begin{align*} 
        h(x)&\geq h(\lambda x + (1-\lambda)y)-(1-\lambda)\innprod{\zeta}{y-x}-\rho (1-\lambda)^{1+\nu} \Vert y-x\Vert^{1+\nu},\\[2mm]
        &~h(y)\geq h(\lambda x + (1-\lambda)y) + \lambda\innprod{\zeta}{y-x}-\rho \lambda ^{1+\nu} \Vert y-x\Vert^{1+\nu}.
    \end{align*}
    Combining these inequalities gives
    $$\begin{array}{ll}
         \lambda h(x) +(1-\lambda) h(y) &\geq h(\lambda x + (1-\lambda)y) -\rho (\lambda (1-\lambda)^{1+\nu}+(1-\lambda)\lambda^{1+\nu}) \,\Vert y-x\Vert^{1+\nu}\\[2mm]
         & \geq h(\lambda x + (1-\lambda)y) -2^{1-\nu}\rho \min\{\lambda,(1-\lambda)\} \,\Vert y-x\Vert^{1+\nu},
    \end{array}
    $$
    where the second inequality follows from the concavity of function $t\mapsto (1-t)^{\nu}+t^{\nu}$ on $[0,1]$ which attains its maximum at $t=\nicefrac{1}{2}$ implying $(1-\lambda)^{\nu}+\lambda^{\nu}\le 2^{1-\nu}$.\\
    \ref{thm:f-chara_p2} The necessity implication follows immediately from \cref{eq:charcterization1}. For the converse implication, let $x,y\in S$ and $\zeta\in\partial h(x)$. Then, there exists a measurable selection $\eta_t\in \partial h(ty+(1-t)x)$ for almost every $t\in [0,1]$ such that
    \[
    h(y)- h(x)-\innprod{\zeta}{y-x}
    = \int_{0}^1 \innprod{\eta_t-\zeta}{y-x}dt
    \ge \int_{0}^1 -\rho t^\nu \Vert y-x\Vert^{1+\nu}dt = -\tfrac{\rho}{1+\nu} \Vert y-x\Vert^{1+\nu},
    \]
    completing the proof in view of Assertion~\ref{thm:f-chara_p1}.
\end{proof}

Applying the above characterization, we show the $\nu$-paraconvexity of a composite function.

\begin{prop}[{\bf Composite paraconvex function}]\label{pro:para-comp}
    Let $S\subseteq \R^n$ be a convex set and let $\varphi:\R^n\to \R$ be a convex and Lipschitz function with constant $L_{0}>0$ on $S$.
    Assume that the
    Jacobian of $g:\R^m\to \R^n$ is H\"{o}lder continuous of order $\nu\in (0,1]$ with constant $L_{1}>0$ on $S$. Then, the composite function $h:=\varphi \circ g$ is $\nu$-paraconvex with constant $\rho=\tfrac{2^{1-\nu}L_{0}L_{1}}{1+\nu}$ on $S$.
\end{prop}
\begin{proof}
    The H\"{o}lder continuity of Jacobian of $g$ ensures
    $$\Vert g(y) - g(x) - J(x)(y-x)\Vert \leq \tfrac{L_{1}}{1+\nu} \Vert y-x\Vert^{1+\nu},~~~~\forall x,y\in S,$$
    as stated in \cite[Lemma 1]{yashtini2016global}.
    Let $x,y \in S$ and $\zeta \in \partial h(x)$.
    By the chain rule theorem, \cite[Theorem 2.3.10]{clarke1990optimization}, there exists $\eta \in \partial \varphi(g(x))$ such that $\zeta = J(x)^T \eta$. Now, using these properties and the convexity of $\varphi$, we get
    $$
    \begin{array}{ll}
         h(y)-h(x) -\langle \zeta , y-x\rangle &= \varphi(g(y)) -\varphi(g(x)) - \langle J(x)^T \eta , y-x\rangle \geq \langle \eta ,g(y)-g(x)\rangle - \langle \eta , J(x)(y-x)\rangle\\[2mm]
         &\geq -\Vert \eta \Vert \,\Vert g(y) - g(x) - J(x)(y-x)\Vert\geq -\tfrac{L_{0}L_{1}}{1+\nu} \Vert y-x\Vert^{1+\nu},
    \end{array}
    $$
    i.e., our desired result follows from \cref{thm:f-chara}~\ref{thm:f-chara_p1}.
\end{proof}

We conclude this subsection by indicating that the convexity of the function $x\mapsto h(x)+C \Vert x\Vert ^{1+\nu}$ is sufficient for the $\nu$-paraconvexity of $h$. Building on this observation, we also derive a sufficient second-order characterization.

\begin{thm}[\textbf{Second order characterization}]\label{thm:S-chara}
    Let $\func{h}{\R^n}{\Rinf}$ be a proper and continuous function and $S\subseteq\dom(h)$ be a convex set. Let $\nu\in (0,1]$ and $C, \overline{C}> 0$. Then, the following assertions hold:
    \begin{enumerate}[label=(\alph*)]
        \item \label{thm:S-chara-1}
            Let $\func{g}{\R^n}{\Rinf}$ be a continuously differentiable function whose gradient is H\"{o}lder continuous of order $\nu$ with constant
            $\overline C$ on $S\subseteq\interior\dom(g)$. If the function $h+C g$ is convex on $S$, then $h$ is $\nu$-paraconvex on $S$ with constant $\tfrac{2^{1-\nu}\overline CC}{1+\nu}$;
        \item \label{thm:S-chara-2}
            If the function $x\mapsto h(x)+C \Vert x\Vert ^{1+\nu}$ is convex on $S$, then $h$ is $\nu$-paraconvex on $S$ with constant $4^{1-\nu}C$;
        \item \label{thm:S-chara-3}
            If $h$ is a twice continuously differentiable function on $S\subseteq\interior\dom(h)$ and
            the matrix
            $$\nabla^{2}h(x)+C\tfrac{1+\nu}{\Vert x\Vert^{3-\nu}} \left(\Vert x\Vert^2 I -(1-\nu) x x^T\right),$$
            is positive semidefinite for each $x\in S$ with $x\neq 0$,
            then $h$ is a $\nu$-paraconvex function on $S$.
    \end{enumerate}
\end{thm}
\begin{proof}
    \ref{thm:S-chara-1}
    By invoking \cref{pro:para-comp}, the function $g$ is $\nu$-paraconvex on
    $S$ with constant $\tfrac{2^{1-\nu}\overline{C}}{1+\nu}$.
    Thus, the function
    $h$ is $\nu$-paraconvex on $S$ with constant $\tfrac{2^{1-\nu}\overline{C}C}{1+\nu}$ by \cref{pro:relWeakConvCal}~\ref{pro:relWeakConvCal-1}.\\
    \ref{thm:S-chara-2}
    According to \cite[Theorem~6.3]{rodomanov2020smoothness}, the gradient of the function
    \(x \mapsto - \|x\|^{1+\nu}\)
    is a H\"{o}lder continuous of order $\nu$ with constant
    $2^{1-\nu}(1+\nu)$.
    Therefore, the claim holds by Assertion~\ref{thm:S-chara-1}.\\
    \ref{thm:S-chara-3}
    This assertion follows directly from Assertion~\ref{thm:S-chara-1}.
\end{proof}

\begin{cor}
    Let $\func{h}{\R^n}{\Rinf}$ be a convex function on a convex set $S\subseteq \dom(h)$. Then, for each $\nu\in (0,1]$ and $C> 0$, the function $x\mapsto h(x)-C \Vert x\Vert ^{1+\nu}$ is $\nu$-paraconvex on $S$.
\end{cor}

\begin{rem}
    In the case that $\nu=1$, a function $\func{h}{\R^n}{\Rinf}$ is $1$-paraconvex if and only if there exists a nonnegative constant $C$ such that $h(x)+C \Vert x\Vert ^{2}$ is convex; see \cite{rolewicz1979gamaparaconvex,rolewicz1979paraconvex,nurminskii1973quasigradient}.
    On the other hand, for the case that $0<\nu<1$, Rolewicz \cite[Example 1]{rolewicz1979gamaparaconvex} presented an example of a $\nu$-paraconvex function for which a similar result does not hold.
    However, we establish that the convexity of $h(x)+C \Vert x\Vert ^{1+\nu}$ for some $0<\nu<1$ and $C> 0$ implies the $\nu$-paraconvexity of $h$ (see \cref{thm:S-chara}).
    Moreover, in \Cref{exa:example_para}, we provide an example that highlights this result, emphasizing the existence of $\nu$-paraconvex functions with $0<\nu<1$ that are not $1$-paraconvex (weakly convex). Additionally, the function $h(x)= -|x|$ is not $\nu$-paraconvex for any $\nu\in (0,1]$.
\end{rem}

The following example provides a $\nu$-paraconvex function which is not $1$-paraconvex (weakly convex). 

\begin{exa}\label{exa:example_para}
    Let $\nu\in (0,1)$. Let us consider the function $\func{h}{\R}{\R}$ given by
    \[
    h(x) = \left\{\begin{array}{ll}
         1 - |x|^{1+\nu}, \quad\quad\quad & -1\leq x \leq 1,  \\
         x^2 -1, \quad &o.w, 
    \end{array}\right.
    \]
    see \Cref{fig:paraconvex}. The function \(x \mapsto h(x) + |x|^{1+\nu}\) is convex implying that \(h\) is a \(\nu\)-paraconvex function. 
    Furthermore, it can be seen that for any \(\rho> 0\), the function \(x \mapsto h(x) + \rho |x|^2\) is not convex and possesses a local maximum at $x=0$ since $h''(x)+2\rho\to -\infty$ as $x\to 0$. 
    This indicates that the original function \(h\) is not \(1\)-~paraconvex (i.e., weakly convex). 
    Nonetheless, the function \(x \mapsto h(x) + \rho |x|^2\) also qualifies as a \(\nu\)-paraconvex function.
    \begin{figure}[htp!]
        \centering
        \subfloat{\includegraphics[width=8.1cm]{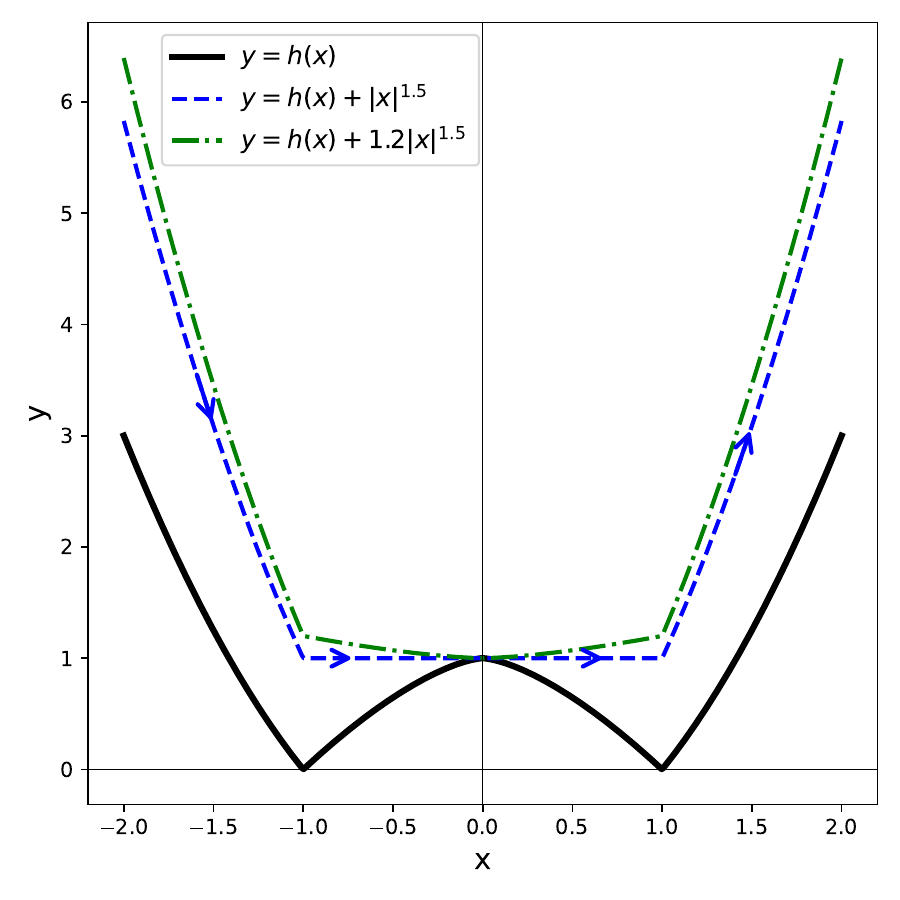}}%
        \qquad
        \subfloat{\includegraphics[width=8.1cm]{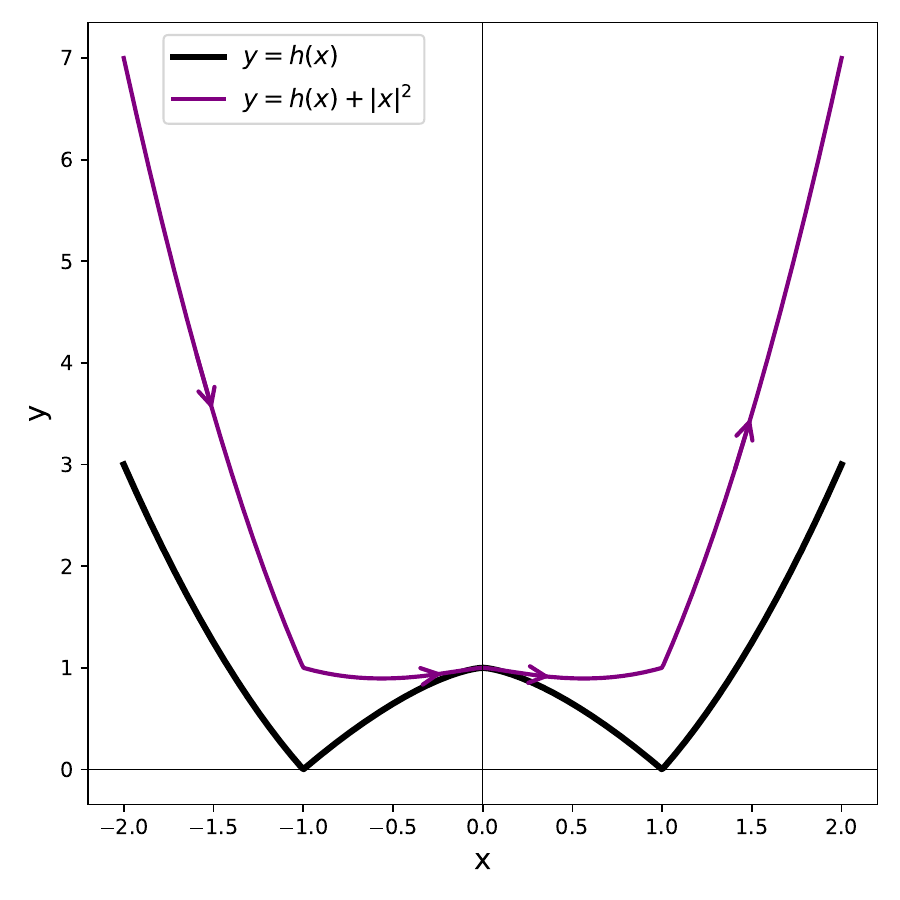}}%
        \qquad
        \caption{Comparisons of the function $h(x)$ with $\nu=0.5$ and its modified versions demonstrating $\nu$-paraconvexity. The left plot shows $h(x)$ (in black), $h(x)+|x|^{1.5}$ (in green), highlighting the convexity of $h(x)+|x|^{1.5}$ which implies $h$ is $0.5$-paraconvex. The right plot shows $h(x)$ alongside $h(x)+|x|^{2}$ (in purple), illustrating that $h$ is not $1$-paraconvex.}
        \label{fig:paraconvex}
    \end{figure}\qed
\end{exa}

\vspace{-4mm}
\subsection{{\bf Paraconvex optimization: Finding global minima}}
Here, we consider the optimization problem of the form \cref{prb2}. Throughout the paper, we consider the following assumption.

\begin{ass}\label{ass}
    We assume that
    \begin{enumerateq}
        \item[1.] The set $X\subseteq \dom(f)$ is nonempty closed and convex;
        \item[2.] The set of minimizers $\mathcal{X}^*:=\argmin_{x\in X} f(x)$ is nonempty and $f^*$ is the optimal value of \cref{prb2};
    \end{enumerateq} 
     and the following properties are satisfied for parameters $\rho> 0$, $0< \nu\leq 1$, $\tfrac{1}{1+\nu}<\delta\leq 1$ and $\mu >0$:
    \begin{enumerateq}
        \item[3.] The function $f$ is $\nu$-paraconvex with constant $\rho$ on a open convex set $X'\subseteq \dom(f)$ such that $X\subseteq X'$;
        \item[4.] The function $f$ admits the HEB of order $\delta$ with constant $\mu$.
    \end{enumerateq}
\end{ass}

As discussed in \Cref{sec:paraconvexFunc}, the class of paraconvex functions includes nonconvex and nonsmooth objectives that arise in structured optimization models. In the following example, we illustrate that such functions may admit saddle points, while their global minimizers can still be located within a relatively large basin of attraction.

\begin{exa}\label{exa:example_para_saddle}
    Let us consider the function $\func{h}{\R^2}{\R}$ given by \(h(x,y)= x^2 + \Big(|y|^{1+\nu} - 1\Big)^2\) with $\nu\in (0,1]$.
    A direct calculation yields
    \[
    \nabla h(x,y) = \begin{bmatrix}
        2x\\
        2(1+\nu){\rm sgn}(y)(|y|^{1+2\nu} -|y|^{\nu})
    \end{bmatrix},\quad\quad
    \nabla^{2} h(x,y) = \begin{bmatrix}
        2  &  0\\
        0  & 2(1+\nu)\Big((1+2\nu)|y|^{2\nu} -\nu|y|^{\nu-1}\Big)
    \end{bmatrix}.
    \]
    It is evident that $h$ is nonconvex; see \Cref{fig:paraconvExa}. Moreover, the points $(0,\pm1)$ and $(0,0)$ are stationary.
    The points $(0,\pm1)$ are global minimizers, whereas $(0,0)$ is a saddle point.
    In particular, when $\nu=1$, the Hessian at $(0,0)$ has eigenvalues $\lambda_1=2$ and $\lambda_2=-4$,
    showing that $(0,0)$ is a strict nondegenerate saddle point.
    In contrast, for $\nu\in(0,1)$, the second eigenvalue tends to $-\infty$ as $y\to0$. Furthermore, the function $h$ is $\nu$-paraconvex with any constant $\rho>2$.
    It is also worth noting that, for each global minimizer, the open ball of radius $r<1$ centered at the minimizer contains no saddle points.\qed
    \begin{figure}[htp]
        \centering
        \subfloat{\includegraphics[width=6.4cm]{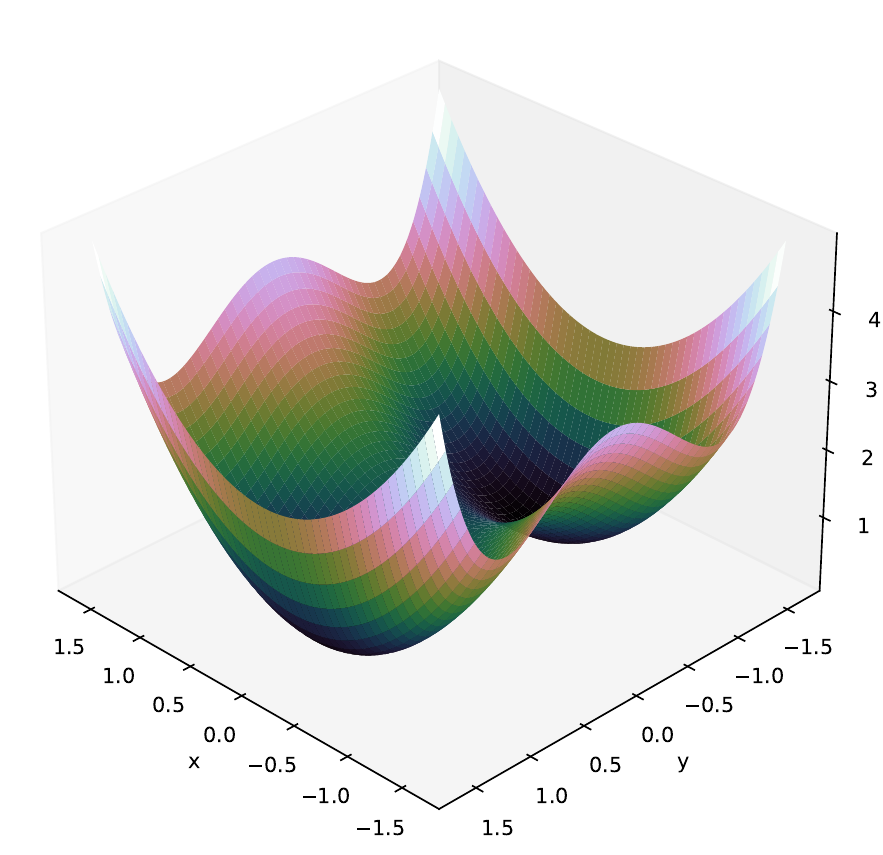}}%
        \qquad
        \subfloat{\includegraphics[width=8cm]{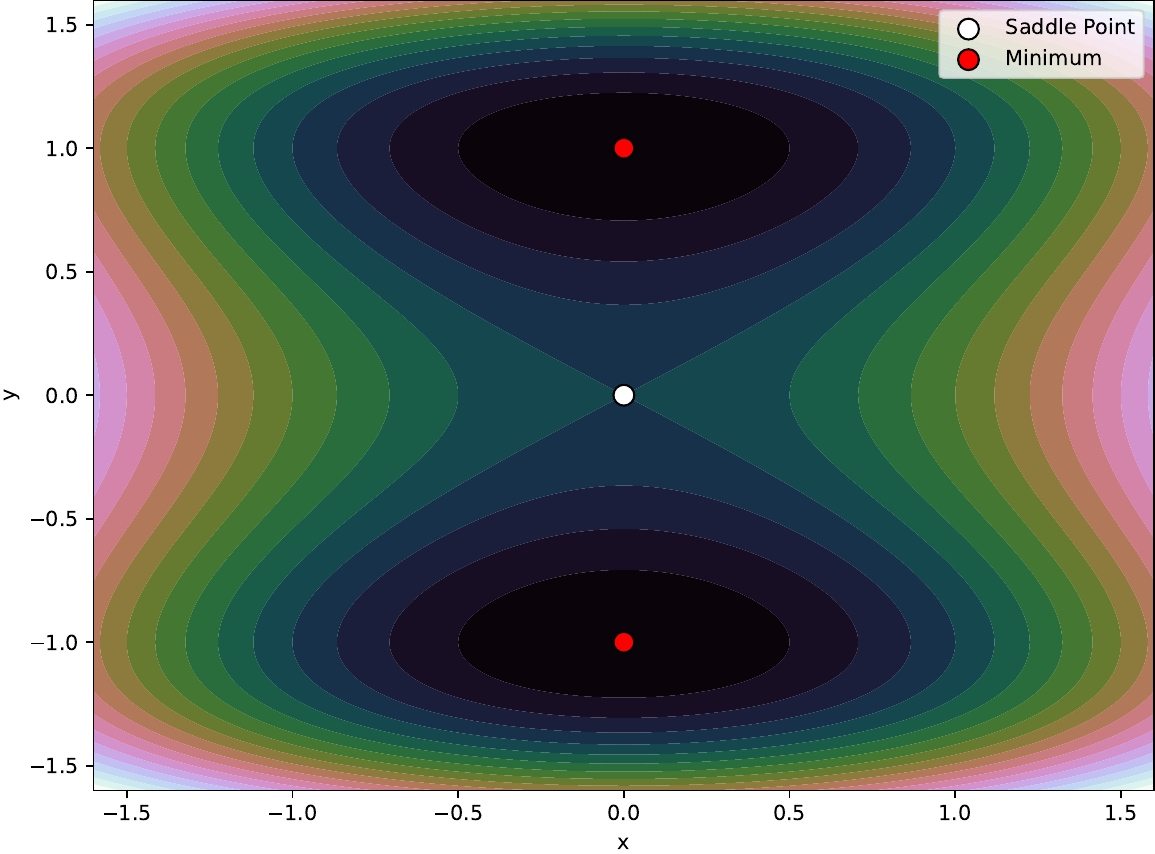}}%
        \caption{Stationary points of the function $h(x,y)= x^2 + \Big(y^2 - 1\Big)^2$: global minima at \((0,\pm 1)\) (red points) and a saddle point at \((0,0)\) (white point).}\label{fig:paraconvExa}
    \end{figure}
\end{exa}

For many problems of the form \cref{prb2}, there is a wide basin around the global minima, as suggested in \Cref{fig:paraconvExa}. Hence, if a local optimization algorithm initializes from a point in this basin, it should converge fast to a global minimizer. This motivates the quest for developing two-stage algorithms in which (i) the outer algorithm (e.g., a spectral approach) is used to find a reasonable approximation of a global minimizer in the basin of the cost function, and (ii) the inner algorithm is used to achieve fast convergence to a global minimizer. In more detail, some sufficient conditions can be considered under which variants of spectral methods \cite[Section VIII]{chi2019nonconvex} can bring us to the vicinity of a global minimizer, and then we can apply a proper algorithm (e.g., PSMs; See \Cref{sec:subGradMethod}) to converge to this global minimizer.

In the next result, we introduce a neighborhood of the optimal solution set of (\ref{prb2}) that is devoid of extraneous stationary points as a direct result of $\nu$-paraconvexity and H\"{o}lderian error boundness. Note that this region ensures a suitable region around the global minia such that if a proper algorithm is initialized from there, then it can generate a sequence converging to the global minimizer. In the following result, let us recall that \Cref{ass} holds.

\begin{lem}\label{lem-dis2}
    Let $x\in X$ be a stationary point of \cref{prb2}. Then, either $x\in \mathcal{X}^*$ or
    \begin{equation}\label{dis2}
        \bigg(\tfrac{\mu}{\rho}\bigg)^{\tfrac{\delta}{\delta(1+\nu)-1}} \leq \dist(x;\mathcal{X}^*).
    \end{equation} 
\end{lem}

\begin{proof}
    Without loss of generality, let $x\in X\setminus \mathcal{X}^*$ be a stationary point of \cref{prb2}. Hence, $0\in \partial(f+\delta_X)(x)$ and \cref{thm:f-chara} yield that
    $$f(y) - f(x) \geq -\rho \Vert y-x\Vert^{(1+\nu)}, \quad \forall y\in X.$$
    Setting $\bar{x} \in \operatorname{proj}_{\mathcal{X}^*}{(x)}$ and applying the HEB, we come to
    \begin{equation*}
             \mu \dist^{\nicefrac{1}{\delta}}(x;\mathcal{X}^*)\leq f(x) - f(\bar x) \leq \rho \Vert x-\bar x\Vert^{1+\nu}
                =\rho \dist^{1+\nu}(x;\mathcal{X}^*),
    \end{equation*}
    leading to \cref{dis2}. 
\end{proof}

Now, for any $\gamma \in (0,1]$, let us define the tube
\begin{equation}\label{eq:Tgamma}
    \mathcal{T}_\gamma :=\left\{x \in  X : \dist(x;\mathcal{X}^*) < \left(\gamma\tfrac{\mu}{\rho}\right)^{\tfrac{\delta}{\delta(1+\nu)-1}}\right\},
\end{equation}
which contains no extraneous stationary points of the problem, due to \cref{lem-dis2}.
Additionally, we set
$$L := \sup \bigg\{\Vert \zeta\Vert : \zeta\in \partial f(x),~ x \in \mathcal{T}_{1}\bigg\}.$$
The following lemma provides a key relationship between $\mu$ and $L$. Specifically, for each $x\in \mathcal{T}_{1}$, the ratio
$\tau_{x} := \tau \dist^{\tfrac{1}{\delta}-1}(x;\mathcal{X}^*)$
lies within the interval $[0,1]$, where $\tau = \tfrac{\mu}{L}$.

\begin{lem}\label{lem:tau}
    Let $\gamma\in (0,1]$ and $x\in \mathcal{T}_{\gamma}$. Then, $\tau_{x} = \tau \dist^{\frac{1}{\delta}-1}(x;\mathcal{X}^*)\in [0,1]$. In addition, if $\delta=1$, then $\tau_{x}=\tau \in (0,1]$.
\end{lem}

\begin{proof}
    Let $x\in \mathcal{T}_{\gamma}$. If $x\in \mathcal{X}^*$, then trivially $\tau_{x}=0$ and there is nothing to prove, i.e., without loss of generality, we assume that
    $x\in \mathcal{T}_{\gamma} \setminus \mathcal{X}^*$.
    For $\bar x \in \operatorname{proj}_{\mathcal{X}^*} {(x)}$, invoking the mean value theorem \cite[Theorem 3.51]{mordukhovich2006variational}, there exist $z\in [x,\bar x)$ and $\zeta \in \partial f(z)$ satisfying
    $f(x) - f(\bar x) \leq \langle \zeta , x-\bar x \rangle.$
    Using the HEB, we come to
    $$\mu \dist^{\nicefrac{1}{\delta}}(x;\mathcal{X}^*)\leq f(x) -f(\bar x) \leq \Vert\zeta\Vert ~\Vert x-\bar x\Vert=\Vert\zeta\Vert ~\dist(x;\mathcal{X}^*).$$
    Furthermore, it is easy to see that $z\in \mathcal{T}_{\gamma}\subseteq \mathcal{T}_1$, and
    so $\Vert \zeta\Vert \leq L$.
    Thus, $\mu \dist^{\nicefrac{1}{\delta}}(x;\mathcal{X}^*)\leq L\dist(x;\mathcal{X}^*)$,
    and consequently $\tau_{x} \leq 1$.
    In the special case where $\delta=1$, it follows that $\tau_{x}=\tau=\tfrac{\mu}{L}\neq 0$.
\end{proof}


\section{Projected subgradient algorithm for paraconvex optimization}\label{sec:subGradMethod}

In this section, we present the projected subgradient methods (PSMs) for the nonsmooth and constrained paraconvex optimization problems of the form \cref{prb2} and establish their convergence analysis for several choices of step-sizes. Let us begin with the following generic PSM, where the starting point $x_0$ satisfies
\[
    x_0 \in \mathcal{T}_\gamma,
\]
as $\mathcal{T}_\gamma$ is defined in \cref{eq:Tgamma}.

\vspace{3mm}
\RestyleAlgo{boxruled}
\begin{algorithm}[H]
\DontPrintSemicolon
\KwIn{$x_0\in \mathcal{T}_\gamma$;~ $\gamma \in (0,1]$.}
\Begin
{   
    \While{the stopping criteria do not hold}{
       Choose $\zeta_k \in \partial f(x_k)$;\;
       
       Set $x_{k+1}=\operatorname{proj}_{X}{\Big(x_k - \alpha_{k} \tfrac{\zeta_k}{\Vert\zeta_k\Vert} \Big)}$ and $k=k+1$;\;
       
        }
    Set $x_{k+1}=x_k$;\;
}
\KwOut{$x_{k+1}$.} 
\caption{Generic Projected Subgradient Method (PSM)\label{alg:subgradient}}
\end{algorithm}

\vspace{3mm}
In the above algorithm, the step-size $\alpha_{k}>0$ plays a key role in the algorithm's progression.
In fact, by employing different step-sizes, such as fixed, diminishing, and Scaled Polyak's step-sizes, we can define various PSMs, where they can then be compared in terms of the convergence rates and overall numerical performance.
Recall that, by \Cref{ass}, the function $f$ is $\nu$-paraconvex on the open convex set $X'\supseteq X$.
This property implies that $f$ is locally Lipschitz continuous at each iterate $x_k\in X$; see \Cref{thm:paralip}~\ref{thm:paralip-2}.
Consequently, the Clarke's subdifferential $\partial f(x_k)$ is nonempty for all $k$.

We begin with the subsequent lemma establishing a fundamental recurrence. This result is instrumental in deriving the convergence rate of the PSMs.

\begin{lem}[{\bf Basic inequalities I}]\label{lembasicrecurrence}
    Let \(x_k\) and \(x_{k+1}\) be two consecutive iterations of \Cref{alg:subgradient} such that \(x_k \in \mathcal{T}_\gamma\) with $\gamma=\nicefrac{1}{2}$. Then, one has
    \begin{align}
        \dist^2(x_{k+1};\mathcal{X}^*) \leq \dist^2(x_{k};\mathcal{X}^*) -\tfrac{\alpha_{k}}{L} (f(x_{k}) - f^{*}) + \alpha_{k}^{2},\hspace{3mm}\label{eq-lem-basic01}\\
        \dist^2(x_{k+1};\mathcal{X}^*) \leq \dist^2(x_{k};\mathcal{X}^*) -\alpha_{k} \tau \dist^{\tfrac{1}{\delta}}( x_{k};\mathcal{X}^*)+ \alpha_{k}^{2}.\label{eq-lem-basic02}
    \end{align}
    If, in addition, $\alpha_{k}\leq \min\Big\{1, \min\{1,\tau\} (\tfrac{\mu}{2\rho})^{\tfrac{\delta}{\delta(1+\nu)-1}}\Big\}$, then $x_{k+1} \in \mathcal{T}_\gamma$.
\end{lem}
\begin{proof}
    For \Cref{alg:subgradient}, if $\zeta_k=0$, the claims are evident then from $x_{k+1}=x_{k}\in \mathcal{X}^*$.
    Without loss of generality, we assume $\zeta_k\neq 0$. To show \cref{eq-lem-basic01}, let $x^* \in \operatorname{proj}_{\mathcal{X}^*}(x_{k})$.
    Then, $\dist(x_{k};\mathcal{X}^*) = \Vert x_{k}-x^* \Vert$ and $f(x^{*})= f^*$.
    Applying on the nonexpansiveness of $\operatorname{proj}_{X}$, it holds that
    \begin{equation}\label{eq:lembasic1}
        \begin{array}{ll}
             \dist^2(x_{k+1};\mathcal{X}^*) &\leq \Vert x_{k+1}- x^* \Vert^{2} \leq \Vert(x_{k}-x^*)-\alpha_{k} \tfrac{\zeta_{k}}{\Vert \zeta_{k} \Vert}\Vert^{2} = \Vert x_{k}-x^* \Vert^{2}+\tfrac{2 \alpha_{k}}{\Vert\zeta_{k}\Vert} \langle\zeta_{k}, x^* - x_{k}\rangle + \alpha_{k}^{2} \\[2mm]
            & \leq  \dist^2(x_{k};\mathcal{X}^*)+\tfrac{2 \alpha_{k}}{\Vert \zeta_{k} \Vert} \left(f^{*}-f(x_{k})+ \rho \dist^{1+\nu}(x_{k};\mathcal{X}^*)\right) + \alpha_{k}^{2},
        \end{array}
    \end{equation}
    where the last inequality comes from $\nu$-paraconvexity of $f$, \cref{thm:f-chara}~\ref{thm:f-chara_p1}.
    Now, inasmuch as $x_{k}\in \mathcal{T}_\gamma\subseteq \mathcal{T}_{1}$, we have $\tfrac{\alpha_{k}}{L}\leq \tfrac{\alpha_{k}}{\Vert \zeta_{k} \Vert}$ and $\dist(x_{k};\mathcal{X}^*) < \left(\tfrac{\mu}{2\rho}\right)^{\tfrac{\delta}{\delta(1+\nu)-1}}$, i.e.,
    $$\rho \dist^{1+\nu}(x_{k};\mathcal{X}^*) < {\tfrac{1}{2}}\,\mu \dist^{\tfrac{1}{\delta}}(x_{k};\mathcal{X}^*) \leq {\tfrac{1}{2}}(f(x_{k})-f^{*}),$$
    using the HEB.
    Substituting these bounds into \cref{eq:lembasic1} results in \cref{eq-lem-basic01}, i.e.,
    \begin{align*}
        \dist^2(x_{k+1};\mathcal{X}^*)
        \leq \dist^2(x_{k};\mathcal{X}^*)+\tfrac{2\alpha_{k}}{\Vert \zeta_{k} \Vert} \Big(f^{*}-f(x_{k})+ {\tfrac{1}{2}} (f(x_{k})-f^{*})\Big) + \alpha_{k}^{2}\leq \dist^2(x_{k};\mathcal{X}^*)-\tfrac{\alpha_{k}}{L} \left(f(x_{k})-f^{*}\right) + \alpha_{k}^{2}.
    \end{align*}
    Moreover, the argument for \cref{eq-lem-basic02} is derived from \cref{eq-lem-basic01}, relying on the HEB, i.e.,
    $$
    \begin{aligned}
        \dist^2(x_{k+1};\mathcal{X}^*) &\leq
        \dist^2(x_{k};\mathcal{X}^*)-\tfrac{\alpha_{k}}{L} \left(f(x_{k})-f^{*}\right) + \alpha_{k}^{2}
        \leq \dist^2(x_{k};\mathcal{X}^*) - \alpha_{k} \tau \dist^{\tfrac{1}{\delta}}(x_{k};\mathcal{X}^*)  + \alpha_{k}^{2}.
    \end{aligned}
    $$
    For deriving the last part, $x_{k+1} \in \mathcal{T}_\gamma$, we consider two cases: (i) $\alpha_{k} \leq \tau\dist^{\tfrac{1}{\delta}}(x_{k};\mathcal{X}^*)$; (ii) $\alpha_{k} > \tau\dist^{\tfrac{1}{\delta}}(x_{k};\mathcal{X}^*)$.
    
    In Case~(i), $\alpha_{k} \leq \tau\dist^{\tfrac{1}{\delta}}(x_{k};\mathcal{X}^*)$, applying the above inequality to \cref{eq-lem-basic01} yields
    $$
    \dist^2(x_{k+1};\mathcal{X}^*) \leq \dist^2(x_{k};\mathcal{X}^*) -\alpha_{k} \tau \dist^{\tfrac{1}{\delta}}(x_{k};\mathcal{X}^*) + \alpha_{k}^{2} \leq \dist^2(x_{k};\mathcal{X}^*)
    < \left(\tfrac{\mu}{2\rho}\right)^{\tfrac{2\delta}{1-\delta(1+\nu)}}.
    $$
    i.e., $x_{k+1}\in \mathcal{T}_\gamma$. 
    
    For Case~(ii), $\alpha_{k} > \tau\dist^{\tfrac{1}{\delta}}(x_{k};\mathcal{X}^*)$, leading to
    \begin{equation}\label{eq:lembasic}
        \dist(x_{k};\mathcal{X}^*) < \left(\tfrac{\alpha_{k}}{\tau}\right)^{\delta}.
    \end{equation}
    Take the convex function
    $\varphi(t) := t - \alpha_{k} \tau t^{\tfrac{1}{2\delta}}   + \alpha_{k}^{2}$
    on $[0,(\tfrac{\alpha_{k}}{\tau})^{2\delta}]$
    where the maximum is achieved at $0$ or $(\tfrac{\alpha_{k}}{\tau})^{2\delta}$.
    On the basis of \cref{eq-lem-basic01} and \cref{eq:lembasic}, $\dist^2(x_{k+1};\mathcal{X}^*) \leq \varphi (\dist^{2}(x_{k};\mathcal{X}^*))$ and
    $\dist^{2}(x_{k};\mathcal{X}^*) \in [0,(\tfrac{\alpha_{k}}{\tau})^{2\delta}]$,
    i.e.,
    \begin{align*}
        \dist^2(x_{k+1};\mathcal{X}^*) &\leq \max \left\{\varphi(t): ~ t\in [0,(\tfrac{\alpha_{k}}{\tau})^{2\delta}]\right\}= \max \left\{ \alpha^{2}_{k},~
        (\tfrac{\alpha_{k}}{\tau})^{2\delta} - \alpha_{k} \tau \tfrac{\alpha_{k}}{\tau}  + \alpha^{2}_{k}\right\}\leq \max \left\{ \alpha^{2\delta}_{k},~
        (\tfrac{\alpha_{k}}{\tau})^{2\delta} \right\}\\
        &\leq \max \left\{ 1,~
        \tfrac{1}{\tau^{2\delta}} \right\}\min\{1, \tau^{2\delta}\}(\tfrac{\mu}{2\rho})^{\tfrac{2\delta}{\delta(1+\nu)-1}} 
        =(\tfrac{\mu}{2\rho})^{\tfrac{2\delta}{\delta(1+\nu)-1}},
    \end{align*}
    where the second and third inequalities are derived from the bound on $\alpha_{k}$, i.e., $x_{k+1}\in \mathcal{T}_\gamma$.
\end{proof}

Next, we drive an upper bound for the difference between the lowest value of the function over the iterations and the optimal value. Let us define $f_{k}^{*}:=\min\{f(x_{i}) : i=1,2,\ldots,k\}$ for all $k\ge 0$.

\begin{lem}[{\bf Basic inequalities II}]\label{lembasicrecurrence2}
    Let $\{x_{k}\}_{k\geq 0}\subseteq \mathcal{T}_\gamma$ with $\gamma=\nicefrac{1}{2}$ is generated by \Cref{alg:subgradient}. Then, one has
    \begin{align*}
        f(x_{k}) - f^{*} \leq \frac{L\dist^2(x_{k};\mathcal{X}^*)+ L \alpha_{k}^{2}}{\alpha_{k}},~~\\
        f_{k}^{*}- f^{*} \leq \frac{L\dist^2(x_{1};\mathcal{X}^*)+ L \sum_{i=1}^{k}\alpha_{i}^{2}}{\sum_{i=1}^{k}\alpha_{i}}.
    \end{align*}
\end{lem}
\begin{proof}
    These inequalities follow directly 
    from \cref{lembasicrecurrence}.
\end{proof}

\subsection{{\bf PSM with constant step-size}}\label{sec:subGradMethodWithConstantstep-size}
Here, we analyze a PSM with a constant step-size, $\alpha_{k}=\alpha>0$ for all $k$.
We show that, with appropriate initialization, the sequence $D_k:= \dist^2(x_k; \mathcal{X}^*)$ converges linearly below a certain fixed threshold.

\begin{thm}[{\bf Convergence of constant step-size method}]\label{thm:constants}
    Let $\tau=\tfrac{\mu}{L}$. Fix a constant step-size $\alpha$ satisfying
    $$
    0<\alpha<\min\left\{\tfrac{2\delta}{\tau}\left(\tfrac{\mu}{2\rho} \right)^{\tfrac{2\delta-1}{\delta(1+\nu)-1}},~\tfrac{\tau}{(\tau^{2\delta}+1)^{\tfrac{1}{2\delta}}}\left(\tfrac{\mu}{2\rho} \right)^{\tfrac{1}{\delta(1+\nu)-1}},~1\right\}.
    $$
    Let $\{x_{k}\}_{k\geq 0}$ be generated by \Cref{alg:subgradient} with $\gamma=\nicefrac{1}{2}$ and the constant step-size $\alpha_{k}=\alpha$. Define the constants
    $$D_{*}:= \left(\tfrac{\alpha}{\tau}\right)^{2\delta},\quad\quad
    q:= 1- \tfrac{\alpha\tau}{2\delta}\left(\tfrac{2\rho}{\mu} \right)^{\tfrac{2\delta-1}{\delta(1+\nu)-1}},\quad\quad 
     \mathcal{D}^{2}:=\max \left\{D_{0}, \alpha^{2}+D_{*}\right\}.$$
    Then, $0<q<1$, and for any $ k\in \N$, we have
    \begin{equation}\label{eq:constant:iteration}
         \sqrt{D_{k}} \leq \mathcal{D} \leq \left(\tfrac{\mu}{2\rho}\right)^{\tfrac{\delta}{\delta(1+\nu)-1}},
        \quad \text { and } \quad
        D_{k}-D_{*} \leq \max \left\{q^{k}\left(D_{0}-D_{*}\right),
        \alpha^{2}\right\}.   
    \end{equation}
\end{thm}
\begin{proof}
    We first verify claims about $\mathcal{D}$ and $q$, and then use the induction to establish the iteration-dependent inequalities.
    First, we will verify $\mathcal{D} \leq \left(\tfrac{\mu}{2\rho}\right)^{\tfrac{\delta}{\delta(1+\nu)-1}}$.
    Inasmuch as $x_0\in \mathcal{T}_\gamma$, $D_0 \leq \left(\tfrac{\mu}{2\rho}\right)^{\tfrac{2\delta}{\delta(1+\nu)-1}}$ due to definition of $\mathcal{T}_\gamma$.
    Moreover, using the upper bound of $\alpha$, we get
    \begin{align*}
        \alpha^{2}+D_{*}
        \leq \alpha^{2\delta} + \tfrac{\alpha^{2\delta}}{\tau^{2\delta}}
        =\alpha^{2\delta} \tfrac{\tau^{2\delta}+1}{\tau^{2\delta}} \leq \left(\tfrac{\mu}{2\rho} \right)^{\tfrac{2\delta}{\delta(1+\nu)-1}},
    \end{align*}
    i.e.,   $\mathcal{D} \leq \left(\tfrac{\mu}{2\rho}\right)^{\tfrac{2\delta}{\delta(1+\nu)-1}}.$
    Regarding to $q \in(0,1)$, it is  clear that  $q<1$.
    Considering
    $$
    0 < \alpha <\tfrac{2\delta}{\tau}\left(\tfrac{\mu}{2\rho} \right)^{\tfrac{2\delta-1}{\delta(1+\nu)-1}},
    $$
    and applying the basic calculus, it holds that $q>0$.
    We now prove the inequalities in \Cref{eq:constant:iteration} using induction. At $k=0$, the inequalities hold trivially. Now by induction, assume that for some $k \in \mathbb{N}$, the following holds (inductive hypothesis):
    $$
    \sqrt{D_{i}} \leq \mathcal{D} \, \,  \text { and } \,\, D_{i}-D_{*} \leq \max \left\{q^{i}\left(D_{0}-D_{*}\right), \alpha^{2}\right\}, \quad \forall i=0,1, \ldots, k.
    $$
    We will verify these claims for $(k+1)$-th step. From \cref{lembasicrecurrence}, and the concavity of function $t^{\tfrac{1}{2\delta}}$ on $[0,+\infty)$, it follows that
    \begin{align*}
        D_{k+1} \leq D_{k} -\alpha \tau D_{k}^{\tfrac{1}{2\delta}}+\alpha^{2}
        =D_{k} -\alpha \tau\left( D_{k}^{\tfrac{1}{2\delta}} - \tfrac{\alpha}{\tau}\right)
        =D_{k} -\alpha \tau \left(D_{k}^{\tfrac{1}{2\delta}} - D_{*}^{\tfrac{1}{2\delta}}\right)
        \leq D_{k} - \tfrac{\alpha\tau}{2\delta D_{k}^{1-\tfrac{1}{2\delta}}}(D_{k}-D_{*}),
    \end{align*}
    i.e.,
    \begin{equation*}
        D_{k+1}-D_{*} \leq \bigg(1 - \tfrac{\alpha\tau}{2\delta D_{k}^{1-\tfrac{1}{2\delta}}}\bigg)(D_{k}-D_{*}).
    \end{equation*}
    We have two cases: (i) $D_{k} \geq D_{*}$; (ii) $D_{k} < D_{*}$.\\
    In Case~(i), since  $D_{k} \leq \mathcal{D}^{2} \leq \left(\tfrac{\mu}{2\rho}\right)^{\tfrac{2\delta}{\delta(1+\nu)-1}}$, we obtain
    $$
    D_{k+1}-D_{*} 
    \leq \bigg(1 - \tfrac{\alpha\tau}{2\delta D_{k}^{1-\tfrac{1}{2\delta}}}\bigg)(D_{k}-D_{*})
    \leq \bigg(1 - \tfrac{\alpha\tau}{2\delta \left(\tfrac{\mu}{2\rho}\right)^{\tfrac{2\delta}{\delta(1+\nu)-1} \tfrac{2\delta-1}{2\delta}}}\bigg)(D_{k}-D_{*})
    =q\left(D_{k}-D_{*}\right).
    $$
    Moreover, $\sqrt{D_{k+1}} \leq \sqrt{D_{k}} \leq \mathcal{D}$, inasmuch as $0<q<1$.
    In Case~(ii), $D_{k}<D_{*}$, relying once again on \cref{lembasicrecurrence}, we get
    $$
    D_{k+1} - D_{*} \leq D_{k} -\alpha \tau D_{k}^{\tfrac{1}{2\delta}}+\alpha^{2} - D_{*}\leq
    D_{k} +\alpha^{2} - D_{*}< \alpha^{2}.
    $$
    In addition, $\sqrt{D_{k+1}} \leq \sqrt{\alpha^{2}+ D_{*}} \leq \mathcal{D}$.
    Hence, in both cases and by the induction assumption, it can be deduced
    $$
        D_{k+1}-D_{*} \leq \max \left\{q\left(D_{k}-D_{*}\right), \alpha^{2}\right\}\leq \max \left\{q^{k+1}\left(D_{0}-D_{*}\right),
        \alpha^{2}\right\}.
    $$
    Since $\sqrt{D_{k+1}} \leq \mathcal{D}$, by induction, the claims hold for all $k\geq 0$.
\end{proof}

\cref{thm:constants} establishes that the sequence
$D_{k}$ decrease to a value below $D_{*}$ at a linear rate. Furthermore, even if $D_k$ becomes smaller than $D_{*}$, it remains close to this threshold.

The subsequent result is a straightforward consequence of \cref{thm:constants} for $\delta= 1$ under \Cref{ass}.

\begin{cor}\label{cor:constants}
    Let $\delta= 1$ and $\tau=\nicefrac{\mu}{L}$. 
    Let $\{x_{k}\}_{k\ge 0}$ be generated by \Cref{alg:subgradient} with $\gamma=\nicefrac{1}{2}$ and the constant step-size $\alpha_{k}=\alpha$ satisfying $0<\alpha<\min\big\{\tfrac{\tau}{\sqrt{2}}\big(\tfrac{\mu}{2\rho} \big)^{\nicefrac{1}{\nu}},~1\big\}$. Then, 
    $0<q:= 1- \tfrac{\alpha\tau}{2}\left(\tfrac{2\rho}{\mu} \right)^{\nicefrac{1}{\nu}}$ and for any $k\in \N$, one has
    $$
    D_{k} \leq \max \left\{D_{0}, \alpha^{2}+D_{*}\right\} \leq \Big(\dfrac{\mu}{2\rho}\Big)^{\nicefrac{2}{\nu}},
    \quad \text { and } \quad
    D_{k}-\tfrac{\alpha^{2}}{\tau^{2}} \leq \max \left\{q^{k}\left(D_{0}-\tfrac{\alpha^{2}}{\tau^{2}}\right), \alpha^{2}\right\}.
    $$
\end{cor}

\begin{proof}
    By straightforward calculations, the result follows directly from \cref{thm:constants}.
\end{proof}

The next theorem provides an upper bound, determined by the step-size, for the distance of the lowest estimated value over iterations from the optimal value.

\begin{thm}[{\bf Convergence rate of constant step-size method}]\label{thm-css}
    Let $\{x_{k}\}_{k\ge 1}$ be generated by \Cref{alg:subgradient} with $\gamma=\nicefrac{1}{2}$ and the constant step-size $\alpha_{k}=\alpha$ satisfying
    $
    0<\alpha<\min\big\{1,\,\min\{1,\tfrac{\mu}{L}\} (\tfrac{\mu}{2\rho})^{\tfrac{1}{\delta(1+\nu)-1}}\big\}.
    $ Then,
    
    $$
    0\leq f^{*}_{k} - f^{*} \leq 2L\alpha, \quad\quad\forall k\geq \tfrac{1}{\alpha^{2}}\left(\tfrac{\mu}{2\rho}\right)^{\tfrac{2\delta}{\delta(1+\nu)-1}}.
    $$
\end{thm}
\begin{proof}
    This inequality follows directly form \Cref{lembasicrecurrence2}, together with
    \cref{lembasicrecurrence}.
\end{proof}

\subsection{{\bf PSMs with diminishing step-sizes}}\label{sec:subGradMethodWithDiminishingstep-size}

In the previous section, we analyzed the constant step-size scheme and established convergence results within a fixed threshold.
Here, we shift our focus to achieving convergence to the optimal solution by studying the PSM with a diminishing step-size.
The next result validates the convergence of 
$\{f^{*}_{k}\}_{k\ge 1}$ and of a subsequence of $\{x_k\}_{k\ge 1}$ in which $\{x_{k}\}_{k\ge 1}$ is generated by the PSM with Nonsummable Diminishing (ND) step-size satisfying
$$
\alpha_{k}\geq 0,\quad \lim_{k\to\infty}\alpha_{k}=0, \quad \sum_{k=1}^{\infty}\alpha_{k} =\infty.
$$

\begin{thm}[{\bf Convergence of ND method}]\label{thm-nsds}
    Let $\{x_{k}\}_{k\geq 1}$ be generated by \Cref{alg:subgradient} with $\gamma=\nicefrac{1}{2}$ and ND step-size $\alpha_k$ 
    satisfying
    $0< \alpha_{k}<\min \Big\{1,\,\min\{1,\tfrac{\mu}{L}\} (\tfrac{\mu}{2\rho})^{\tfrac{1}{\delta(1+\nu)-1}}\Big\}.$
    Then,
    $$
    \lim_{k\to \infty} f^{*}_{k} = f^{*},\quad\quad
    \liminf_{k\to\infty} f(x_{k}) = f^{*},\quad\text{and}\quad \liminf_{k\to\infty} \dist(x_{k};\mathcal{X}^*)=0.
    $$
    Moreover, if $\{x_{k}\}_{k\geq 1}$ is a bounded sequence, it has a convergent subsequence to some optimal solution $x^{*}\in \mathcal{X}^*$.
\end{thm}
\begin{proof}
    According to \cref{lembasicrecurrence,lembasicrecurrence2}, we have
    $$
    0\leq f_{k}^{*}- f^{*} \leq \frac{L\dist^2(x_{1};\mathcal{X}^*)+ L \sum_{i=1}^{k}\alpha_{i}^{2}}{\sum_{i=1}^{k}\alpha_{i}}.
    $$
    Let us consider a fixed $j\in \mathbb{N}$.
    For $k>j$, it holds that
    $$
        0\leq \frac{\sum_{i=1}^{k}\alpha_{i}^{2}}{\sum_{i=1}^{k}\alpha_{i}}
        =\frac{\sum_{i=1}^{j-1}\alpha_{i}^{2}}{\sum_{i=1}^{k}\alpha_{i}} + \frac{\sum_{i=j}^{k}\alpha_{i}^{2}}{\sum_{i=1}^{k}\alpha_{i}}
        \leq\frac{\sum_{i=1}^{j-1}\alpha_{i}^{2}}{\sum_{i=1}^{k}\alpha_{i}} + \displaystyle\max_{j\leq i\leq k}\{\alpha_{i}\}\frac{\sum_{i=j}^{k}\alpha_{i}}{\sum_{i=1}^{k}\alpha_{i}}
        \leq\frac{\sum_{i=1}^{j-1}\alpha_{i}^{2}}{\sum_{i=1}^{k}\alpha_{i}} +
        \displaystyle\max_{j\leq i}\{\alpha_{i}\}.
    $$
    Taking the limit as $k\to \infty$, we get
    $0\leq\tfrac{\sum_{i=1}^{\infty}\alpha_{i}^{2}}{\sum_{i=1}^{\infty}\alpha_{i}}\leq \max_{j\leq i}\{\alpha_{i}\}$, for each $j\in \N$.
    Now, taking limit as $j\to \infty$, we conclude that $\tfrac{\sum_{i=1}^{\infty}\alpha_{i}^{2}}{\sum_{i=1}^{\infty}\alpha_{i}}= 0$.
    On the other hand, $\tfrac{L\dist^2(x_{1};\mathcal{X}^*)}{\sum_{i=1}^{\infty}\alpha_{i}}= 0$.
    Thus, $\displaystyle\lim_{k\to \infty} f_{k}^{*}= f^{*}$.
    
    Next, we demonstrate that $\displaystyle\liminf_{k\to\infty} f(x_{k}) = f^{*}.$
    To this end, we construct a subsequence $\{x_{k_j}\}_{j\in \mathbb{N}}$ of $\{x_{k}\}_{k\in \mathbb{N}}$ converging to some optimal solution.
    If $f(x_{k})=f^{*}$ for some $k\in \mathbb{N}$, the result is valid.
    Otherwise, assume $f^{*}<f(x_{k})$, for all $k\in \mathbb{N}$.
    For each $k\in \mathbb{N}$, there is nondecreasing index $i_{k} \in \{1,2,\ldots,k\}$ such that $f^{*}_{k} = f(x_{i_{k}})$.
    Therefore, we have $\displaystyle\lim_{k\to \infty} f(x_{i_{k}}) =f^{*}$.
    Let us define $k_{1}:=i_{1}=1$.
    Since $f^{*}<f(x_{1})$, there exists $k\in \mathbb{N}$, such that $f(x_{i_{k}}) < f(x_{1})$.
    Now, let us set $k_{2}:=\min\{i_{k}\in \mathbb{N}: f(x_{i_{k}}) < f(x_{1})\}$, i.e., $k_{1}=1<k_{2}$ and $f(x_{k_{2}})<f(x_{k_{1}})$. Continuing this process, we construct a subsequence $x_{k_j}$ such that
    $$
    k_{1}<k_{2}<\ldots<k_{j},\quad\text{and}\quad f(x_{k_{j}})<f(x_{k_{j-1}})<\ldots<f(x_{k_{1}}),
    $$
    where $k_{t} =\min \{i_{k}\in \mathbb{N}: f(x_{i_{k}}) < f(x_{k_{t-1}})\}$ for each $1\leq t\leq j$.
    Now, noticing $\displaystyle\lim_{k\to \infty} f(x_{i_{k}})=f^{*}<f(x_{k_{j}})$, there exists $k\in \mathbb{N}$ such that $f(x_{i_{k}}) < f(x_{k_{j}})$.
    As such, setting  $k_{j+1}:=\min \{i_{k}\in \mathbb{N}: f(x_{i_{k}}) < f(x_{k_{j}})\}$,
    we have $f(x_{k_{j+1}}) < f(x_{k_{j}})$ and $k_{j}<k_{j+1}$.
    Otherwise, if $k_{j+1}<k_{j}$, $f(x_{k_{j+1}}) < f(x_{k_{j}})< f(x_{k_{j-1}})$, this leads to a contradiction.
    Thus, $\{x_{k_j}\}_{j\in\mathbb{N}}$ is a subsequence of the sequences $\{x_k\}_{k\in\mathbb{N}}$ and $\{x_{i_k}\}_{k\in\mathbb{N}}$.
    Therefore, $\displaystyle\lim_{j\to \infty}f(x_{k_j})=\displaystyle\lim_{k\to \infty} f(x_{i_{k}})= f^{*}$
    and
    $$
    f^{*}\leq \liminf_{k\to\infty} f(x_{k})
    \leq \lim_{j\to \infty}f(x_{k_j})=f^{*}.
    $$
    This implies  $
    \displaystyle\liminf_{k\to\infty} f(x_{k}) = f^{*}$. Furthermore, the HEB yields $\liminf_{k\to\infty} \dist(x_{k};\mathcal{X}^*)=0$.
    
    We address the final claim by assuming that
    $\{x_{k}\}_{k\in \N}$ is a bounded sequence.
    Consequently, the sequence $\{x_{k_j}\}_{j\in \N}$ has a convergent subsequence.
    Without loss of generality, we assume $x_{k_j}\to x^{*}$ for some $x^{*}\in X$, i.e., 
    $f^{*}=\lim_{j\to \infty}f(x_{k_j})=f(x^{*}).$
    As such, $x^{*}\in \mathcal{X}^*$, completing the proof.
\end{proof}

The previous theorem established that PSMs with nonsummable step-size exhibit subsequential convergence.
The following theorem extends this result by demonstrating the convergence of $\{f(x_k)-f^*\}_{k\geq 1}$ and $\{\dist(x_k;\mathcal{X}^*)\}_{k\geq 1}$ where $\{x_{k}\}_{k\geq 1}$ is generated by the PSM when the step-size is Square-Summable but Not Summable (SSNS) satisfying the following conditions:
$$
\alpha_{k}\geq 0, \quad \sum_{k=1}^{\infty}\alpha_{k} =\infty, \quad \sum_{k=1}^{\infty}\alpha_{k}^{2} <\infty.
$$

\begin{thm}[{\bf Convergence of SSNS method}]\label{thm-ssns}
    Let $\{x_{k}\}_{k\geq 1}$ be generated by \Cref{alg:subgradient} with $\gamma=\nicefrac{1}{2}$ and the SSNS step-size satisfying
    $0\leq \alpha_{k}<\min \Big\{1,\,\min\{1,\tfrac{\mu}{L}\} (\tfrac{\mu}{2\rho})^{\tfrac{1}{\delta(1+\nu)-1}}\Big\}.$
    Then,
    $$
    \lim_{k\to \infty} f(x_{k})= \lim_{k\to \infty} f^{*}_{k} = f^{*},\quad\text{and}\quad \lim_{k\to\infty} \dist(x_{k};\mathcal{X}^*)=0.
    $$
\end{thm}

\begin{proof}
    Invoking \cref{thm-nsds}, it is clear that
    $
    \liminf_{k\to \infty} f(x_{k})= \lim_{k\to \infty} f^{*}_{k} = f^{*},
    $
    which ensures that
    \begin{equation}\label{eq:ssns1}
        \lim_{j\to \infty} f(x_{k_{j}}) = f^{*},
    \end{equation}
    for some subsequence $\{x_{k_j}\}_{j\in\mathbb{N}}$ of $\{x_{k}\}_{k\in\mathbb{N}}$.
    We first show that $\dist(x_{k};\mathcal{X}^*)$ tends to zero.
    By \cref{lembasicrecurrence}, for each $k\in \mathbb{N}$, we have $x_{k}\in \mathcal{T}_\gamma$, and
    \begin{align}
        \dist^2(x_{k+1};\mathcal{X}^*) &\leq \dist^2(x_{k};\mathcal{X}^*) -\tfrac{\alpha_k}{L}\left(f(x_{k})-f^{*}\right)+\alpha_{k}^{2}
        \leq
        \dist^2(x_{k-1};\mathcal{X}^*) +\alpha_{k-1}^{2}+\alpha_{k}^{2}\nonumber\\[-1mm]
        &\leq \ldots\label{eq:ssns2}\leq \dist^2(x_{1};\mathcal{X}^*) +\sum_{i=1}^{k} \alpha_{i}^{2}\leq \left(\tfrac{\mu}{2\rho}\right)^{\tfrac{2\delta}{\delta(1+\nu)-1}} +\sum_{i=1}^{\infty} \alpha_{i}^{2}.
    \end{align}
    As such, $\dist(x_{k};\mathcal{X}^*)$ is a bounded sequence, inasmuch as
    $\sum_{i=1}^{\infty} \alpha_{i}^{2}<\infty$, i.e., $\dist(x_{k_j};\mathcal{X}^*)$ has a convergent subsequence.
    Without loss of generality, we assume that $\dist(x_{k_j};\mathcal{X}^*)$ converges to some $d\geq 0$, and we hence need to show $d=0$. 
    Suppose on the contrary that $d>0$, i.e., for sufficiently large $j$, $\dist(x_{k_j};\mathcal{X}^*)> \tfrac{d}{2}$.
    Using this inequality and the HEB, it follows that
    $$
    \mu\left(\tfrac{d}{2}\right)^{\nicefrac{1}{\delta}}
    < \mu\dist^{\nicefrac{1}{\delta}}(x_{k_j};\mathcal{X}^*)
    \leq f(x_{k_j})-f^{*},
    $$
    for sufficiently large $j$, which contradicts \cref{eq:ssns1}, i.e., $\dist(x_{k_j};\mathcal{X}^*)\to 0$ as $j\to\infty$.
    Let $\epsilon>0$ be arbitrary and fixed.
    There exists $j_{0}\in \mathbb{N}$ such that $\dist(x_{k_{j_{0}}};\mathcal{X}^*)<\tfrac{\epsilon^{2}}{2}$ and $\sum_{i=k_{j_{0}}}^{k} \alpha_{i}^{2}<\tfrac{\epsilon^{2}}{2}$, for each $k\geq k_{j_{0}}$.
    Following \cref{eq:ssns2}, for each $k\geq k_{j_{0}}$, yields
    \begin{center}
        $
        \dist^2(x_{k};\mathcal{X}^*)\leq \dist^2(x_{k_{j_{0}}};\mathcal{X}^*) +\sum_{i=k_{j_{0}}}^{k} \alpha_{i}^{2} <
        \tfrac{\epsilon^{2}}{2}+\tfrac{\epsilon^{2}}{2}=\epsilon^{2}.
        $
    \end{center}
    Consequently, $\dist(x_{k};\mathcal{X}^*)\to 0$.
    On the other hand, Using the mean value theorem \cite[Theorem 3.51]{mordukhovich2006variational} it can be concluded that
    $$0\le f(x_k) - f^* \leq L\dist(x_k;\mathcal{X}^*),\quad\quad k\ge 1,$$
    confirming $f(x_k)\to f^*$ and
    adjusting our desired results.
\end{proof}

We next analyze the convergence of the PSMs with specific step-size schemes, including diminishing step-sizes and geometrically decaying step-sizes.
First, we focus on the PSM with diminishing step-size of the form, $\alpha_{k}=\lambda (k+ k_{0})^{-r}$ where $\lambda, k_{0}>0$ and $0<r<1$. Such step-sizes are widely used in both stochastic and deterministic versions of the subgradient method.
The sublinear convergence of the PSM with this step-size is established in the following result.

\begin{thm}[{\bf Convergence rate of diminishing method}]\label{thm:nonsummablestep-size}
    Let $r\in (0,1)$ and $\lambda >0$.
    Set
    $$
    A:=2^{r\delta}~\sqrt{\left(\tfrac{2\lambda}{\tau}\right)^{2\delta}+\lambda^{2}},\quad \text{and}\quad  k_{0}:=\max\bigg\{ \Big(\tfrac{4r\delta A^{\tfrac{2\delta-1}{\delta}}}{\lambda\tau}\Big)^{\tfrac{1}{1-2r(1-\delta)}}, A^{\tfrac{1}{\delta r}}\left(\tfrac{2\rho}{\mu}\right)^{\tfrac{1}{r(\delta(1+\nu)-1)}},1\bigg\}.
    $$
    Let $\{x_{k}\}_{k\geq 0}$ be generated by \Cref{alg:subgradient} with $\gamma=\nicefrac{1}{2}$ and step-size $\alpha_{k}=\lambda (k+ k_{0})^{-r}$ and let the initial point $x_{0} \in \mathcal{T}_\gamma$ satisfying $\dist(x_{0};\mathcal{X}^*)\leq \min\left\{(\tfrac{2\lambda}{\tau})^{\delta}, A k_{0}^{-\delta r}\right\}$. Then,
    \begin{equation}\label{eq:nonsummableconvergence}
        \dist(x_{k} ; \mathcal{X}^*) \leq \dfrac{A}{(k+ k_{0})^{\delta r}}.
    \end{equation}
\end{thm}

\begin{proof}
    Clearly, $\dist(x_{0};\mathcal{X}^*)\leq A k_{0}^{-\delta r}$, by virtue of the initialization condition.
    Let us proceed by induction. Assume that the inequality \cref{eq:nonsummableconvergence} holds for $k$, $\dist(x_{k};\mathcal{X}^*)\leq A(k+ k_{0})^{-\delta r}$.
    Noticing $ k_{0}\geq A^{\tfrac{1}{\delta r}}\left(\tfrac{2\rho}{\mu}\right)^{\tfrac{1}{r(\delta(1+\nu)-1)}}$, it follows that $x_{k}\in \mathcal{T}_\gamma$.
    From \cref{lembasicrecurrence}, we obtain
    \begin{equation}\label{eq:nonsummable1}
        \dist^2(x_{k+1};\mathcal{X}^*) \leq \dist^2(x_{k};\mathcal{X}^*)- \lambda (k+ k_{0})^{-r} \tau  \dist^{\nicefrac{1}{\delta}}(x_{k};\mathcal{X}^*)  + \lambda^{2} (k+ k_{0})^{-2r}.
    \end{equation}
    Let us consider the index set
    $$
    I:=\Big\{i\in \mathbb{N} :~ {\tfrac{1}{2}}\, \tau \dist^{\nicefrac{1}{\delta}}(x_{i};\mathcal{X}^*) \leq \frac{\lambda}{(i+ k_{0})^{r}}\Big\}.
    $$
    There are two possible cases for the index $k+1$: (i) $k+1\in I$; (ii) $k+1\not\in I$.
    In Case~(i), $k+1\in I$, it follows that: $\dist(x_{k+1};\mathcal{X}^*) \leq \frac{A}{(k+1+ k_{0})^{\delta r}}$,
    since $\left(\tfrac{2\lambda}{\tau}\right)^\delta < A$.
    Alternatively, in Case~(ii), $k+1\not\in I$, let us further consider two subcases: $k$: either $k\in I$ or $k\not\in I$. In the first case, $k\in I$, employing the inequality \cref{eq:nonsummable1}, we come to
    $$
        \dist^2(x_{k+1};\mathcal{X}^*) \leq \dist^2(x_{k};\mathcal{X}^*)  + \frac{\lambda^{2}}{(k+ k_{0})^{2r}}\leq  \frac{\left(\tfrac{2\lambda}{\tau}\right)^{2\delta}}{(k+ k_{0})^{2\delta r}} + \frac{\lambda^{2}}{(k+ k_{0})^{2r}}
        \leq \left(\left(\tfrac{2\lambda}{\tau}\right)^{2\delta}+ \lambda^{2}\right) \frac{1}{(k+ k_{0})^{2\delta r}}
        \leq \frac{A^{2}}{(k+1+ k_{0})^{2\delta r}}.
    $$
    We now verify the second case, $k\not\in I$, i.e.,
    $$
    {\tfrac{1}{2}}\, \tau \dist^{\nicefrac{1}{\delta}}(x_{k};\mathcal{X}^*) > \dfrac{\lambda}{(k+ k_{0})^{r}}.$$
    Substituting this into \cref{eq:nonsummable1} yields
    \begin{equation*}\label{nonsummable3}
         \dist^2(x_{k+1};\mathcal{X}^*) <\dist^2(x_{k};\mathcal{X}^*)- \dfrac{ {\tfrac{1}{2}}\lambda\tau}{(k+ k_{0})^{r}} \dist^{\nicefrac{1}{\delta}}(x_{k};\mathcal{X}^*).
    \end{equation*}
    Let us consider the convex function $\func{\varphi}{\R}{\R}$ given by
    $$\varphi(t) := t - \dfrac{ {\tfrac{1}{2}}\lambda\tau}{(k+ k_{0})^{r}}  t^{\tfrac{1}{2\delta}},$$
    defined over $[0,A^{2}(k+ k_{0})^{-2\delta r}]$ where the function achieves its maximum at $0$ or $A^{2}(k+ k_{0})^{-2\delta r}$.
    Using the inductive assumption and \cref{eq:nonsummable1}, we deduce
    $$
        \dist^2(x_{k+1};\mathcal{X}^*) 
        < \varphi\big(\dist^{2}(x_{k};\mathcal{X}^*)\big)
        \leq \max \left\{\varphi(0), \varphi\Big(A^{2}(k+ k_{0})^{-2\delta r}\Big)\right\}
        =
        \dfrac{A^{2}}{(k+ k_{0})^{2\delta r}}- \dfrac{{\tfrac{1}{2}} \lambda\tau A^{\nicefrac{1}{\delta}}} {(k+ k_{0})^{2r}}.
    $$
    To ensure the result, it is enough to verify 
    \begin{equation}\label{eq:nonsummable4}
        \dfrac{A^{2}}{(k+ k_{0})^{2\delta r}}- \dfrac{ \tfrac{1}{2}\lambda\tau A^{\nicefrac{1}{\delta}}} {(k+ k_{0})^{2r}} \leq \dfrac{A^{2}}{(k+1+ k_{0})^{2\delta r}}.
    \end{equation}
    Using the convexity of the function $\dfrac{A^{2}}{x^{2\delta r}}$ on positive real number set, we get
    $$
    \dfrac{A^{2}}{(k+1+ k_{0})^{2\delta r}} \geq \dfrac{A^{2}}{(k+ k_{0})^{2\delta r}} - \dfrac{2\delta r A^{2}}{(k+ k_{0})^{2\delta r+1}}.
    $$
    Moreover,
    $
    \tfrac{2\delta r A^{2}}{(k+ k_{0})^{2\delta r+1}} \leq
    \tfrac{\tfrac{1}{2}\lambda\tau A^{\nicefrac{1}{\delta}}} {(k+ k_{0})^{2r}}
    $
    if and only if
    $
    \big(\tfrac{4\delta r A^{2-\tfrac{1}{\delta}}}{\lambda \tau}\big)^{\tfrac{1}{1-2r(1-\delta)}} \leq
    k+ k_{0}.
    $
    Thus, \cref{eq:nonsummable4} holds as 
    \[ \big(\tfrac{4r\delta A^{\tfrac{2\delta-1}{\delta}}}{\lambda\tau}\big)^{\tfrac{1}{1-2r(1-\delta)}}\leq  k_{0},\] 
    giving our desired results.
\end{proof}

\Cref{thm:nonsummablestep-size} implies that the subgradient methods projected with decreasing step size $\alpha_{k}=\lambda (k+ k_{0})^{-r}$ when $0<r<1$ have a
convergence rate of \(O(K^{-r\delta})\).
However, for the case where $\tfrac{1}{1+\nu} <\delta < 1$ we can  achieve the convergence rate of
\(O\Big(K^{\tfrac{-\delta}{2(1-\delta)}}\Big)\).
Our next result establishes the sublinear convergence of the PSM with diminishing step-size $\alpha_{k}=\lambda (k+ k_{0})^{-r}$ with $r>1$.

\begin{thm}[{\bf Convergence rate of diminishing method}]\label{thm:decayingstep-size}
    Let $\tfrac{1}{1+\nu}<\delta<1$, $r:=\tfrac{1}{2(1-\delta)}$,
    $A:=(\tfrac{8\delta r}{\tau^{2}})^{\delta r}$, and $\lambda:= {\tfrac{1}{2}} \tau A^{\nicefrac{1}{\delta}}$
    with $\tau=\tfrac{\mu}{L}$, let $\{x_{k}\}_{k\geq 0}$ be generated by \Cref{alg:subgradient} with $\gamma=\nicefrac{1}{2}$ and step-size $\alpha_{k}=\lambda (k+ k_{0})^{-r}$, and let the initial point $x_{0} \in \mathcal{T}_\gamma$ satisfying $\dist(x_{0};\mathcal{X}^*)\leq A  k_{0}^{-\delta r},$
    with $ k_{0} =\max\left\{4\delta r, A^{\tfrac{1}{\delta r}}(\tfrac{2\rho}{\mu})^{\tfrac{2(1-\delta)}{\delta(1+\nu)-1}}\right\}$.
    Then, 
    \begin{equation}\label{eq:decayingconvergence}
        \dist(x_{k} ; \mathcal{X}^*) \leq \dfrac{A}{(k+ k_{0})^{\delta r}}.
    \end{equation}
\end{thm}

\begin{proof}
    By induction, we show the inequality \cref{eq:decayingconvergence}. The result for $k=0$ holds.
    We now assume that the inequality \cref{eq:decayingconvergence} holds for iteration $k$. We will show that it remains true for $k+1$.
    First, we observe that 
    $$
    \dist(x_{k} ; \mathcal{X}^*) \leq \frac{A}{(k+ k_{0})^{\delta r}} \leq \frac{A}{ k_{0}^{\delta r}} \leq \frac{A}{A(\tfrac{2\rho}{\mu})^{\tfrac{2(1-\delta)\delta r}{\delta(1+\nu)-1}}} = (\tfrac{\mu}{2\rho})^{\tfrac{\delta}{\delta(1+\nu)-1}},
    $$
    i.e., $x_{k}\in \mathcal{T}_\gamma$. Thanks to \cref{lembasicrecurrence}, it holds that
    \begin{equation}\label{eq:decaying3}
        \dist^2(x_{k+1};\mathcal{X}^*) \leq \dist^2(x_{k};\mathcal{X}^*)- \lambda (k+ k_{0})^{-r} \tau  \dist^{\nicefrac{1}{\delta}}(x_{k};\mathcal{X}^*)  + \lambda^{2} (k+ k_{0})^{-2r}.
    \end{equation}
    We consider the convex function $\varphi(t) := t - \lambda (k+ k_{0})^{-r} \tau  t^{\tfrac{1}{2\delta}}   + \lambda^{2} (k+ k_{0})^{-2r}$
    on $[0,A^{2}(k+ k_{0})^{-2\delta r}]$, leading to
    $$
    \begin{aligned}
        \dist^2(x_{k+1};\mathcal{X}^*) &\leq \varphi(\dist^{2}(x_{k};\mathcal{X}^*))\leq \max \left\{\varphi(0), \varphi(A^{2}(k+ k_{0})^{-2\delta r})\right\}\\[2mm]
        &=\max \left\{ \lambda^{2} (k+ k_{0})^{-2r},~
        A^{2}(k+ k_{0})^{-2\delta r}- \lambda \tau  A^{\nicefrac{1}{\delta}}(k+ k_{0})^{-2r}  + \lambda^{2} (k+ k_{0})^{-2r}\right\},
    \end{aligned}
    $$
    where \cref{eq:decaying3} was applied in the first inequality.
    To establish the desired result, it suffices to show
    \begin{align}
        &\hspace{1cm}\dfrac{\lambda^{2}}{(k+ k_{0})^{2r}}  \leq \dfrac{A^{2}}{(k+1+ k_{0})^{2\delta r}},\label{eq:decaying4}\\
        &\dfrac{A^{2}}{(k+ k_{0})^{2\delta r}}+ \dfrac{\lambda^{2} - \lambda \tau  A^{\nicefrac{1}{\delta}}}{(k+ k_{0})^{2r}} \leq \dfrac{A^{2}}{(k+1+ k_{0})^{2\delta r}}.\label{eq:decaying5}
    \end{align}
    Let us begin with verifying \cref{eq:decaying4}.
    Inasmuch as the function $\frac{x^{2r}}{(1+x)^{2\delta r}}$ is increasing on $[0,\infty)$,
    \begin{equation}\label{eq:decaying7}
        \dfrac{ k_{0}^{2r}}{(1+ k_{0})^{2\delta r}} \leq \dfrac{(k+ k_{0})^{2r}}{(k+1+ k_{0})^{2\delta r}}.
    \end{equation}
    Using the convexity of the function $\frac{1}{x^{2\delta r}}$ on positive real numbers, we get 
    \begin{equation}\label{eq:decaying8}
        \dfrac{1}{(1+ k_{0})^{2\delta r}} \geq \dfrac{1}{ k_{0}^{2\delta r}} - \dfrac{2\delta r}{ k_{0}^{2\delta r+1}}=\dfrac{1}{ k_{0}^{2\delta r}} - \dfrac{2\delta r}{ k_{0}^{2r}},
    \end{equation}
    with $2\delta r+1=2r$. Combining \cref{eq:decaying7} with \cref{eq:decaying8}, it follows that
    $$
    \frac{(k+ k_{0})^{2r}}{(k+1+ k_{0})^{2\delta r}} \geq \frac{ k_{0}^{2r}}{(1+ k_{0})^{2\delta r}} \geq
    \frac{ k_{0}^{2r}}{ k_{0}^{2\delta r}} - \frac{2\delta r k_{0}^{2r}}{ k_{0}^{2r}}
    = k_{0} - 2\delta r \geq 4\delta r - 2\delta r = 2\delta r = \frac{\lambda^{2}}{A^{2}},
    $$
    ensuring \cref{eq:decaying4}.
    Next, we justify the inequality \cref{eq:decaying5}.
    From the convexity of $\dfrac{A^{2}}{x^{2\delta r}}$, we obtain
    $$
    \frac{A^{2}}{(k+1+ k_{0})^{2\delta r}} - \frac{A^{2}}{(k+ k_{0})^{2\delta r}} \geq \frac{-2\delta rA^{2}}{(k+ k_{0})^{2\delta r+1}} = \frac{-2\delta rA^{2}}{(k+ k_{0})^{2r}}.
    $$
    By combining this inequality with $\lambda^{2} - \lambda\tau A^{\nicefrac{1}{\delta}}=- 2\delta r A^{2}$, which follows from the definition of $A$, $\lambda$, and $r$, the inequality \cref{eq:decaying5} is verified.
\end{proof}

We next focus on a Geometrically Decaying (GD) step-size to ensure linear convergence to the optimal solution set.
In this case, the step-size forms a decreasing geometric sequence that tends to zero, i.e., $\alpha_{k}=\lambda q^{k}$ where $\lambda >0$ and $0<q<1$.
\cref{thm:geodecayingstep-size} establishes the linear convergence of the PSM with GD step-size.

\begin{thm}[{\bf Convergence rate of GD method}]\label{thm:geodecayingstep-size}
    Let $\delta= 1$, $\tau=\tfrac{\mu}{L}\neq 1$ and
    $$
    \max\Big\{\tfrac{5\tau^{2}-4}{2\tau^{2}},0\Big\} < \beta < {\tfrac{1}{2}},\quad 0<\lambda< \tfrac{\tau}{2} \left(\tfrac{\mu}{2\rho}\right)^{\nicefrac{1}{\nu}},
    \quad 
    q:=\sqrt{1-\tfrac{(1-2\beta) \tau^{2}}{4}}, \quad
    A:=\max\left\{\tfrac{2\lambda}{\tau},\dist(x_0;\mathcal{X}^*)\right\}.
    $$
    Let $\{x_{k}\}_{k\geq 0}$ be generated by \Cref{alg:subgradient} with $\gamma=\nicefrac{1}{2}$ and step-size $\alpha_{k}=\lambda q^{k}$ and let the initial point $x_{0} \in \mathcal{T}_\gamma$ satisfying
    \begin{equation}\label{eq:geodecaying1}
        \dist(x_0;\mathcal{X}^*)\leq\tfrac{2\lambda}{\tau -\sqrt{\tau^{2}-4(1-q^{2})}}.
    \end{equation}
    Then, 
    \begin{equation}\label{eq:geodecayingconvergence}
        \dist(x_{k} ; \mathcal{X}^*) \leq A q^{k}.
    \end{equation}
\end{thm}

\begin{proof}
    By \cref{lem:tau}, it follows that $\tau\in (0,1)$.
    Consequently,
    $\tfrac{5\tau^{2}-4}{2\tau^{2}}< {\tfrac{1}{2}}$, which ensures the existence of $\beta$ within the specified range. Additionally, it is evident that $0 < q < 1$. Observing that \(1-q^{2}<\tfrac{ \tau^{2}}{4}\), we deduce $\tau^{2}-4(1-q^{2}) >0$, i.e., the inequality \cref{eq:geodecaying1} is well-defined.
    We proceed by induction to verify the inequality~\cref{eq:geodecayingconvergence}.
    For $k=0$, we note that $\dist(x_{0}; \mathcal{X}^*) \leq A= Aq^{0}$. 
    Assume that \cref{eq:geodecayingconvergence} holds for $i=0,1,\ldots,k$, and we show that it holds for $k+1$.
    Using the upper bound of $\lambda$ and the assumption $x_0\in \mathcal{T}_\gamma$, we get
    $$
    \dist(x_{k} ; \mathcal{X}^*) \leq A q^{k}\leq A =\max\left\{\tfrac{2\lambda}{\tau},\dist(x_0;\mathcal{X}^*)\right\}\leq \left(\tfrac{\mu}{2\rho}\right)^{\nicefrac{1}{\nu}}.
    $$
    Hence, $x_{k}\in \mathcal{T}_\gamma$, i.e., applying \cref{lembasicrecurrence} to the case $\delta=1$ yields
    $$
    \dist^2(x_{k+1};\mathcal{X}^*) \leq \dist^2(x_{k};\mathcal{X}^*)- \lambda q^{k} \tau  \dist(x_{k};\mathcal{X}^*)  + \lambda^{2} q^{2k}.
    $$
    Let us define the convex function $\varphi(t) := t^{2} - \lambda q^{k} \tau  t   + \lambda^{2} q^{2k}$ on $[0,Aq^{k}]$, which attains its maximum at $0$ or $A q^{k}$.
    By the inductive assumption, $\dist(x_{k};\mathcal{X}^*) \in [0,A q^{k}]$, i.e., 
    $$
        \dist^2(x_{k+1};\mathcal{X}^*) \leq \varphi \big(\dist(x_{k};\mathcal{X}^*)\big) \leq \max \left\{\varphi(0), \varphi(A q^{k})\right\}
         = q^{2k} \max \left\{ \lambda^{2} ,~
        A^{2} - \lambda \tau A   + \lambda^{2}\right\}.
    $$
    To ensure $\dist^2(x_{k+1};\mathcal{X}^*)\leq A^{2} q^{2(k+1)}$, it is sufficient to verify
    \begin{equation}\label{eq:geodecaying3}
        \lambda^{2}\leq A^{2}q^{2}, \quad
        A^{2}-  \lambda \tau A   + \lambda^{2} \leq A^{2}q^{2}.
    \end{equation}
    For the first inequality, since $A^2\geq \tfrac{4\lambda^2}{\tau^2}$ and  $\tau^2 \leq q^2$, it follows that $ \lambda^{2}\leq A^{2}q^{2}$, where $\tau^{2} \leq q^{2}$ arises from the lower bound of $\beta$.
    To show the second inequality in \cref{eq:geodecaying3}, we consider the quadratic equation
    \begin{equation}\label{eq:geodecayingquadratic}
        \left(1  -q^{2}\right)Z^{2} - \lambda \tau Z  + \lambda^{2} =0,
    \end{equation}
    where its positive roots, $z_{1} \leq z_{2}$ are given by
    $$z_{1,2}:=\tfrac{\lambda\tau \pm \lambda\sqrt{ \tau^{2} -  4(1-q^{2})}}{2(1  -q^{2})}=
    \tfrac{2\lambda}{\tau \mp \sqrt{\tau^{2}-4(1-q^{2})}}.$$
    To satisfy the second inequality in \cref{eq:geodecaying3}, it is sufficient to verify that $z_{1}\leq A\leq z_{2}$.
    Regarding the smaller root, it is clear that
    $z_{1} \leq \tfrac{2\lambda}{\tau}\leq A$.
    For the larger root, we consider two cases: (i) $A=\dist(x_{0};\mathcal{X}^*)$; (ii) $A=\tfrac{2\lambda}{\tau}$. In Case~(i), $A=\dist(x_{0};\mathcal{X}^*)$, the upper bound in \cref{eq:geodecaying1} and $0<q<1$ ensures the result. In Case~(ii),  $A=\tfrac{2\lambda}{\tau}$, it is clear that
    $A\leq \tfrac{2\lambda}{\tau - \sqrt{\tau^{2}-4(1-q^{2})}}=z_{2}.$
    Therefore, the inductive step is complete, adjusting our desired result.
\end{proof}

\subsection{{\bf PSM with Scaled Polyak's step-size}}\label{sec:Polyak subgradient method2}

In this section, we concentrate on the PSM with the Scaled Polyak's step-size, $\alpha_k = \tfrac{f(x_k)-f^*}{\sigma\Vert\zeta_k\Vert}$ where $\sigma>0$.
Unlike the PSMs discussed in the previous sections, this approach requires knowing the optimal value of the problem \cref{prb2}. This requirement is reasonable in certain scenarios, such as when employing the exact penalty method to solve nonlinear equations. We conclude this section by analyzing the convergence of the sequence generated by \Cref{alg:subgradient} with a scaled version of Polyak's step-size. We will show that the sequence generated by this method attains a sublinear convergence rate if $\tfrac{1}{1+\nu}<\delta< 1$, and enjoys the $Q$-linear convergence if $\delta= 1$.

\begin{thm}[{\bf Convergence rate of Scaled Polyak's method}]\label{thm:polyak}
Fix the parameters $\sigma>\tfrac{1}{2}$ and $0<\gamma<{\tfrac{2\sigma-1}{2\sigma}}$. Let $\{x_{k}\}_{k\geq 0}$ be generated by \Cref{alg:subgradient} with Scaled Polyak's step-size.
Then, the following assertions hold:
\begin{enumerate}[label=(\alph*)]
    \item\label{thm:polyak:p1}
        For each integer $k\ge 0$, $x_k\in\mathcal{T}_\gamma$ and 
        \begin{equation}\label{eq:polyak1}
            \dist^{2}(x_{k+1};\mathcal{X}^*) \leq \bigg(1 -\Big(2\sigma(1-\gamma)-1\Big)\sigma^{-2}\tau_{k}^{2} \bigg) \dist^{2}(x_{k};\mathcal{X}^*),
        \end{equation}
        where $\tau_{k}:=\tau\dist^{\tfrac{1}{\delta}-1}(x_{k};\mathcal{X}^*)$ and $\tau = \tfrac{\mu}{L}$;
    
    \item\label{thm:polyak:p2}
        The sequence $\{\dist(x_k; \mathcal{X}^*)\}_{k\ge 0}$ either converges to zero in a finite number of steps or converges to zero monotonically and
        $\lim_{k\to\infty} f(x_k) =f^{*}$;
    
    \item\label{thm:polyak:p3}
        If $\delta= 1$, the sequence $\{\dist(x_k; \mathcal{X}^*)\}_{k\ge 0}$ Q-linearly converges with rate $\sqrt{1-\Big(2\sigma(1-\gamma)-1\Big)\sigma^{-2}\tau^2}$;
        
    \item\label{thm:polyak:p4}
        If $\tfrac{1}{1+\nu}<\delta<1$, the sequence $\{\dist(x_k; \mathcal{X}^*)\}_{k\ge 0}$ sublinearly converges, i.e., there exists $\eta>0$ such that
        $$\dist^{2}(x_k;\mathcal{X}^*)\leq \eta k^{-\tfrac{\delta}{1-\delta}},$$
        for all $k\in \mathbb{N}$ sufficiently large;
        
    \item\label{thm:polyak:p5}
        For each $k\in \mathbb{N}$,
        \begin{equation}\label{eq:polyak:pv}
            f^{*}_{k}-f^{*} \leq \frac{\sigma L\dist(x_{0};\mathcal{X}^*)}{\sqrt{2\sigma(1-\gamma)-1}\sqrt{k+1}}.
        \end{equation}
\end{enumerate}
\end{thm}

\begin{proof}
\ref{thm:polyak:p1} By assumption it holds that $x_{0}\in \mathcal{T}_\gamma$. 
Using induction, assume that $x_{k}\in \mathcal{T}_\gamma$ for some $k\ge 0$. We prove that $x_{k+1}\in \mathcal{T}_\gamma$ for $k=k+1$.
Two cases are recognized: (i) $\zeta_{k}=0$; (ii) $\zeta_{k}\neq 0$.
In Case~(i), $\zeta_k=0$, it clearly $x_{k+1}=x_{k}\in \mathcal{T}_\gamma$.
Moreover, $0=\zeta_k\in \partial f(x_{k})$ and $x_{k}$ is a stationary point for problem \cref{prb2}.
Hence, $x_{k} \in \mathcal{X}^*$ by \cref{lem-dis2}.
Thus, $\dist(x_{k+1};\mathcal{X}^*)=\dist(x_{k};\mathcal{X}^*)=0$, and
\cref{eq:polyak1} is satisfied.
In Case~(ii), $\zeta_{k}\neq 0$, let us further consider two subcases: either $x_{k}\in \mathcal{X}^*$ or $x_{k}\not\in \mathcal{X}^*$.
In the case $x_{k}\in \mathcal{X}^*$, it follows that $f(x_{k})=f^*$, $\alpha_k=0$, $\dist(x_{k};\mathcal{X}^*)=0$, and $\tau_{k}=0$. Then, we obtain $x_{k+1}=\operatorname{proj}_{X}{(x_{k})}=x_{k}\in \mathcal{T}_\gamma$, and also \cref{eq:polyak1} holds.

Now, let us examine the latter case $x_{k}\not\in \mathcal{X}^*$.
We choose $\bar x\in \operatorname{proj}_{\mathcal{X}^*}{(x_{k})}$ implying $\Vert x_{k} -\bar x\Vert = \dist(x_{k};\mathcal{X}^*)$ and $f(\bar{x})=f^{*}$.
Drawing on the nonexpansiveness of the projection operator and  \cref{thm:f-chara}, we have
\begin{align}
        \dist^2(x_{k+1};\mathcal{X}^*) & \leq \Vert x_{k+1} -\bar x\Vert^2  \leq \dist^2(x_{k};\mathcal{X}^*) + \tfrac{2\big(f(x_{k})-f(\bar x)\big)}{\sigma\Vert\zeta_{k}\Vert^2} \langle \zeta_{k}, \bar x -x_{k}\rangle + \tfrac{\big(f(x_{k})-f(\bar x)\big)^2}{\sigma^{2}\Vert\zeta_{k}\Vert^2}\nonumber\vspace{2mm}\\
         & \leq \dist^2(x_{k};\mathcal{X}^*) + \tfrac{2\big(f(x_{k})-f(\bar x)\big)}{\sigma\Vert\zeta_{k}\Vert^2} \Big(f(\bar x)-f(x_{k}) +\rho\Vert x_{k} -\bar x\Vert^{1+\nu}\Big)+ \tfrac{\big(f(x_{k})-f(\bar x)\big)^2}{\sigma^{2}\Vert\zeta_{k}\Vert^{2}}\nonumber\vspace{2mm}\\
         & = \dist^2(x_{k};\mathcal{X}^*) +\tfrac{f(x_{k})-f^{*}}{\sigma^{2}\Vert\zeta_{k}\Vert^2}
         \Big(2\sigma\rho\dist^{1+\nu}(x_{k};\mathcal{X}^*)-(2\sigma-1)(f(x_{k}) - f^{*})\Big).\label{eq:polyak2}
\end{align}
Since $x_{k}\in \mathcal{T}_\gamma\subseteq \mathcal{T}_{1}$, we get $\dist(x_{k};\mathcal{X}^*) < \left(\gamma\tfrac{\mu}{\rho}\right)^{\tfrac{\delta}{\delta(1+\nu)-1}}$ and $\Vert \zeta_{k}\Vert\leq L$.
Therefore,
\begin{align}\label{eq:polyak3}
    \rho\dist^{1+\nu}(x_{k};\mathcal{X}^*) < \gamma\mu\dist^{\tfrac{1}{\delta}}(x_{k};\mathcal{X}^*)\leq \gamma (f(x_{k}) - f^{*}),
\end{align}
where the second inequality follows from  the HEB. Applying these inequalities in \cref{eq:polyak2}, we deduce
\begin{equation*}
    \begin{array}{ll}
       \dist^2(x_{k+1};\mathcal{X}^*)
         &\leq \dist^2(x_{k};\mathcal{X}^*) +\tfrac{f(x_{k})-f^{*}}{\sigma^{2}\Vert\zeta_{k}\Vert^2}
         \Big(2\sigma\gamma\mu\dist^{\tfrac{1}{\delta}}(x_{k};\mathcal{X}^*) - (2\sigma-1)\mu\dist^{\tfrac{1}{\delta}}(x_{k};\mathcal{X}^*)\Big)\vspace{2mm}\\
         &\leq \dist^2(x_{k};\mathcal{X}^*) +\tfrac{\mu^{2}\dist^{\tfrac{2}{\delta}}(x_{k};\mathcal{X}^*)}{\sigma^{2}\Vert\zeta_{k}\Vert^2}
         \Big(2\sigma\gamma  - 2\sigma +1\Big)
         \leq \bigg(1 -\Big(2\sigma(1-\gamma)-1\Big)\sigma^{-2}\tau_{k}^{2} \bigg) \dist^2(x_{k};\mathcal{X}^*).
    \end{array}
\end{equation*}
Hence, the inequality \cref{eq:polyak1} is valid.
Additionally, inasmuch as $x_{k}\in \mathcal{T}_\gamma$ and $x_{k}\not\in \mathcal{X}^*$, we conclude $\tau_{k}\in (0,1]$ by \cref{lem:tau}, i.e., $0<1 - \Big(2\sigma(1-\gamma)-1\Big)\sigma^{-2}\tau_{k}^2 < 1$.
Hence, $\dist(x_{k+1};\mathcal{X}^*)< \dist(x_{k};\mathcal{X}^*)$, which yields
$x_{k+1} \in \mathcal{T}_\gamma$.\\
\ref{thm:polyak:p2} In order to prove the convergence of $\dist(x_k;\mathcal{X}^*)$, based on Assertion~\ref{thm:polyak:p1} either there exists some $k_{0}\in \mathbb{N}$ such that $\dist(x_k;\mathcal{X}^*) = 0$ for all $k\geq k_{0}$, or $\dist(x_k;\mathcal{X}^*)$ is a decreasing positive sequence.
We show that $\dist(x_k;\mathcal{X}^*) \to 0$ in the latter case. Suppose for contradiction that $\dist(x_k;\mathcal{X}^*)$ converges to some $d>0$.
Taking \cref{eq:polyak1} into account, we get
$d^2\leq \bigg(1-\Big(2\sigma(1-\gamma)-1\Big)\sigma^{-2}\tau^{2} d^{\tfrac{2}{\delta}-2}\bigg)~d^{2}$,
leading to a contradiction.
Hence, $\dist(x_k;\mathcal{X}^*)$ decreasingly converges to zero.
Moreover, the closedness of $\mathcal{X}^*$ implies that $x_{k}\to x^{*}$, for some $x^{*}\in \mathcal{X}^*$, i.e., it results in $f(x_{k})\to f(x^{*})=f^{*}$, due to continuity of $f$.\\ 
\ref{thm:polyak:p3} In this assertion, \cref{lem:tau} ensures that $\tau_{k}=\tau=\tfrac{\mu}{L} \in (0,1]$.
It follows from inequality \cref{eq:polyak1} that
$$\dist^{2}(x_{k+1};\mathcal{X}^*) \leq \bigg(1 - \Big(2\sigma(1-\gamma)-1\Big)\sigma^{-2}\tau^{2}\bigg) \dist^{2}(x_{k};\mathcal{X}^*),$$ where $0<1 -\Big(2\sigma(1-\gamma)-1\Big)\sigma^{-2}\tau^{2}<1$. Therefore, $\dist(x_k;\mathcal{X}^*)$ converges Q-linearly to zero with rate $\sqrt{1-\Big(2\sigma(1-\gamma)-1\Big)\sigma^{-2}\tau^2}$. Consequently, $x_k$ converges $\mathcal{X}^*$.\\
\ref{thm:polyak:p4}
The inequality \cref{eq:polyak1} can be rewritten as 
$$\dist^{\nicefrac{1}{\delta}}(x_k;\mathcal{X}^*) \leq \dfrac{\sigma^{2}}{\Big(2\sigma(1-\gamma)-1\Big)\tau^{2}}\Big(\dist(x_k;\mathcal{X}^*)-\dist(x_{k+1};\mathcal{X}^*)\Big).$$
Thus, the sublinear convergence of $\dist(x_k;\mathcal{X}^*)$ is guaranteed from
\cref{lem-sequence}.\\
\ref{thm:polyak:p5} Combining \cref{eq:polyak2} and \cref{eq:polyak3} with $\Vert \zeta_{k}\Vert\leq L$, it follows that
\begin{align*}
    \dist^2(x_{k+1};\mathcal{X}^*)
    &\leq \dist^2(x_{k};\mathcal{X}^*) +\tfrac{f(x_{k})-f^{*}}{\sigma^{2}\Vert\zeta_{k}\Vert^2}
    \Big(2\sigma\gamma(f(x_{k})-f^{*}) - (2\sigma-1)(f(x_{k})-f^{*})\Big)\\
    &\leq \dist^2(x_{k};\mathcal{X}^*) -\tfrac{(f(x_{k})-f^{*})^{2}}{\sigma^{2}L^2}
    \Big(2\sigma(1-\gamma)-1\Big)\leq \ldots\leq \dist^2(x_{0};\mathcal{X}^*) -\tfrac{2\sigma(1-\gamma)-1}{\sigma^{2}L^2} \sum_{i=0}^{k} (f(x_{i})-f^{*})^{2}\\
    &\leq \dist^2(x_{0};\mathcal{X}^*) -\tfrac{2\sigma(1-\gamma)-1}{\sigma^{2}L^2} (f^{*}_{k}-f^{*})^{2} (k+1).
\end{align*}
Since $\dist^2(x_{k+1};\mathcal{X}^*)\geq 0$, the inequality \cref{eq:polyak:pv} is obtained, adjusting our claims.
\end{proof}


\section{Application to robust matrix recovery}\label{sec:NumRes}
This section presents a detailed analysis of our numerical experiments to demonstrate the efficiency of various projected subgradient methods employing several step-size strategies, including the following:
\begin{itemize}
    \item[$\bullet$] Diminishing: \Cref{alg:subgradient} with the step-size $\alpha_k:=\frac{\alpha_0}{\sqrt{k+1}}$ and $\alpha_0>0$;
    \item[$\bullet$] Decaying: \Cref{alg:subgradient} with the step-size $\alpha_k:=\alpha_0 \lambda^k$  and $\alpha_0>0$;
    \item[$\bullet$] Polyak: \Cref{alg:subgradient} with the step-size $\alpha_k:= \tfrac{f(x_k)-f^*}{\Vert\zeta_k\Vert}$;
    \item[$\bullet$] Scaled Polyak: \Cref{alg:subgradient} with the step-size $\alpha_k:=\tfrac{f(x_k)-f^*}{4\Vert\zeta_k\Vert}$.
\end{itemize}

\subsection{{\bf Robust low-rank matrix recovery}}
The experiments are performed on various applications of {\it robust low-rank matrix recoveries} using different datasets. The applications include {\it robust matrix completion} on the MovieLens dataset, {\it image inpainting}, {\it robust nonnegative matrix factorization for face recognition}, and {\it matrix compression}. The basis of all these applications is the robust matrix recovery problem
\begin{equation} \label{eq:robustMC}
    \min_{U,V} \|X - UV\|_1, 
\end{equation}
where the matrix $X$ is approximated by $UV$. Here, $U \in \mathbb{R}^{m \times r}$ and $V \in \mathbb{R}^{r \times n}$ represent low-rank matrices \cite{li2020nonconvex}. This formulation results in a nonconvex and nonsmooth paraconvex optimization problem (as confirmed in \Cref{pro:para-comp}), which presents significant computational challenges. 
In our implementation, the code was entirely written in Python (a $3.11.7$ version) and tested on a MacBook Pro, macOS Ventura $13.6.6$, equipped with 32 GB of RAM.

\subsubsection{{\bf Robust matrix completion problem}}
In data mining and machine learning, many datasets are represented as matrices, often with missing values ranging from minimal to substantial. The problem of {\it robust matrix completion (RMC)} \cite{wang2012probabilistic} is formulated as 

\begin{equation}\label{eq:mc}
    \min_{U\in\R^{m\times r},V\in\R^{r\times n}}  \frac{1}{2} \|M\odot(X-UV)\|_1,
\end{equation}
where $\odot$ stands for the Hadamard product of two matrices. Missing values can be addressed by defining a binary weight matrix $M$ given by
\begin{align*}
    M_{ij}=\left\{
        \begin{array}{ll}
            1 &~~~~ \mathrm{if}~ X_{ij}~ \mathrm{is~known},\\
            0 &~~~~ \mathrm{if}~ X_{ij}~ \mathrm{is~unknown}.
        \end{array}
    \right.
\end{align*}
Recovering such matrices is challenging due to the sparsity of observed values, compounded by noise, corruption, and outliers.

We evaluated our method on the widely-used MovieLens (ML-100K) dataset, which contains approximately 100,000 ratings from 943 users for 1,682 movies (ratings range from 1 to 5). The dataset was split into an $80\%$ training set and a $20\%$ test set to evaluate reconstruction quality. The proposed algorithms are applied to the problem \cref{eq:mc} with the rank $r=20$ and a maximum of $1000$ iterations. For the decaying strategy, the parameters are set as $\gamma_k =0.95$ and $\alpha_0 =10^{-3}$,  while for the diminishing strategy, $\alpha_0=10^{-3}$. The reconstruction accuracy is measured by the {\it root mean squared error (RMSE)}, and the results are presented in \Cref{tab:rmse_mc}. The convergence behavior is analyzed based on the loss values, where the results are presented in \Cref{fig:loss_mc}.
To demonstrate the robustness of our algorithm, we employed two different initialization methods: one based on SVD (shown in the left subfigure of \Cref{fig:loss_mc}) and the other with a random starting point (shown in the right subfigure of \Cref{fig:loss_mc}).

\Cref{fig:loss_mc} illustrates that the Polyak and scaled Polyak strategies achieved lower loss values compared to the decaying and diminishing strategies under both initialization methods. However, \Cref{tab:rmse_mc} indicates the decaying strategy achieved the best performance with an RMSE of $1.153$, closely followed by Scaled Polyak with an RMSE of $1.254$. Both methods indicate fast convergence and accurate predictions. In contrast, the diminishing algorithm results in the highest RMSE of $1.656$. 

\begin{table}[htp!]
    \centering
    \begin{tabular}{|l|c|c|c|c|}
        \hline
        Step-size&Polyak& \text{Scaled Polyak} &Diminishing& Decaying \\ \hline
        svd initialization&1.460 & 1.254  &   1.655983 & 1.153\\ \hline
          random initialization    &1.350&  1.288&   1.085& 1.066\\\hline
    \end{tabular}
    \caption{RMSE values for the MovieLens 100k dataset.}
    \label{tab:rmse_mc}
\end{table}

\begin{figure}[!htbp]
\centering
\subfloat[Loss values with SVD starting point.]
{\includegraphics[width=7.5cm]{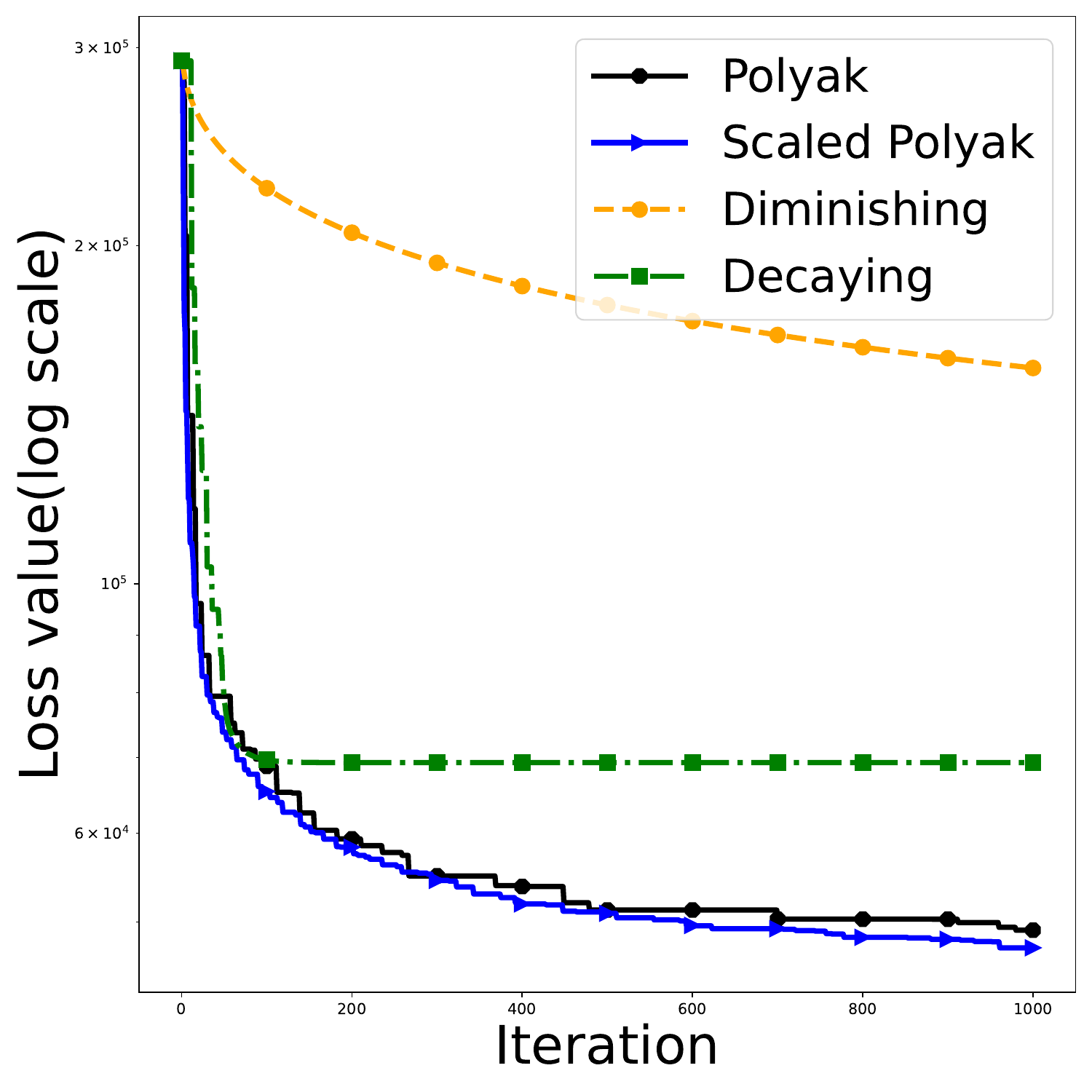}}
\hspace{0.2cm}
\subfloat[Loss values with random starting point.]
{\includegraphics[width=7.5cm]{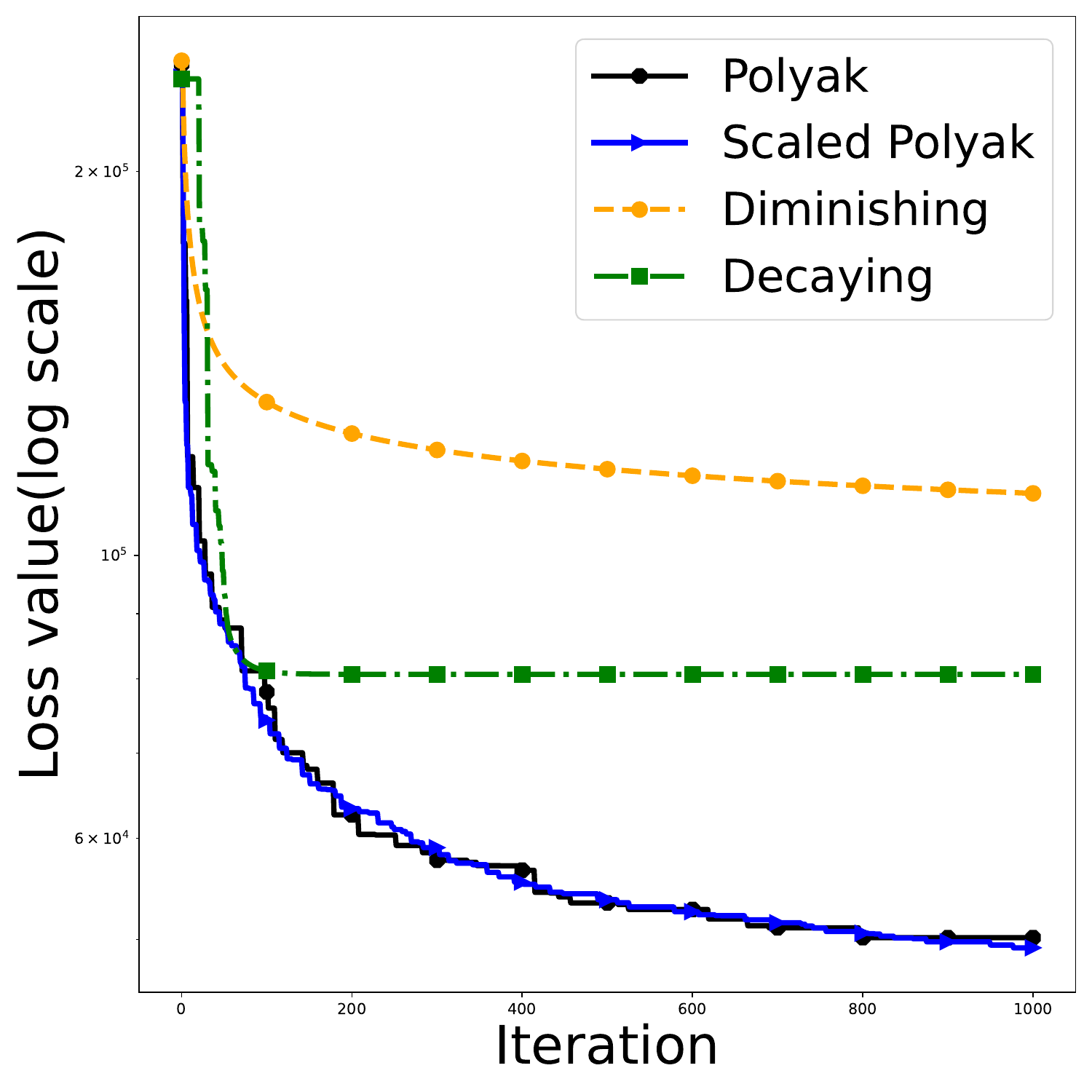}}
\caption{Convergence of the loss function (log scale) over 1000 iterations for the MovieLens 100K dataset, comparing four step-size strategies: Polyak, scaled Polyak, diminishing, and decaying using two initial points.}
\label{fig:loss_mc}
\end{figure}

\subsubsection{{\bf Image inpainting problem}}
Image inpainting aims to reconstruct missing or corrupted regions of an image in order to restore it as closely as possible to its original form; cf. \cite{guillemot2013image}. Here, we applied robust matrix completion (RMC) to images corrupted with 
 $40\%$ noise, where  $40\%$ of the pixel values were masked. 
The experimental setup uses the following parameters: $r=100$, $iterations=1000$, and $\alpha_0=10^{-2}$ for the diminishing step-size and $\alpha_0=10^{-3}$ for the decaying step-size. The quality of reconstruction is assessed using the {\it peak signal-to-noise ratio (PSNR)} and the {\it signal-to-noise ratio (SNR}), summarized in \Cref{tab:psnrimages4per}, and the loss function evolution in \Cref{fig:loss_image_4per}. The main images with their $40\%$ corrupted are plotted in \Cref{fig:inpaint_org}, and the reconstructed ones are plotted in \Cref{fig:inpaint_rec}.

\begin{table}[!ht]
\centering
\begin{tabular}{|l|cc|cc|cc|cc|}
\hline
\multirow{2}{*}{Image} & \multicolumn{2}{c|}{\textbf{Polyak}} & \multicolumn{2}{c|}{\textbf{Scaled Polyak}} & \multicolumn{2}{c|}{\textbf{Diminishing}} & \multicolumn{2}{c|}{\textbf{Decaying}} \\ \cline{2-9}
 & PSNR & SNR & PSNR & SNR & PSNR & SNR & PSNR & SNR \\ \hline
lighthouse & 20.86 & 7.39 & 21.25 & 7.78 & 17.64 & 4.17 & 15.06 & 1.60 \\
man        & 25.99 & 13.12 & 26.47 & 13.60 & 24.22 & 11.34 & 20.91 & 8.04 \\
kiel       & 23.04 & 8.11 & 21.32 & 6.40 & 18.05 & 3.13 & 15.10 & 0.18 \\
houses     & 21.76 & 9.77 & 20.96 & 8.97 & 16.50 & 4.51 & 14.05 & 2.06 \\
\hline
\end{tabular}
\caption{Quantitative evaluation of image reconstruction using the PSNR and SNR for four images (Lighthouse, Man, Kiel, and Houses), corrupted with $40\%$ noise. The results are reported for four step-size strategies: Polyak, Scaled Polyak, Diminishing, and Decaying. All values are rounded to two decimals.}\label{tab:psnrimages4per}
\end{table}

\begin{figure}[!htbp]
\centering
\subfloat[Original Man]{\includegraphics[width=4cm]{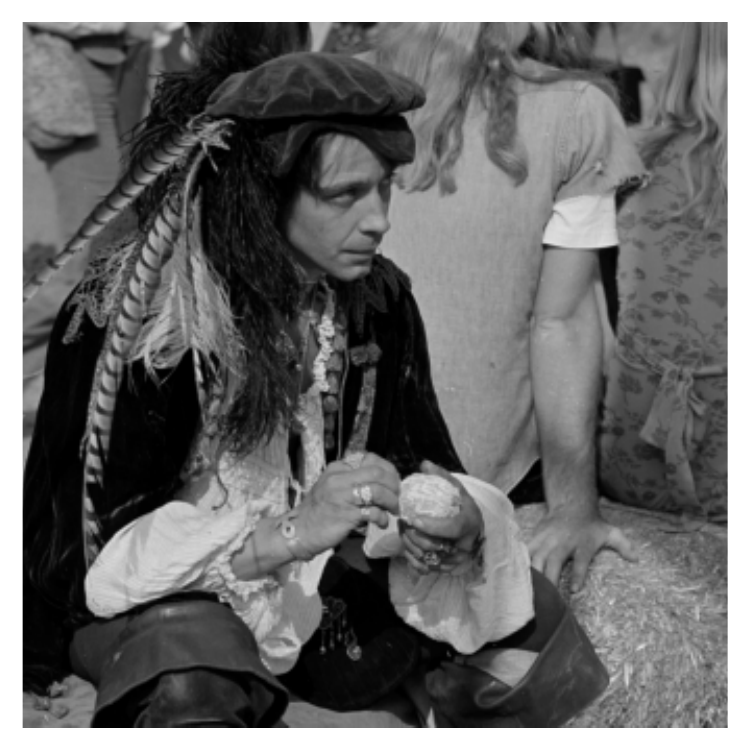}}%
\hspace{0.1cm}
\subfloat[Original Kiel]{\includegraphics[width=4cm]{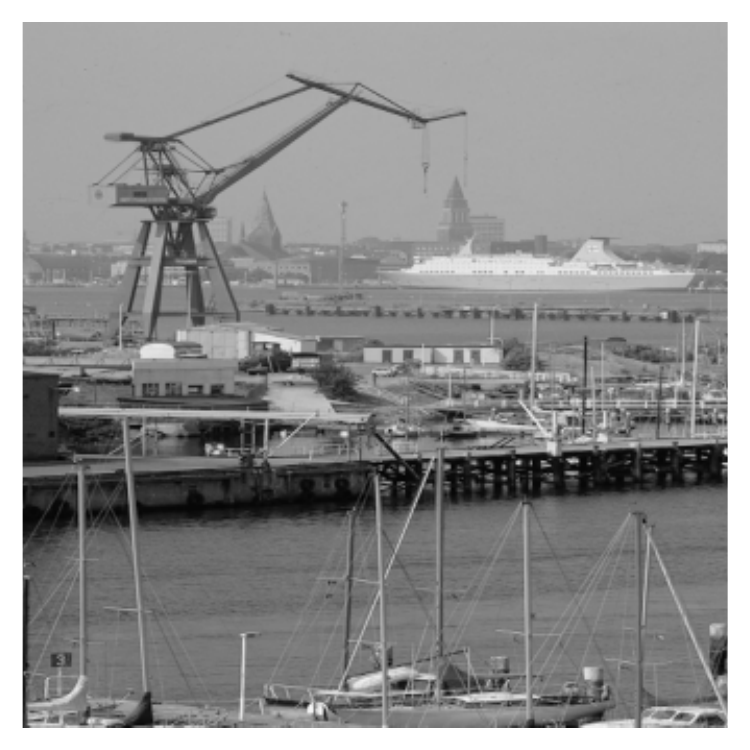}}%
\hspace{0.1cm}
\subfloat[Original Lighthouse]{\includegraphics[width=4cm]{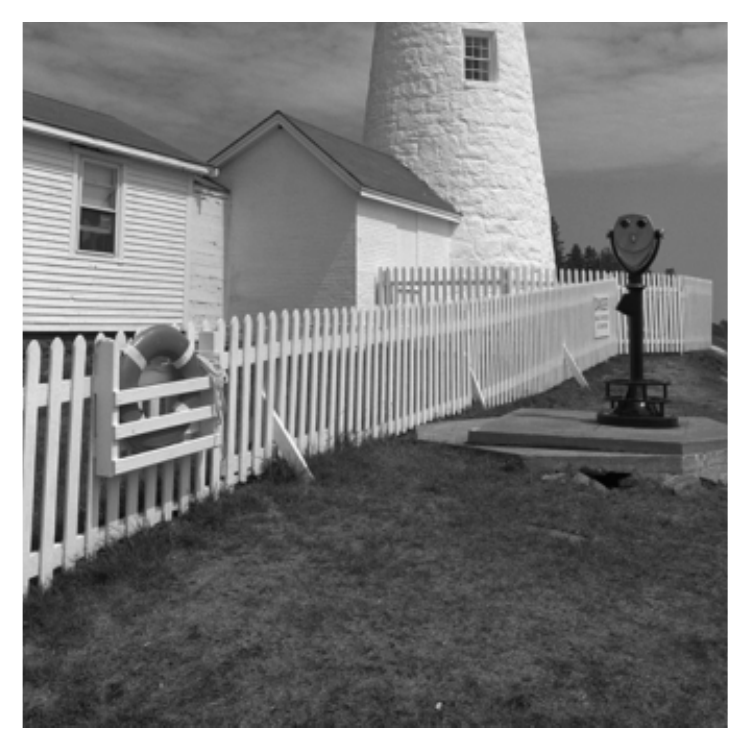}}%
\hspace{0.1cm}
\subfloat[Original Houses]{\includegraphics[width=4cm]{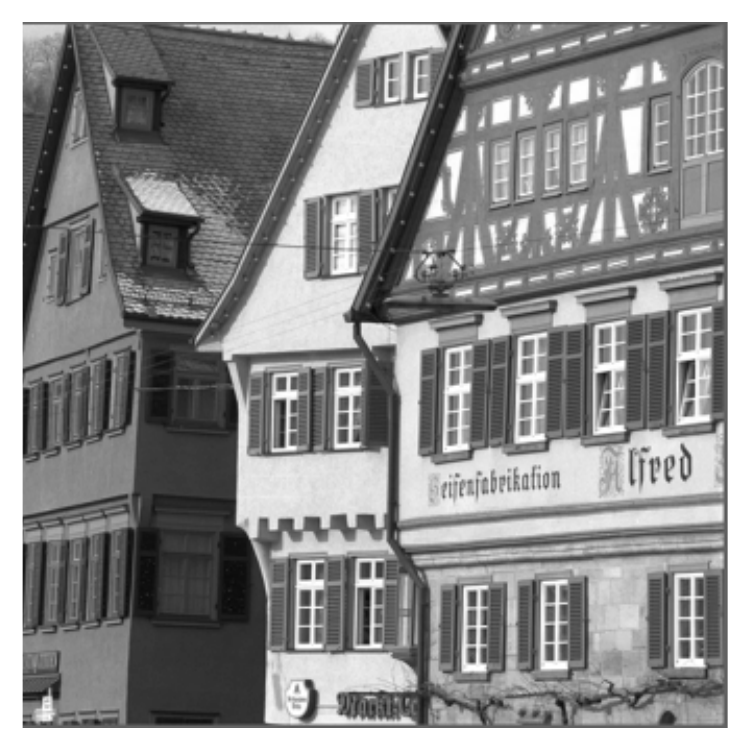}}%
\hspace{0.1cm}
\subfloat[Corrupted Man]{\includegraphics[width=4cm]{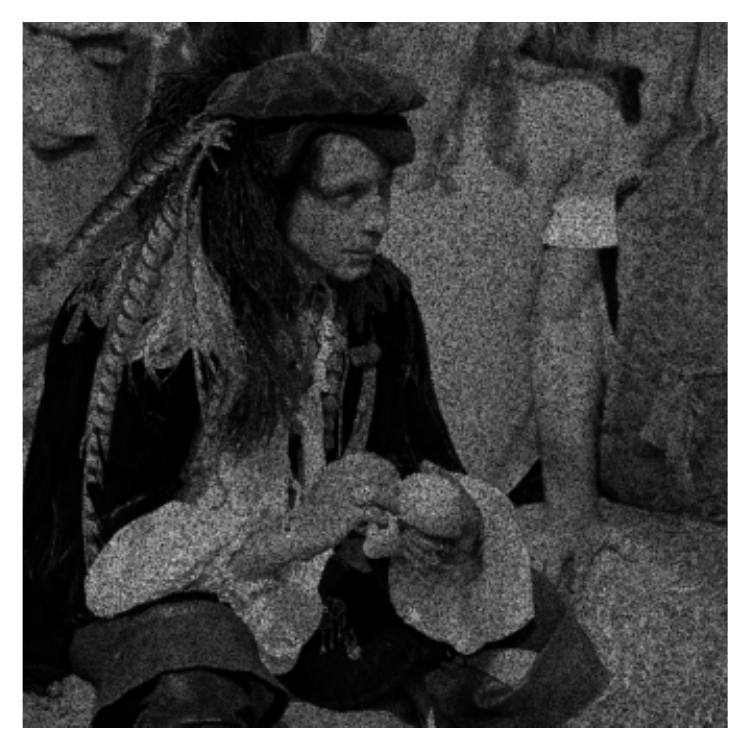}}%
\hspace{0.1cm}
\subfloat[Corrupted Kiel]{\includegraphics[width=4cm]{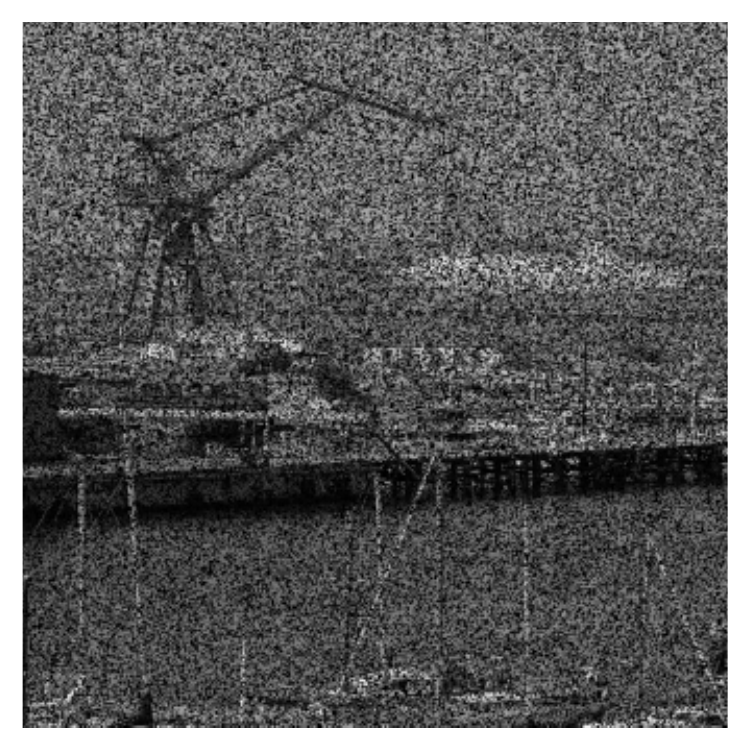}}%
\hspace{0.1cm}
\subfloat[Corrupted Lighthouse]{\includegraphics[width=4cm]{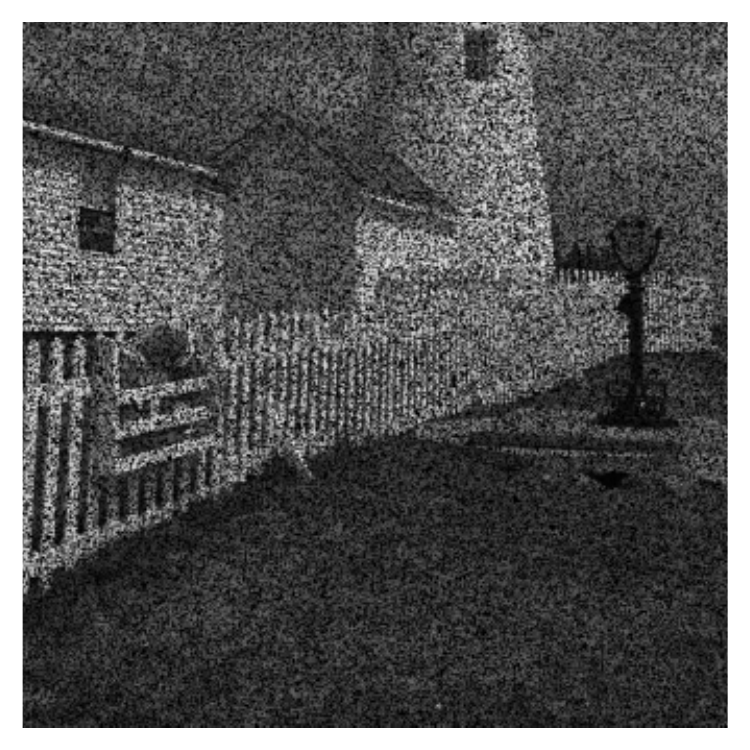}}%
\hspace{0.1cm}
\subfloat[Corrupted Houses]{\includegraphics[width=4cm]{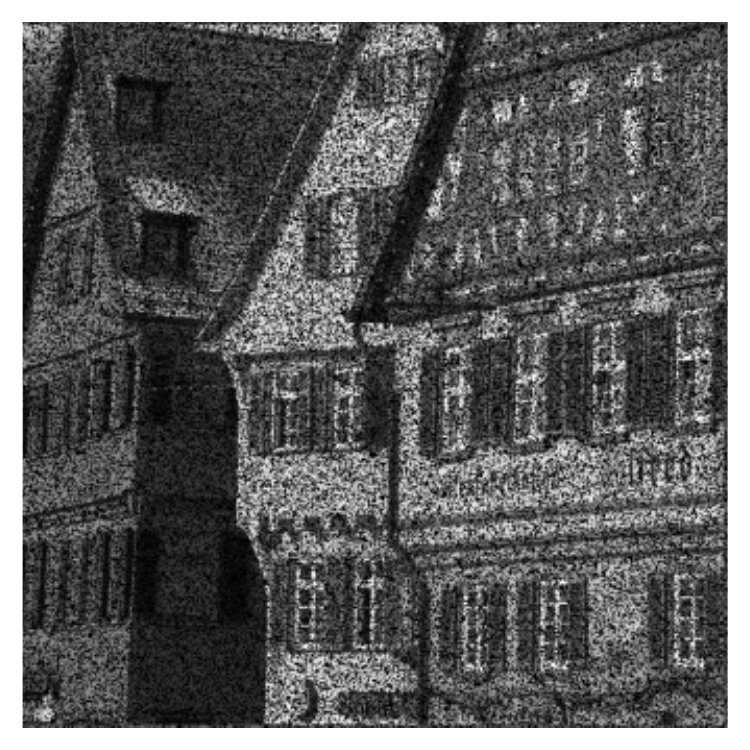}}%
\hspace{0.1cm}
\caption{The first row displays the original images: "Man" (a), "Kiel" (b), "Lighthouse" (c), and "Houses" (d). The second row shows the corresponding corrupted images (e–h) with $40\%$ random uniform noise applied.}
\label{fig:inpaint_org}
\end{figure}

The results in \Cref{fig:inpaint_rec,fig:loss_image_4per} and \Cref{tab:psnrimages4per} show that the Polyak and Scaled Polyak strategies consistently outperformed the diminishing and decaying methods across all metrics. Notably, the Scaled Polyak strategy achieved the highest PSNR and SNR values for most images, including the PSNR of 26.47 dB and the SNR of 13.60 for the 'Man' image, indicating superior reconstruction accuracy. The Polyak strategy also delivered strong results, with a PSNR of 25.99 dB and an SNR of 13.12 for the same image. In contrast, the diminishing and decaying methods demonstrated significantly poorer performance.
The convergence plots in \Cref{fig:loss_image_4per} highlight that Scaled Polyak and Polyak strategies provide the fastest and smoothest loss reductions, outperforming the diminishing and decaying methods. The decaying strategy showed minimal progress after an initial decrease, while the diminishing strategy exhibited slow and inconsistent convergence. The scaled Polyak strategy achieved the best convergence rates, which is closely followed by the Polyak strategy.

\begin{figure}[!htbp]
\captionsetup{justification=raggedright,singlelinecheck=false}
\centering
\subfloat[Scaled Polyak for Lighthouse]{\includegraphics[width=4.2cm]{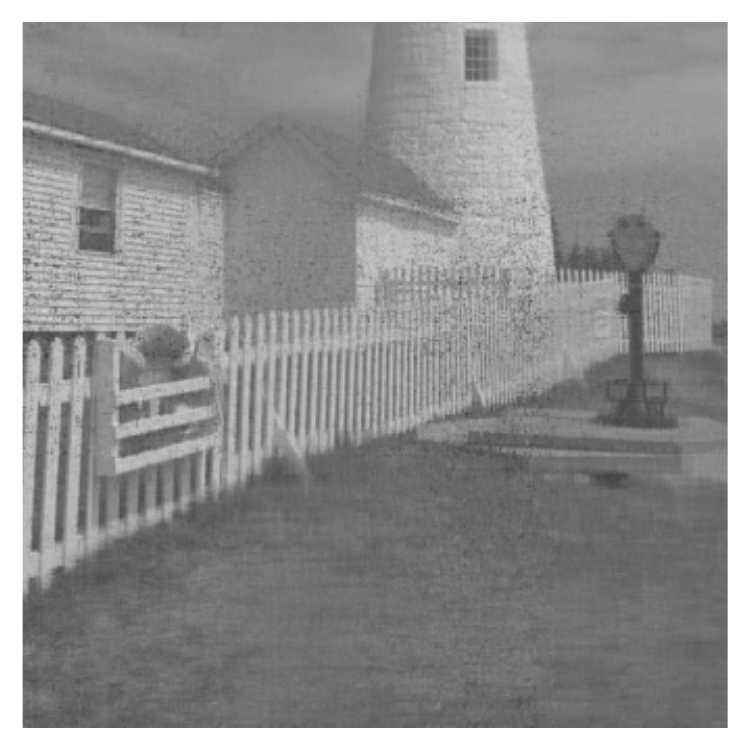}}%
\hspace{-0.1cm}
\subfloat[Polyak for Lighthouse]{\includegraphics[width=4.2cm]{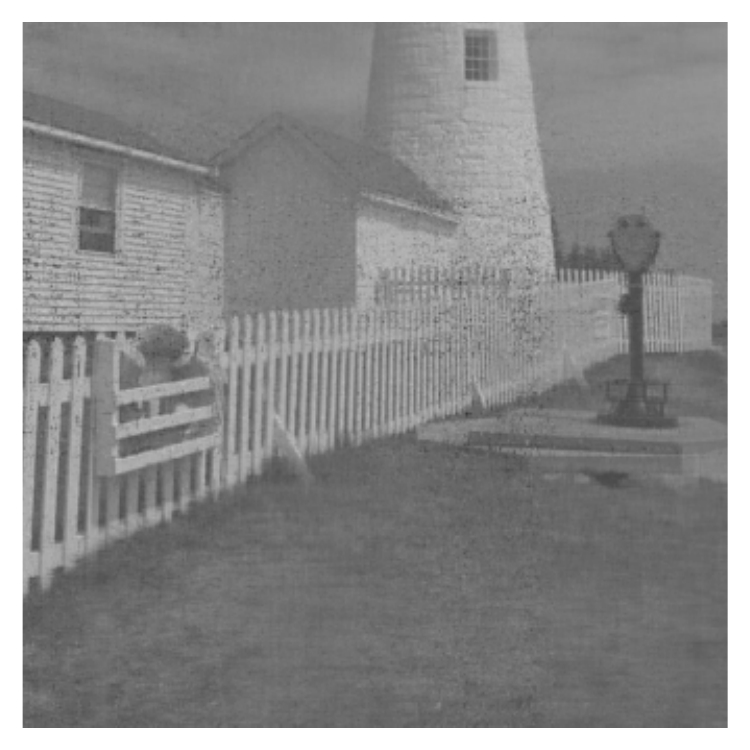}}%
\hspace{-0.1cm}
\subfloat[Decaying for Lighthouse]{\includegraphics[width=4.2cm]{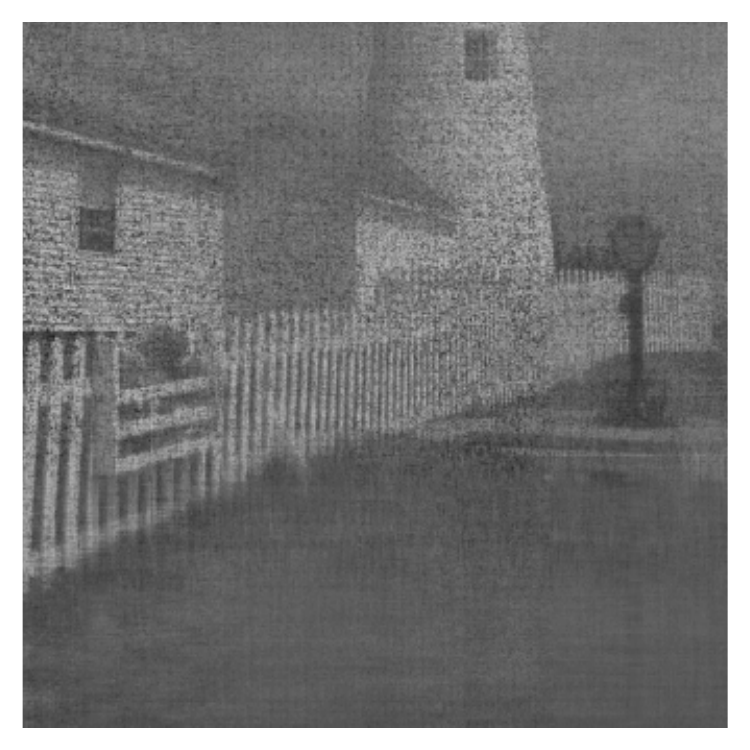}}%
\hspace{-0.1cm}
\subfloat[Diminishing for Lighthouse]{\includegraphics[width=4.2cm]{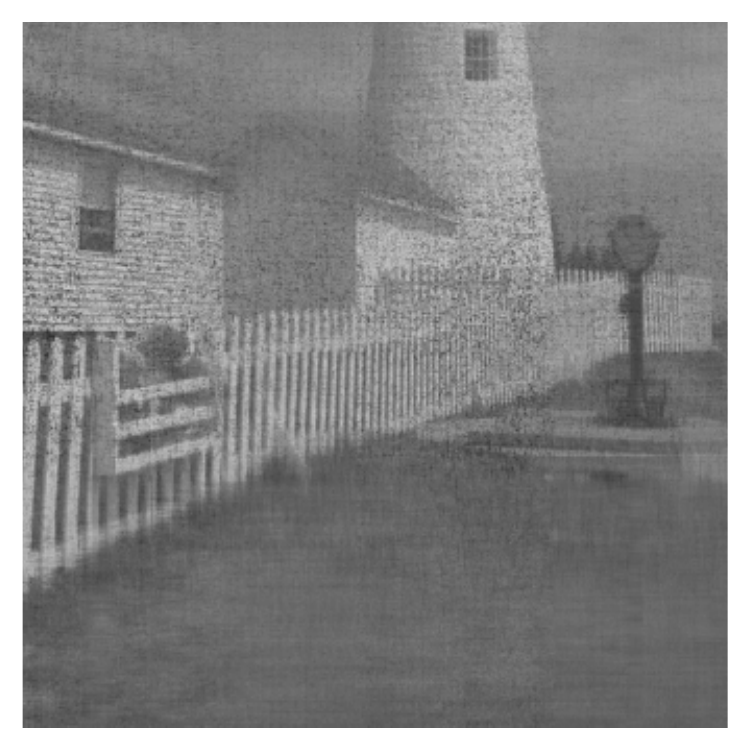}}%
\hspace{-0.1cm}
\subfloat[Scaled Polyak for Man.]{\includegraphics[width=4.2cm]{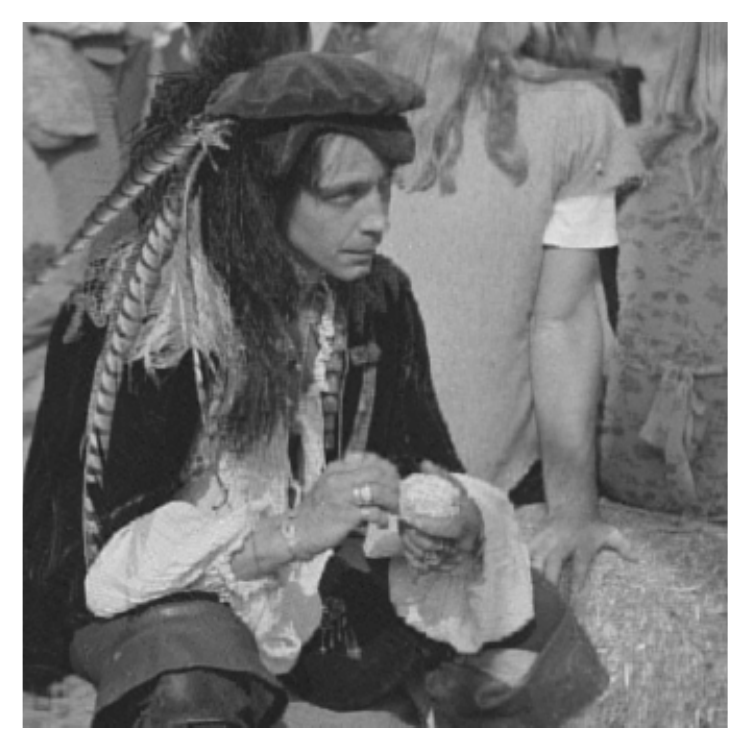}}%
\hspace{-0.1cm}
\subfloat[Polyak for Man]{\includegraphics[width=4.2cm]{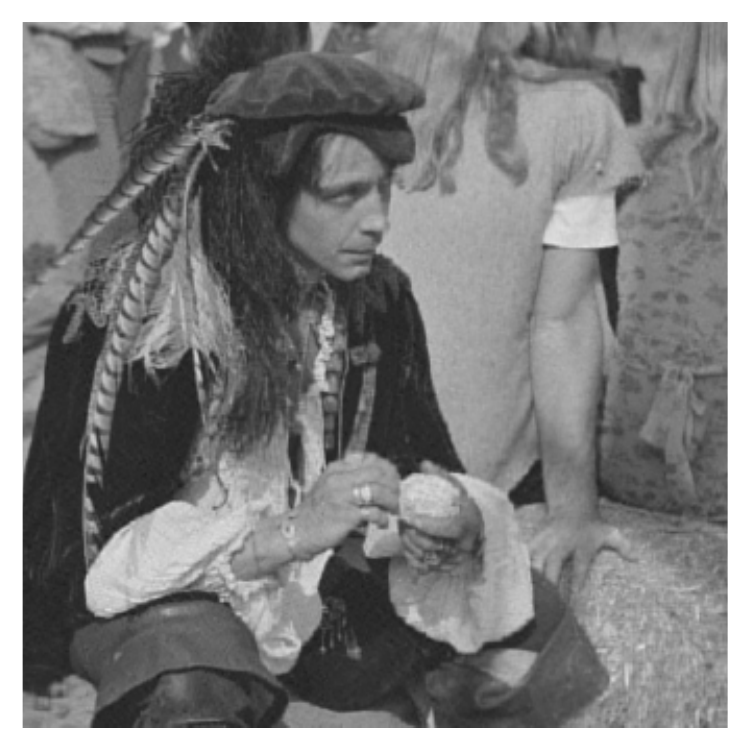}}%
\hspace{-0.1cm}
\subfloat[Decaying for Man]{\includegraphics[width=4.2cm]{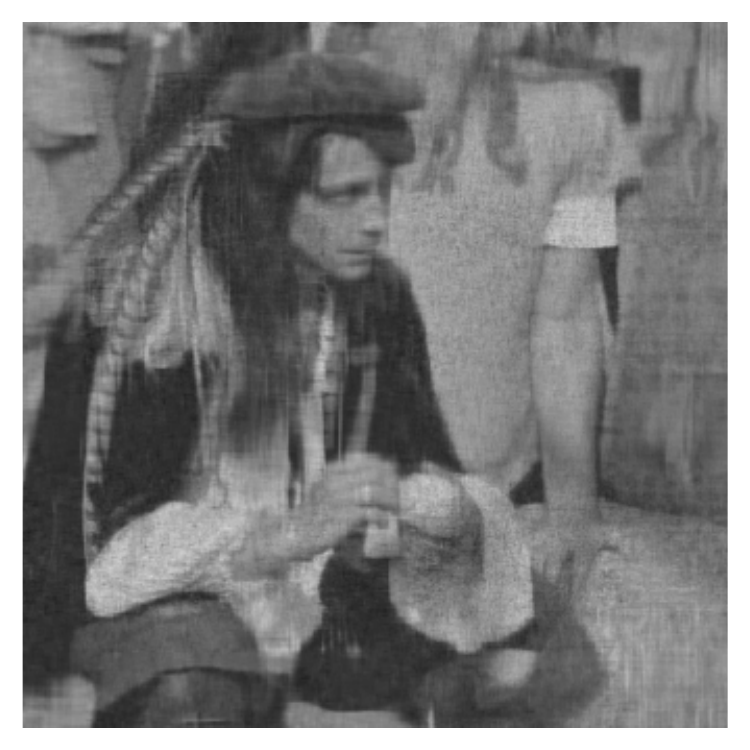}}%
\hspace{-0.1cm}
\subfloat[Diminishing for Man]{\includegraphics[width=4.2cm]{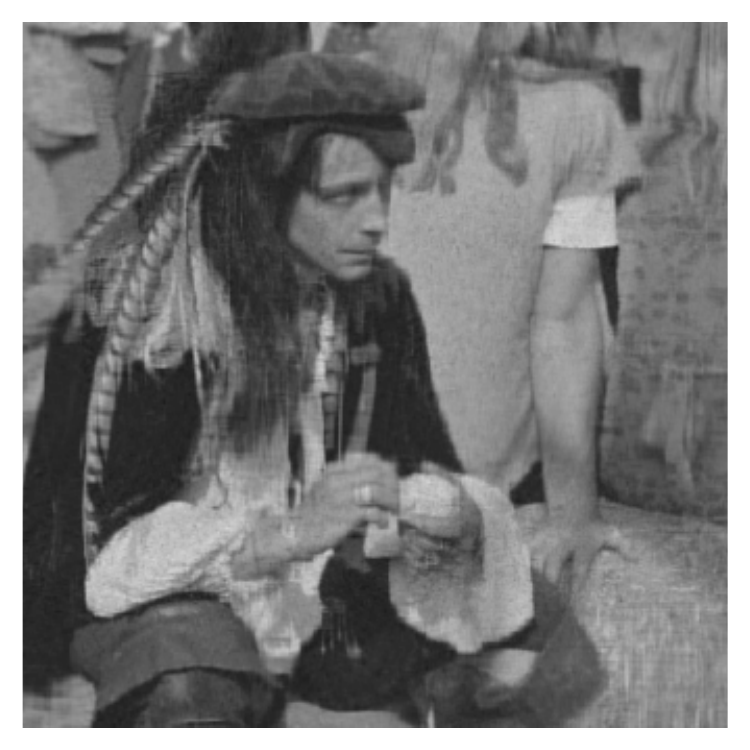}}%
\hspace{-0.1cm}
\subfloat[Scaled Polyak for Kiel.]{\includegraphics[width=4.2cm]{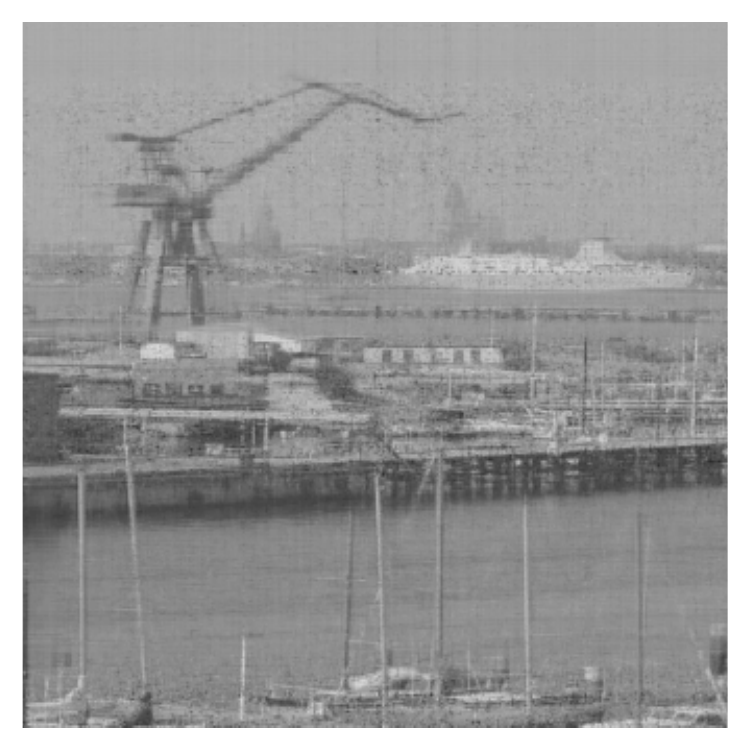}}%
\hspace{-0.1cm}
\subfloat[Polyak for Kiel]{\includegraphics[width=4.2cm]{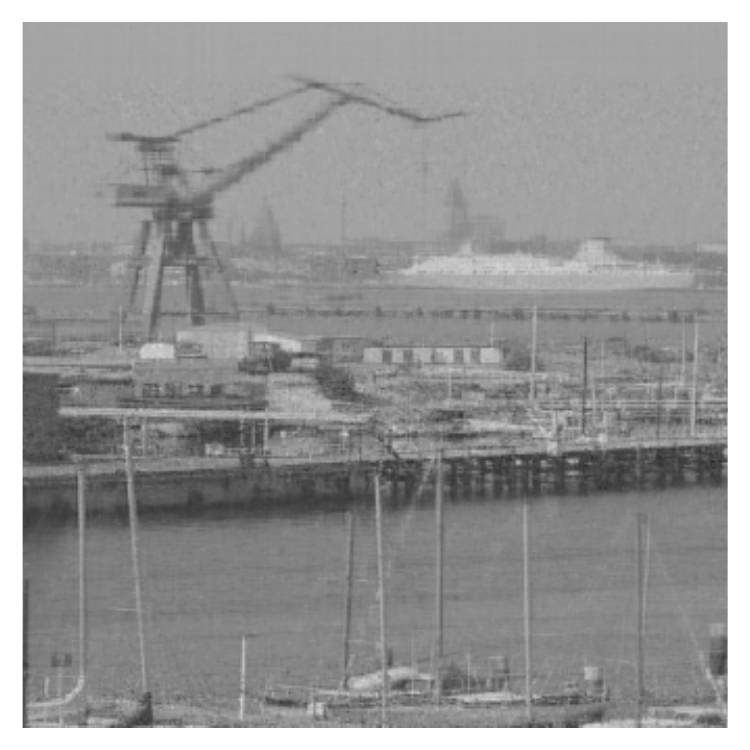}}%
\hspace{-0.1cm}
\subfloat[Decaying for Kiel]{\includegraphics[width=4.2cm]{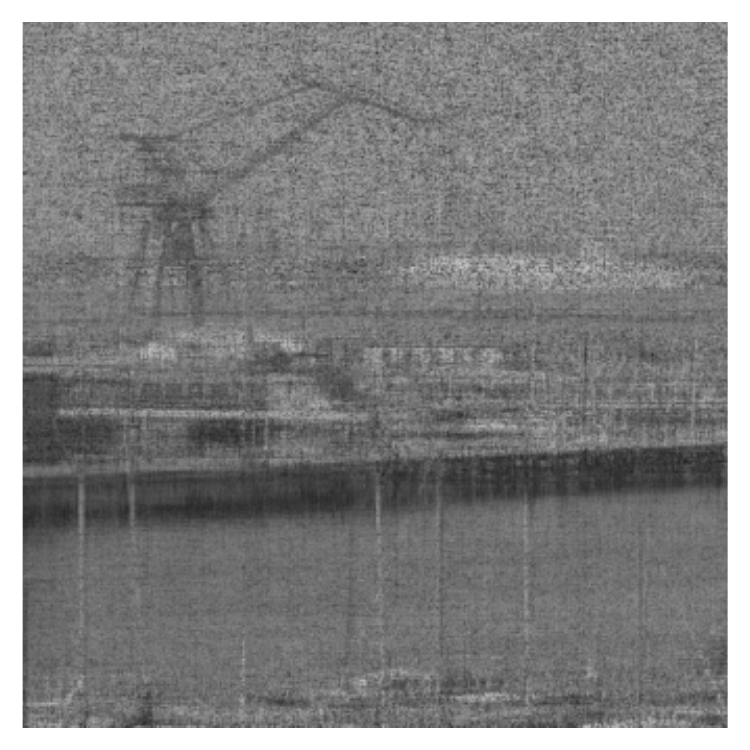}}%
\hspace{-0.1cm}
\subfloat[Diminishing for Kiel]{\includegraphics[width=4.2cm]{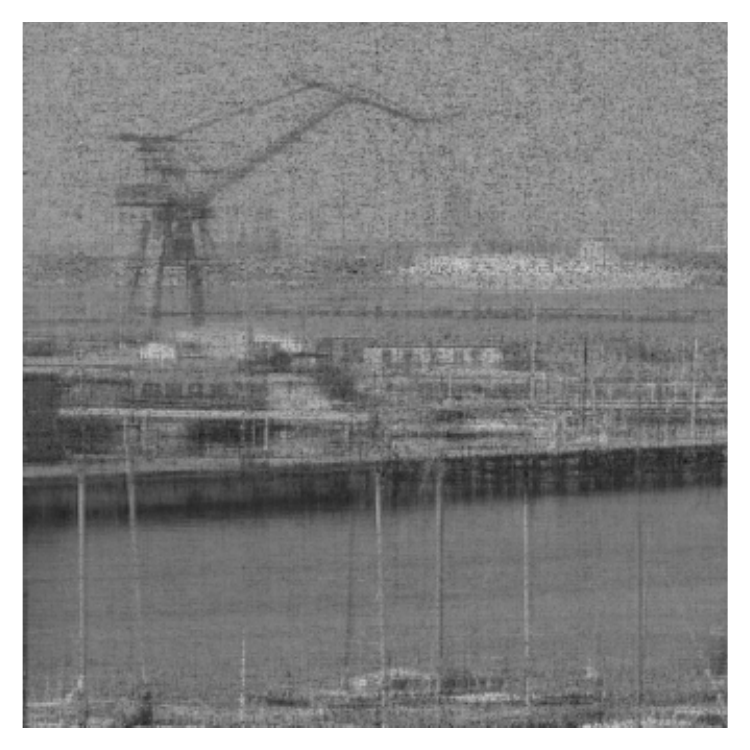}}%
\hspace{-0.1cm}
\subfloat[Scaled Polyak for Houses.]{\includegraphics[width=4.2cm]{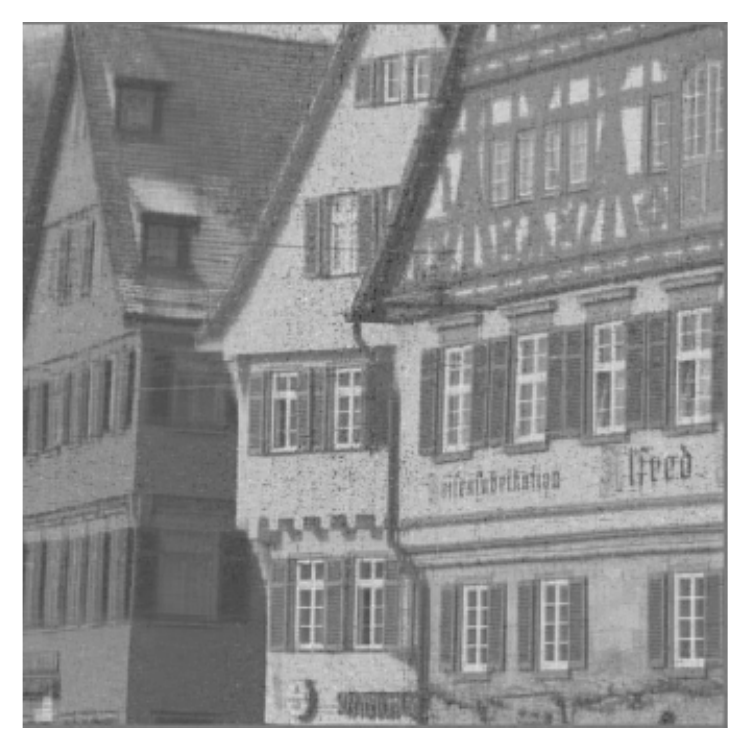}}%
\hspace{-0.1cm}
\subfloat[Polyak for Houses]{\includegraphics[width=4.2cm]{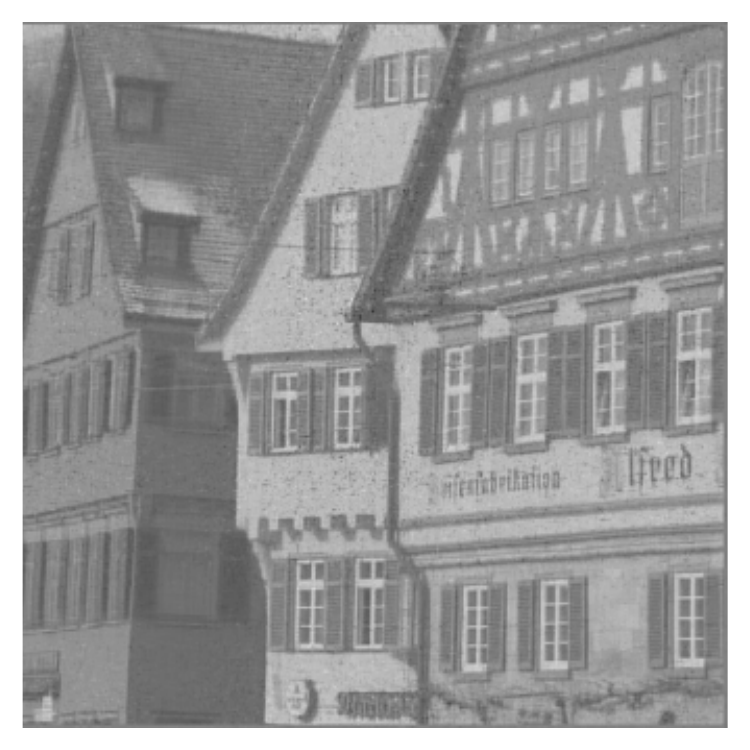}}%
\hspace{-0.1cm}
\subfloat[Decaying for Houses]{\includegraphics[width=4.2cm]{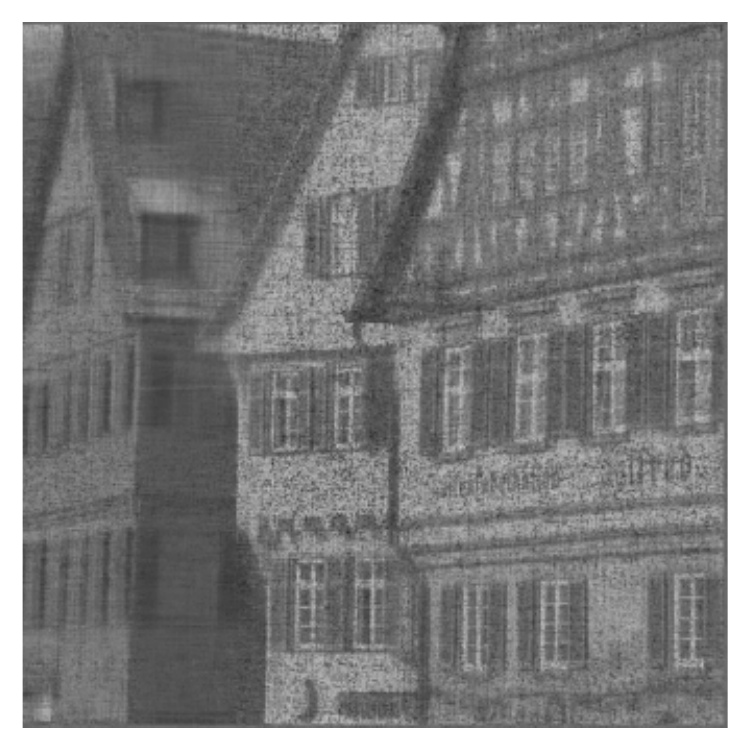}}%
\hspace{-0.1cm}
\subfloat[Diminishing for Houses]{\includegraphics[width=4.2cm]{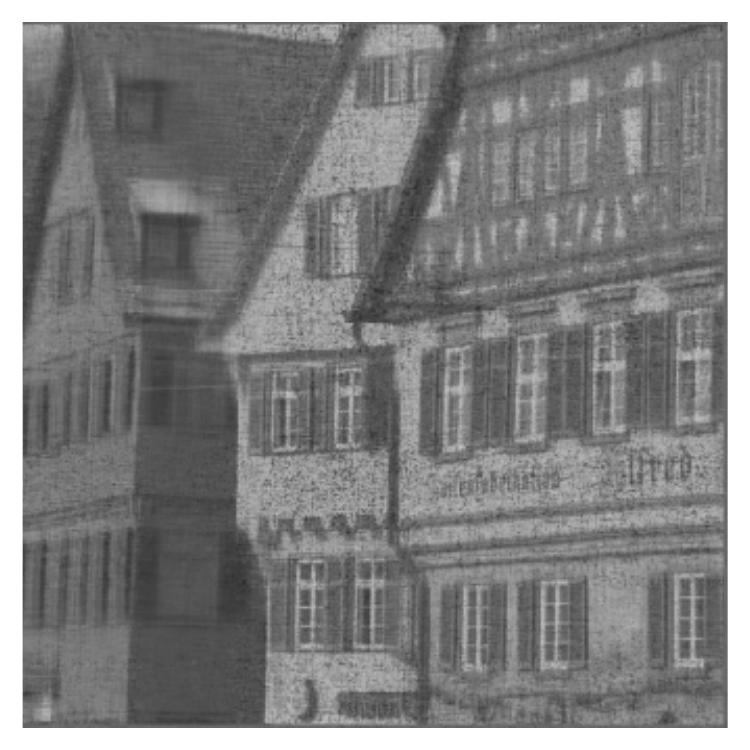}}%
\hspace{-0.1cm}
\caption{Qualitative comparison of image inpainting results for four corrupted images (lighthouse, man, kiel, and houses) using Polyak, scaled Polyak, diminishing, and decaying step-size strategies. Each column shows the reconstructions for the step-size method.}
\label{fig:inpaint_rec}
\end{figure}

\begin{figure}[htp!]
\centering
\subfloat[Loss values reconstruction for Lighthouse.]{\includegraphics[width=7.5cm]{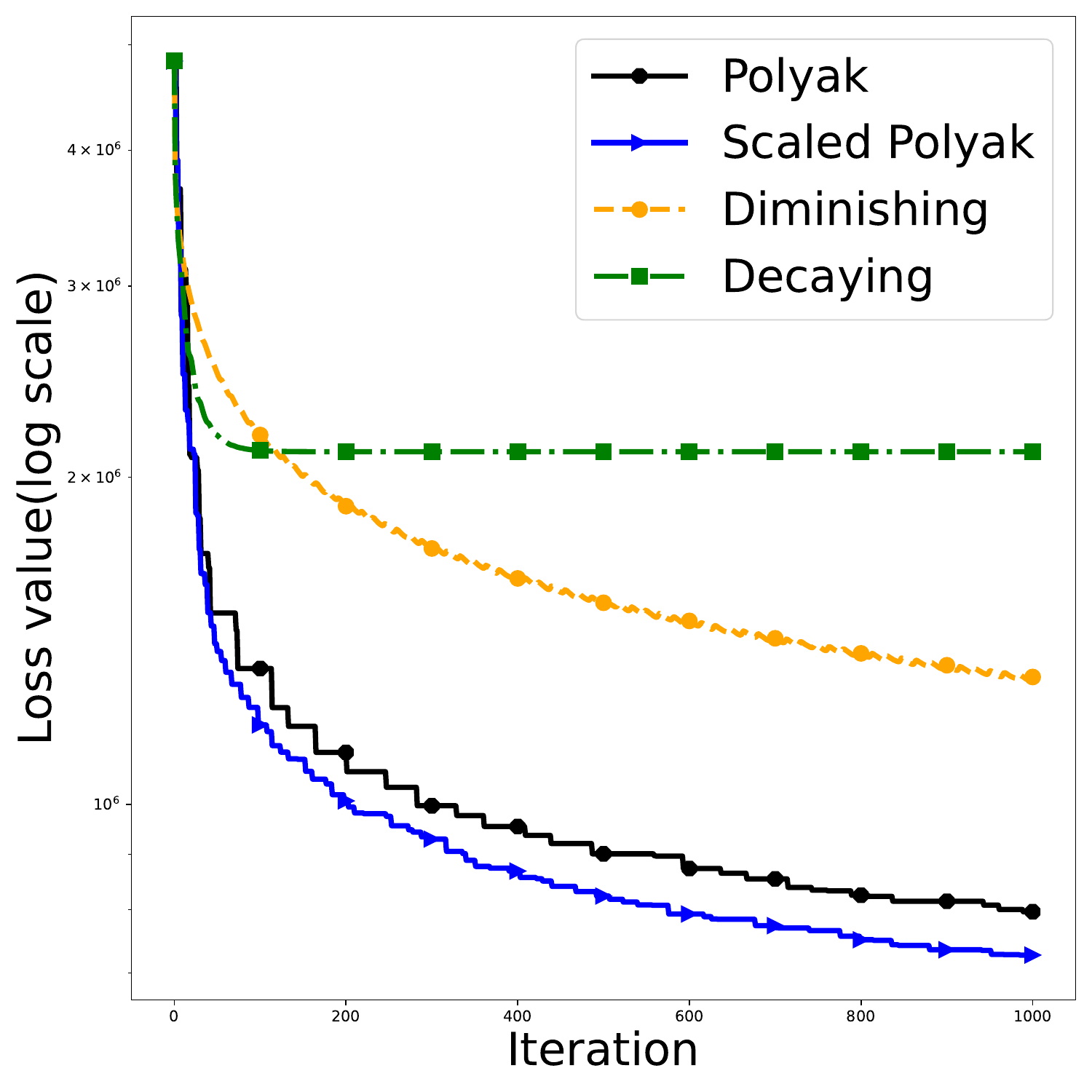}}%
\hspace{0.1cm}
\subfloat[Loss value reconstruction for Man.]{\includegraphics[width=7.5cm]{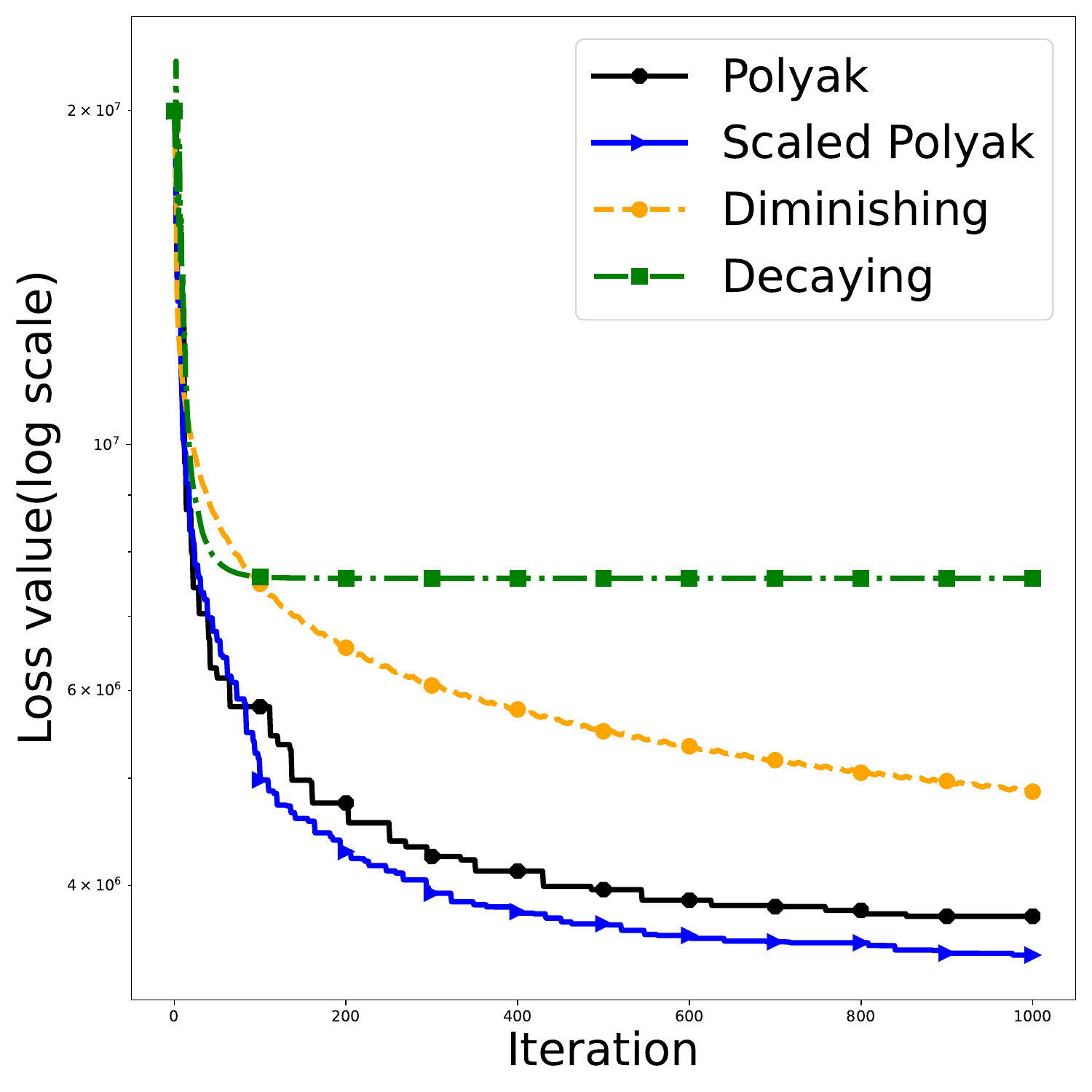}}%
\hspace{0.1cm}
\subfloat[Loss value reconstruction for Kiel.]{\includegraphics[width=7.5cm]{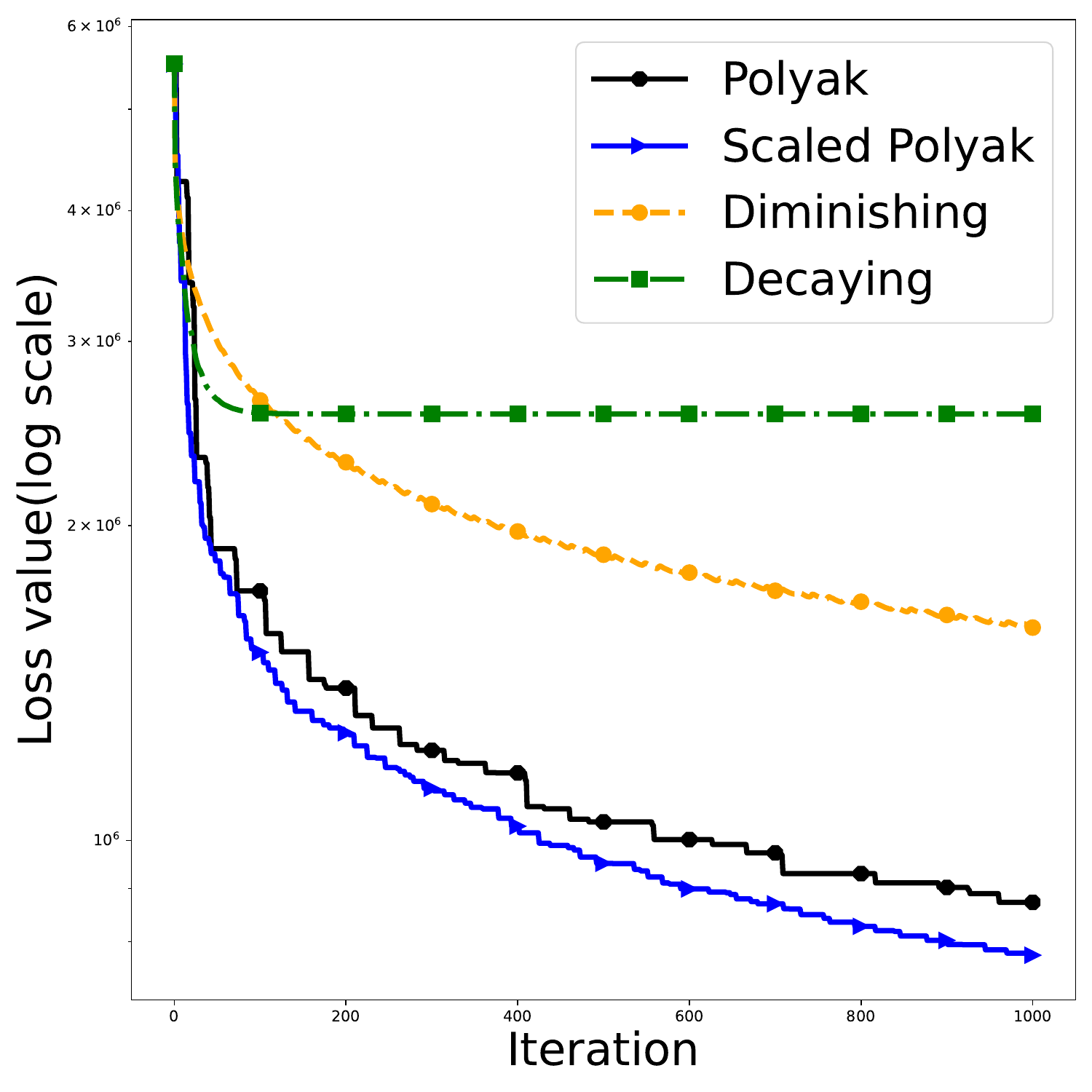}}%
\hspace{0.1cm}
\subfloat[Loss value reconstruction for House.]{\includegraphics[width=7.5cm]{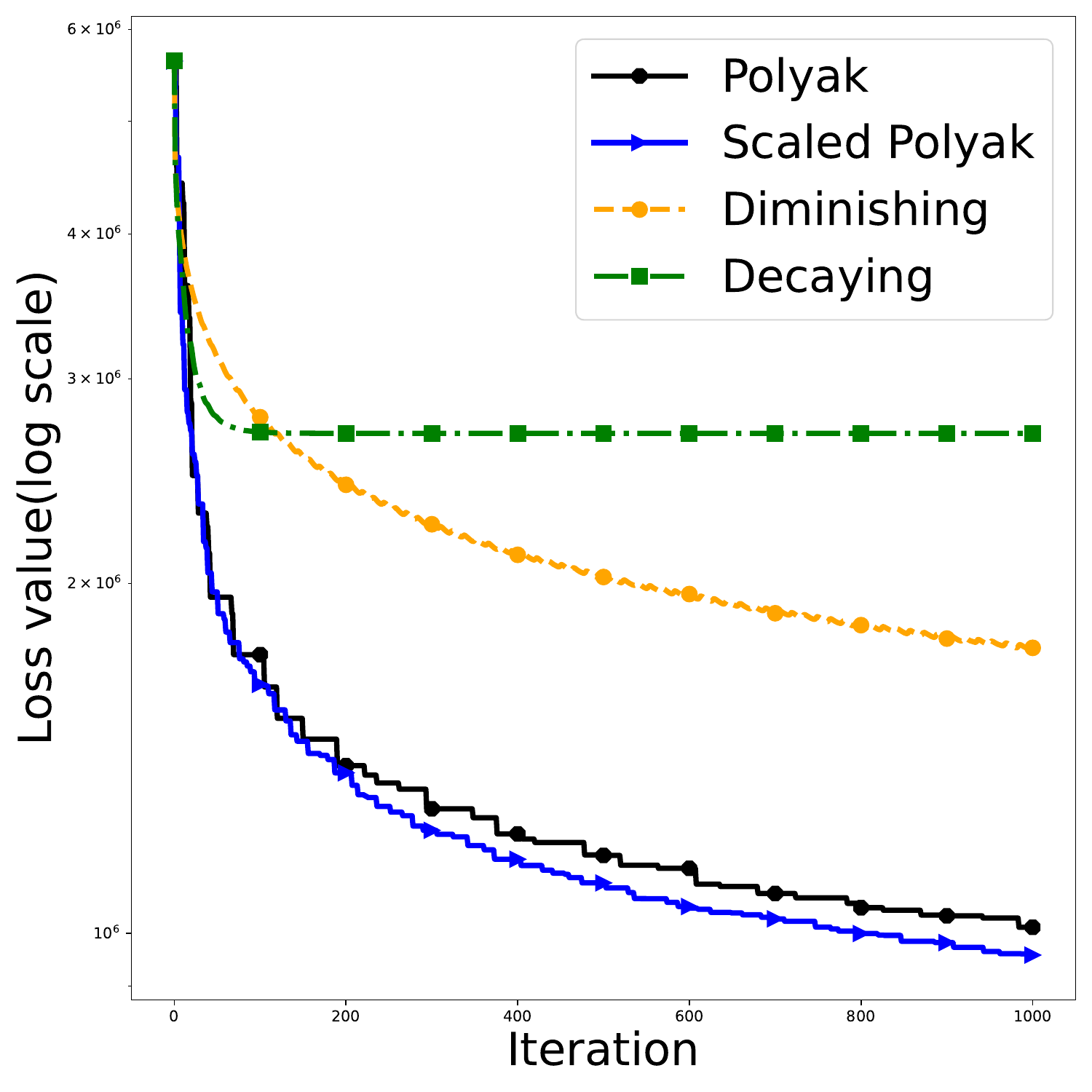}}%
\hspace{0.1cm}
\caption{Qualitative comparison of image inpainting results for four corrupted images (Lighthouse, Man, Kiel, and Houses) using Polyak, scaled Polyak, diminishing, and decaying step-size strategies. Each column shows the reconstructions for the step-size method.}
\label{fig:loss_image_4per}
\end{figure}

\subsubsection{{\bf Matrix compression using RNMF}}
Matrix compression reduces the storage of large matrices and computational demands by approximating them with compact representations that retain essential information while minimizing redundancy; cf. \cite{yuan2005projective}. This process is crucial in fields such as data analysis, machine learning, and scientific computing, where efficient handling of high-dimensional data is vital. Techniques such as low-rank approximations, including nonnegative matrix factorization (NMF), are widely used. These methods decompose the original matrix into two lower-rank factors that, when combined, approximate the original data, enabling efficient storage, faster computations, and, in some cases, enhanced interpretability of the underlying structures.

Here, we applied the projected subgradient method to the reformulated robust matrix completion (RMC) problem in \cref{eq:robustMC}, treating it as a robust nonnegative matrix factorization (RNMF) problem
\begin{equation}\label{eq:RNMF}
\min_{U\geq 0\in\R^{m\times r},V\geq 0 \in\R^{r\times n}} \frac{1}{2} \|X-UV\|_1, 
\end{equation}
using the aforementioned step-size strategies on the 'Man' image. The goal is to compress the image by representing it with two low-rank matrices for different ranks ($r=5,10,50,100$).

The results in \Cref{fig:compress_rec,fig:compress_loss} illustrate the impact of rank on reconstruction quality during matrix compression. Notably, the reconstruction quality is improved across all step-size strategies as the rank increased from 5 to 100. For lower ranks (5 and 10), the Decaying and Diminishing methods are struggling, leaving significant visual artifacts in the reconstructions. In contrast, the scaled Polyak and Polyak methods demonstrate better stability and higher-quality reconstructions, even at lower ranks. All methods improved at higher ranks (50 and 100), but the Scaled Polyak method consistently produced sharper and more accurate results.
\Cref{fig:compress_loss} further highlights the convergence behavior of the step-size strategies across ranks. The Decaying method stagnates early and fails to reduce the loss significantly. The Diminishing method shows some progress, but at a slower rate. However, the Polyak and scaled Polyak methods effectively reduce the loss, with the scaled Polyak consistently achieving the lowest final loss values. These findings demonstrate that while rank significantly impacts reconstruction performance, the Polyak-based methods remain robust across different ranks, rendering them particularly effective for matrix compression tasks.

\begin{figure}[!htb]
\centering
\subfloat[Decaying rank 5.]{\includegraphics[width=3.8cm]{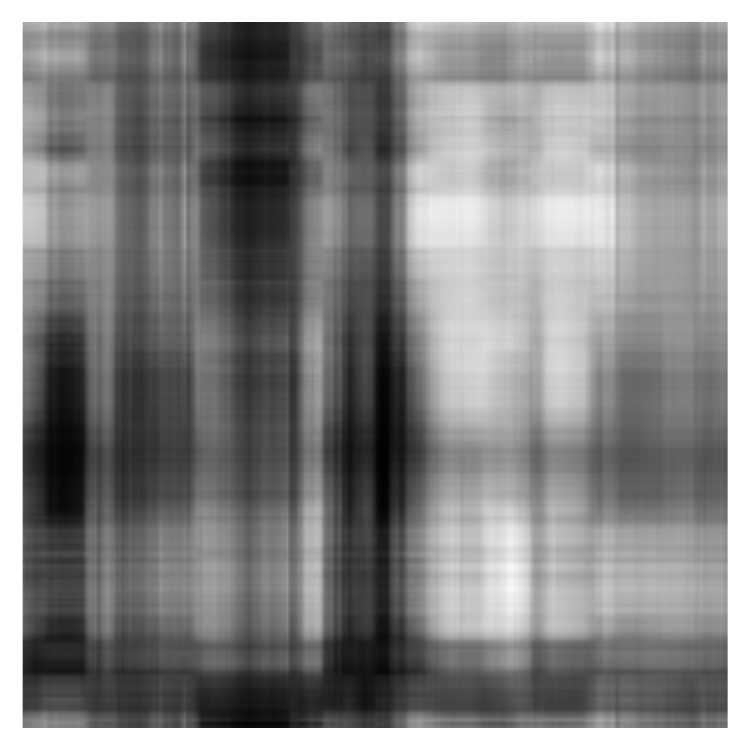}}%
\hspace{0.1cm}
\subfloat[Decaying rank 10.]{\includegraphics[width=3.8cm]{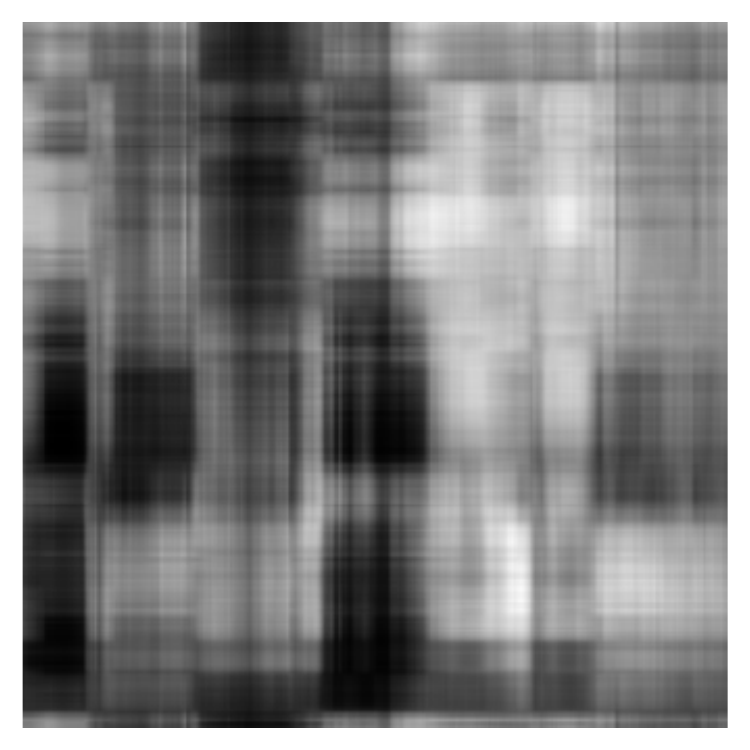}}%
\hspace{0.1cm}
\subfloat[Decaying rank 50.]{\includegraphics[width=3.8cm]{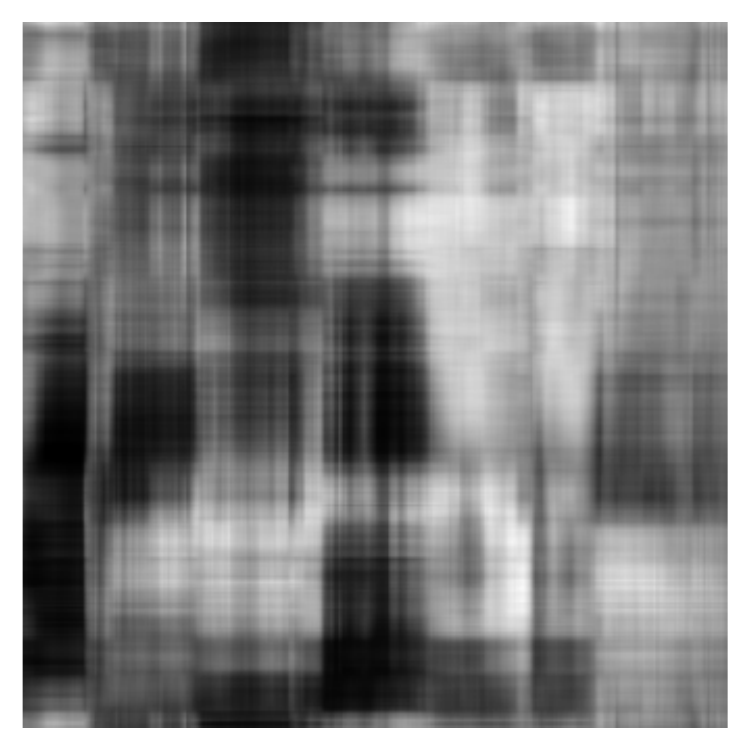}}%
\hspace{0.1cm}
\subfloat[Decaying rank 100.]{\includegraphics[width=3.8cm]{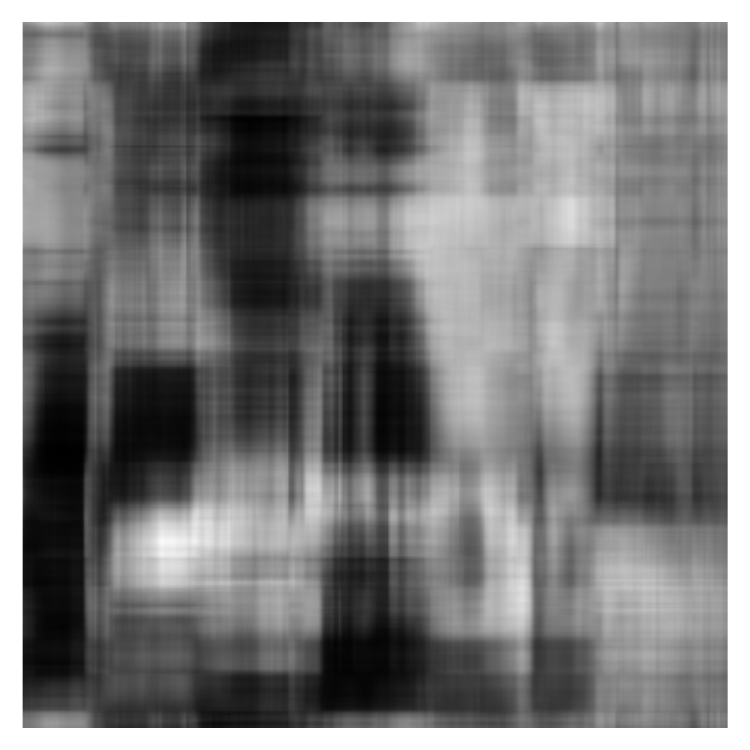}}%
\hspace{0.1cm}
\subfloat[Diminishing rank 5.]{\includegraphics[width=3.8cm]{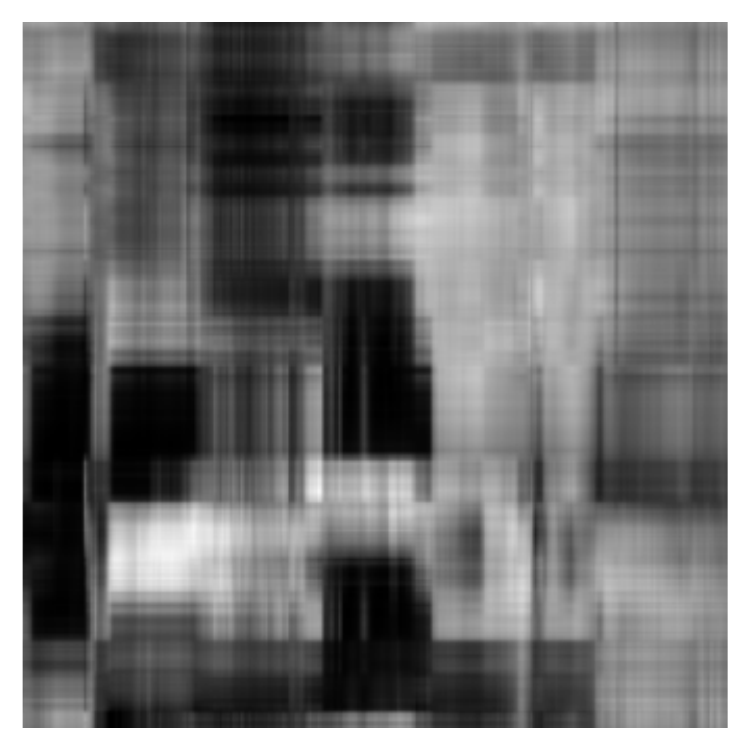}}%
\hspace{0.1cm}
\subfloat[Diminishing rank 10.]{\includegraphics[width=3.8cm]{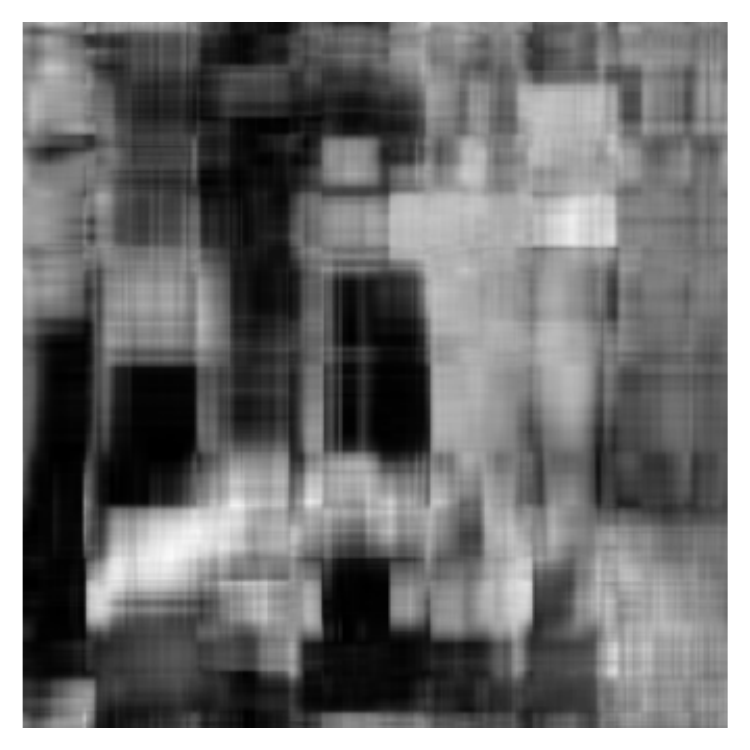}}%
\hspace{0.1cm}
\subfloat[Diminishing rank 50.]{\includegraphics[width=3.8cm]{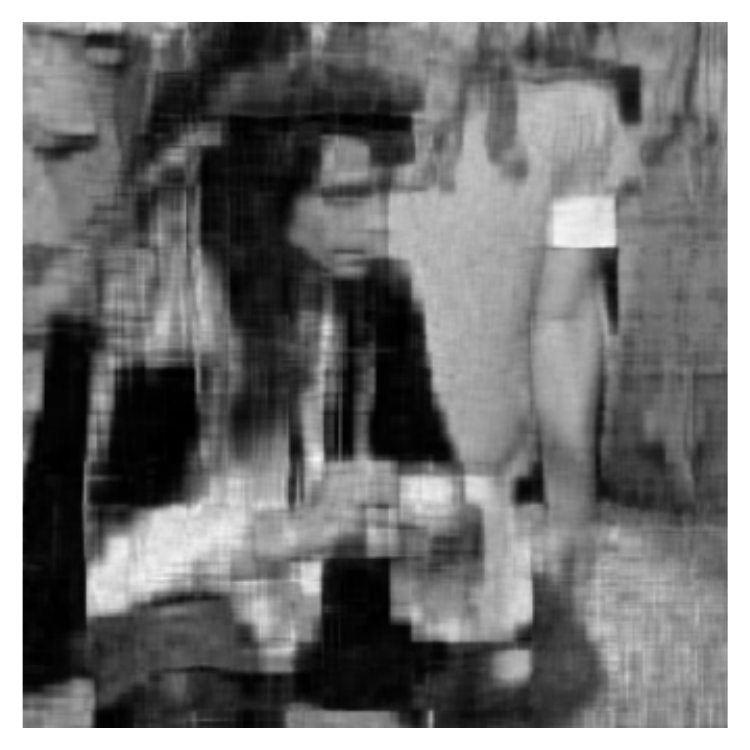}}%
\hspace{0.1cm}
\subfloat[Diminishing rank 100.]{\includegraphics[width=3.8cm]{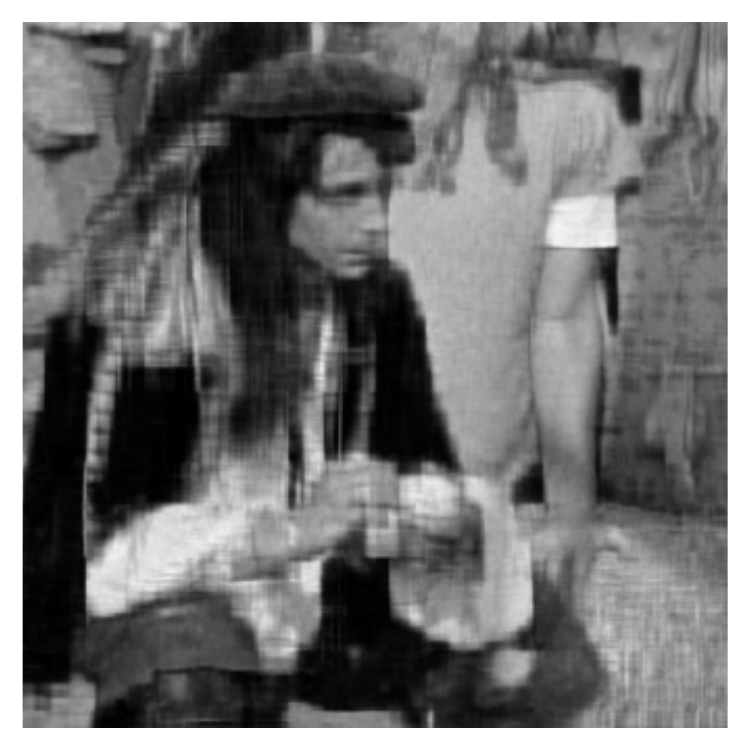}}%
\hspace{0.1cm}
\subfloat[Polyak rank 5.]{\includegraphics[width=3.8cm]{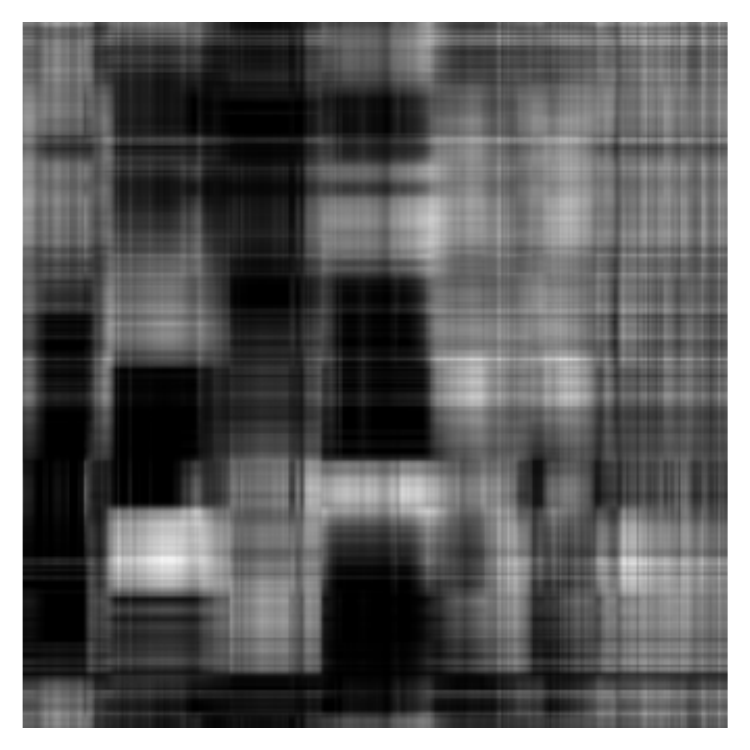}}%
\hspace{0.1cm}
\subfloat[Polyak rank 10.]{\includegraphics[width=3.8cm]{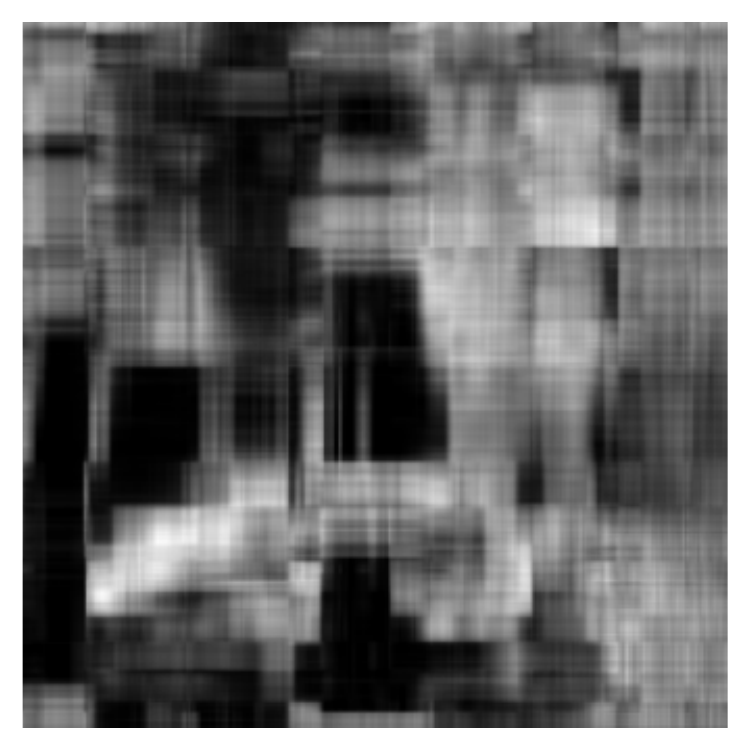}}%
\hspace{0.1cm}
\subfloat[Polyak rank 50.]{\includegraphics[width=3.8cm]{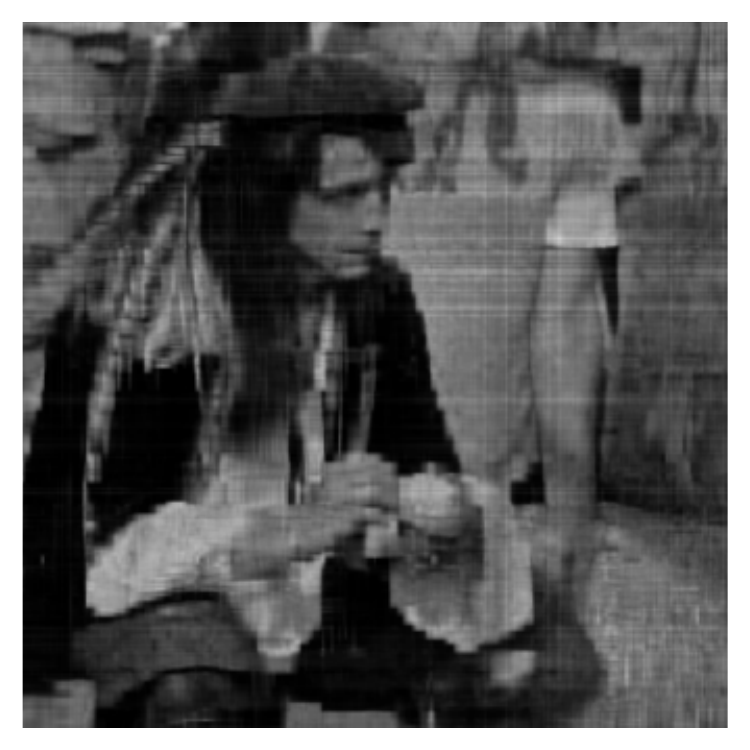}}%
\hspace{0.1cm}
\subfloat[Polyak rank 100.]{\includegraphics[width=3.8cm]{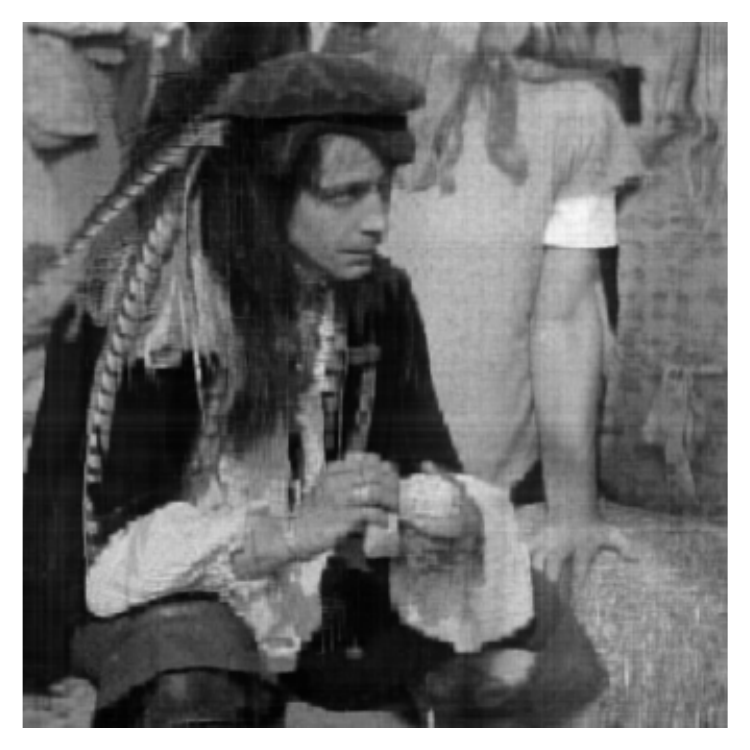}}%
\hspace{0.1cm}
\subfloat[Scaled Polyak rank 5.]{\includegraphics[width=3.8cm]{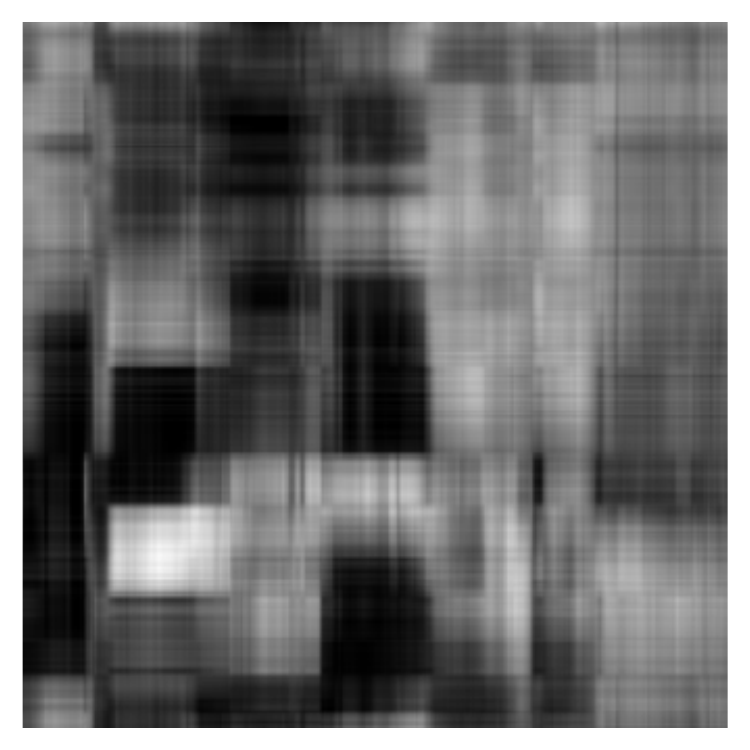}}%
\subfloat[Scaled Polyak rank 10.]{\includegraphics[width=3.8cm]{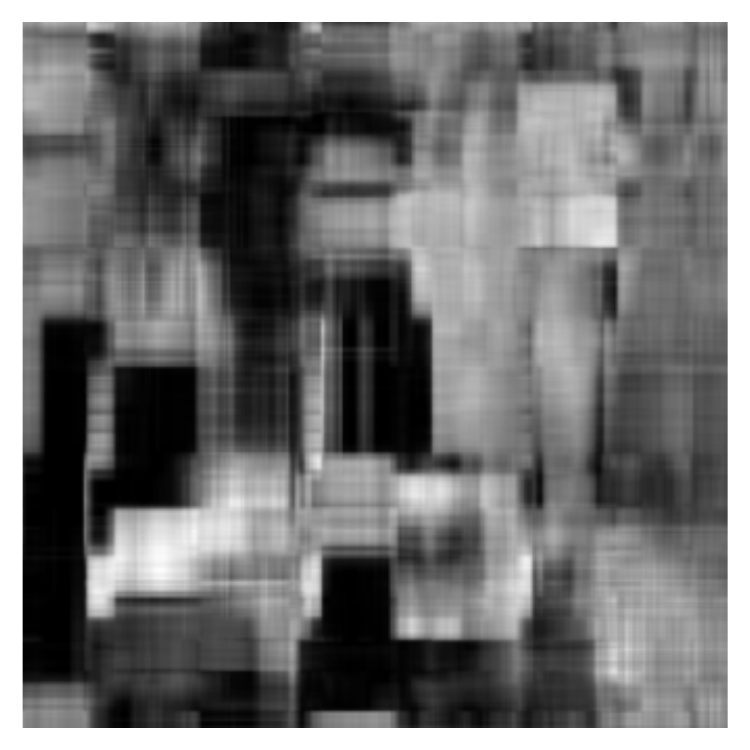}}%
\hspace{0.1cm}
\subfloat[Scaled Polyak rank 50.]{\includegraphics[width=3.8cm]{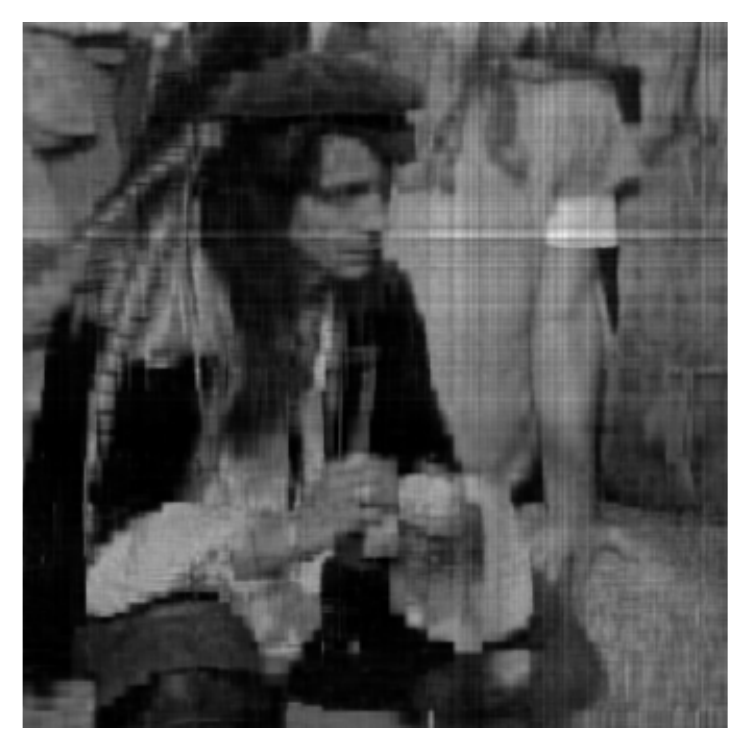}}%
\hspace{0.1cm}
\subfloat[Scaled Polyak rank 100.]{\includegraphics[width=3.8cm]{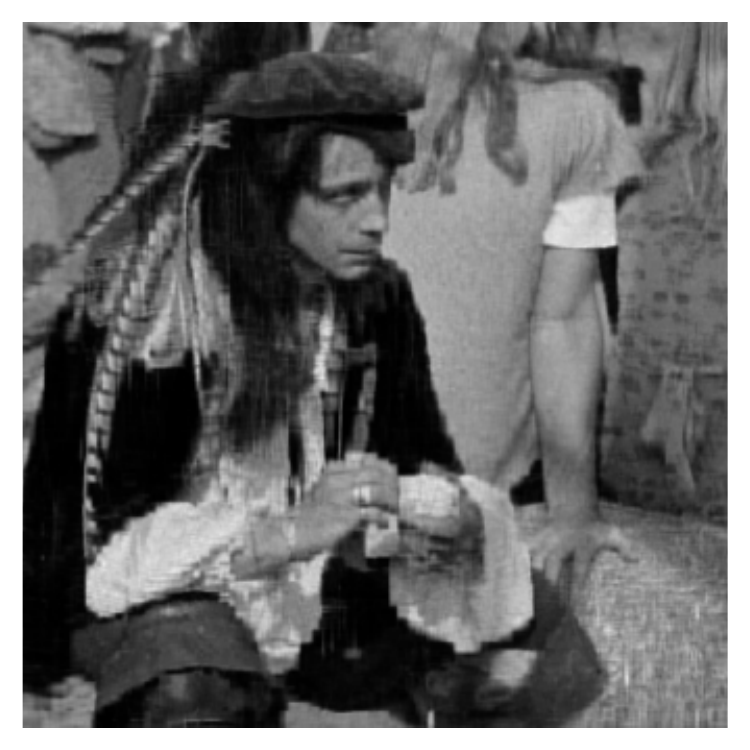}}%
\hspace{0.1cm}
\caption{Comparison of reconstructed images for different ranks ($r=5, 10, 50, 100$) using four step-size strategies: Decaying, Diminishing, Polyak, and Scaled Polyak.}
\label{fig:compress_rec}
\end{figure}

\begin{figure}[!htb]
\centering
\subfloat[Loss values rank 5.]{\includegraphics[width=8.2cm]{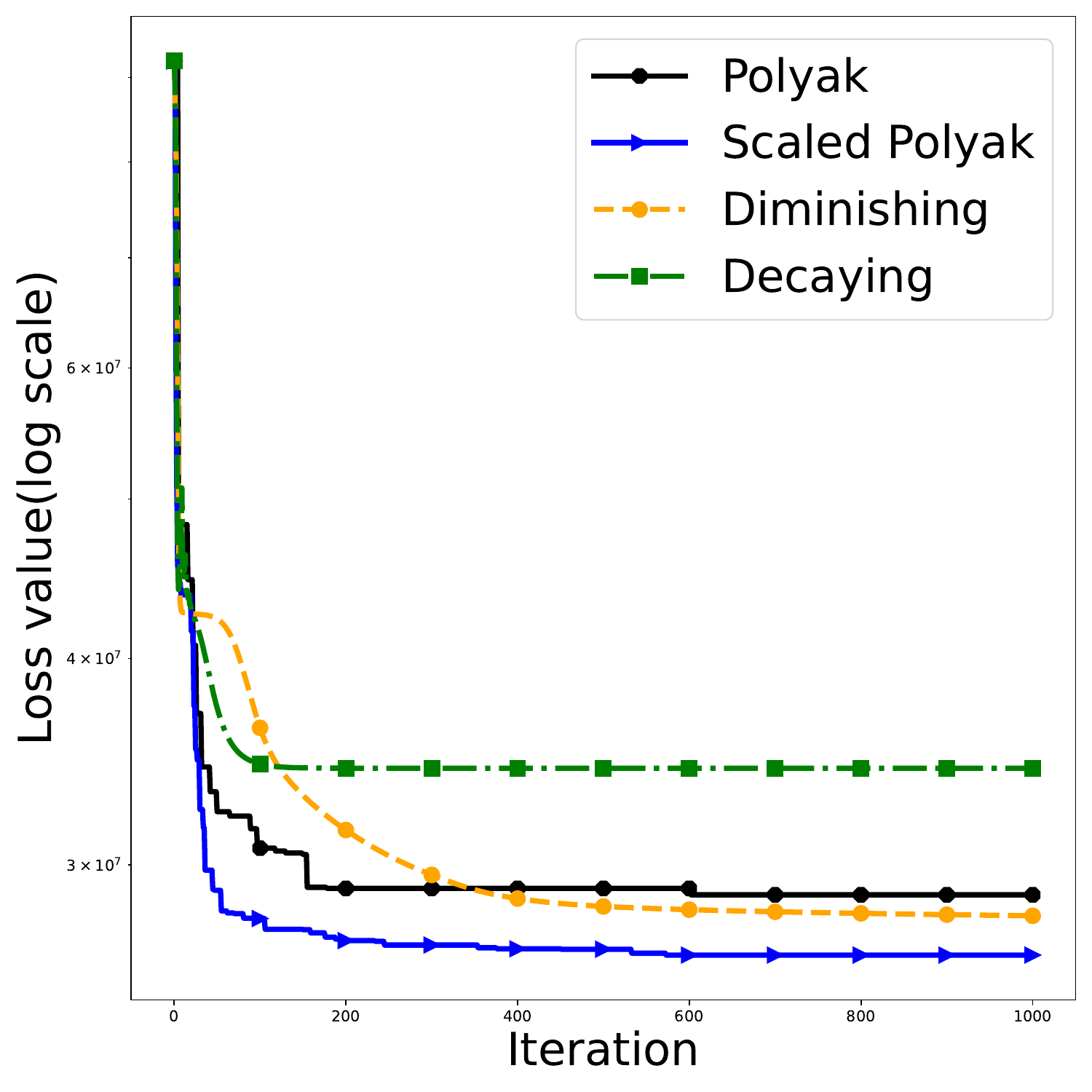}}%
\hspace{0.1cm}
\subfloat[Loss values rank 10.]{\includegraphics[width=8.2cm]{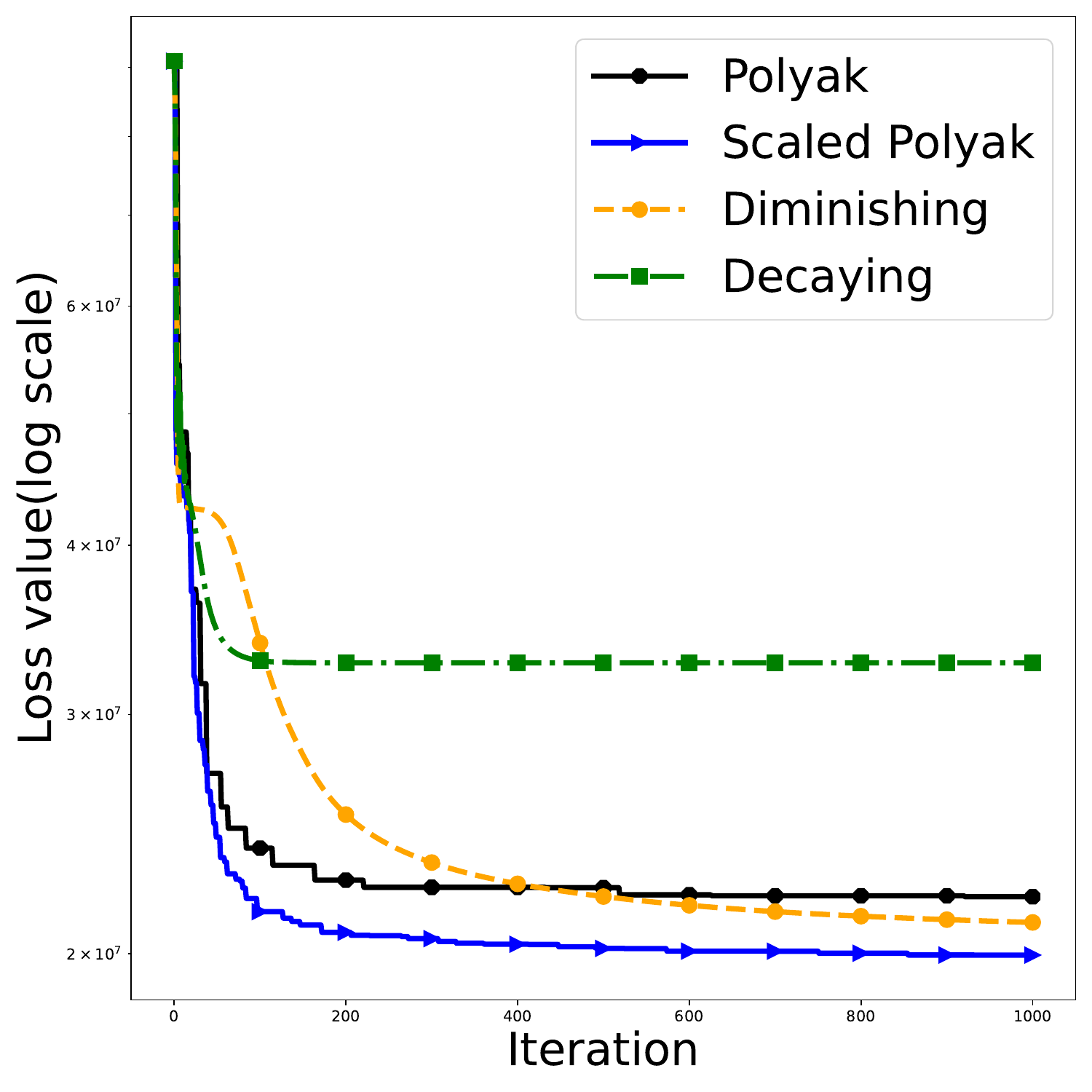}}%
\hspace{0.1cm}
\subfloat[Loss values rank 50.]{\includegraphics[width=8.2cm]{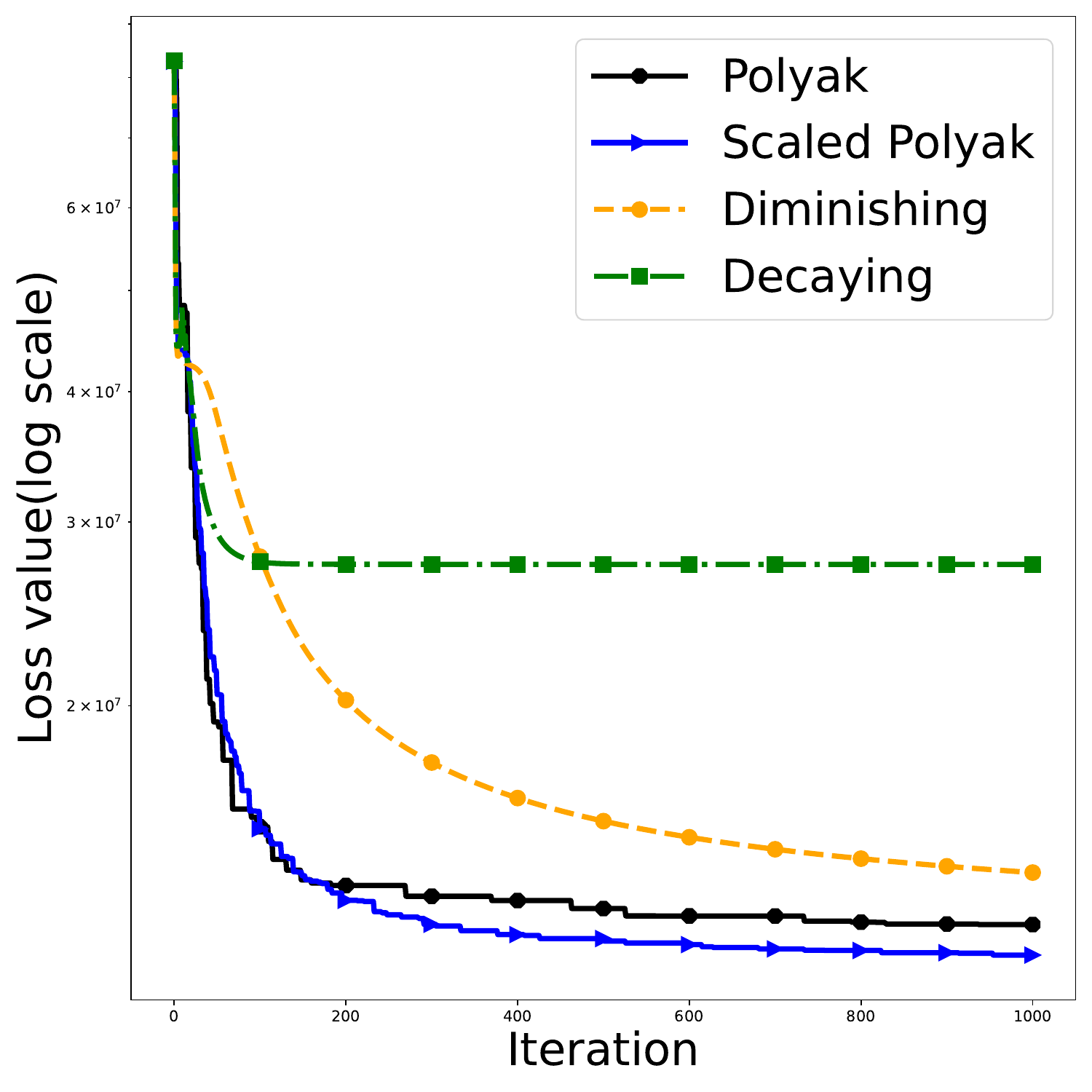}}%
\hspace{0.1cm}
\subfloat[Loss values rank 100.]{\includegraphics[width=8.2cm]{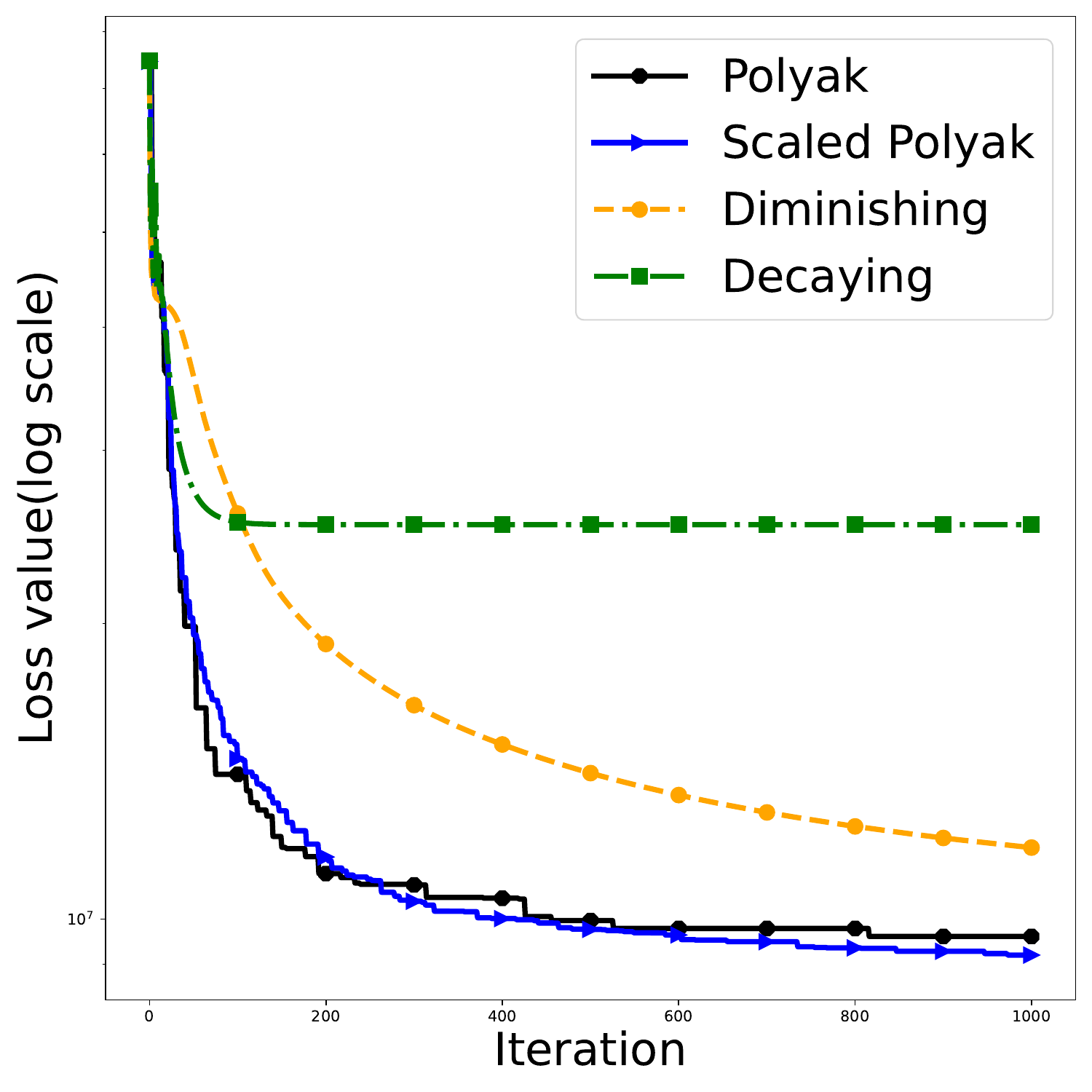}}%
\hspace{0.1cm}
\caption{Convergence plots showing the loss value (log scale) over 1000 iterations for different ranks ($r=5, 10, 50, 100$) using the decaying, diminishing, Polyak, and scaled Polyak step-sizes.}
\label{fig:compress_loss}
\end{figure}

\subsubsection{{\bf Face recognition using RNMF}}

NMF has been extensively applied to face recognition tasks due to its ability to decompose facial images into additive, parts-based representations that align with human facial features' intuitive understanding; cf. \cite{guillamet2002non}. NMF captures localized features such as eyes, noses, and mouths by factorizing a nonnegative data matrix into two nonnegative matrices, facilitating more interpretable and efficient data representations.

Here, we applied the projected subgradient methods to the RNMF problem \cref{eq:RNMF} on the Olivetti Faces dataset. The performance of the considered methods is evaluated using a K-Nearest Neighbors (KNN) classifier for three different neighborhood sizes ($k=1,3,5$). For the reconstruction process, we set the rank $r=128$ and a maximum of iterations to $1000$. Reconstructed images are assessed as quantitative metrics. Two initialization methods were used: Nonnegative SVD-R~\cite{trigeorgis2014deep} (\Cref{fig:loss_FR-1}) and Nonnegative random initialization (\Cref{fig:loss_FR-2}). The results indicate that the choice of starting point plays a crucial role in RNMF performance, as highlighted in \Cref{fig:loss_FR}.
All methods performed well for $k=1$ with Polyak achieving the highest precision, $92.5\%$ (\Cref{tab:rmse_methods}). For ($k=3$), Polyak again shows the best accuracy ($86.2\%$), followed by the scaled Polyak and the diminishing ($85.0\%$), while the decaying dropped to $82.5\%$. For $k=5$, the Polyak and diminishing both achieve ($85.0\%$), the scaled Polyak slightly lags at ($82.5.0\%$), and the decaying drops further to ($78\%$).
These findings suggest that while the Polyak and diminishing methods show stability across different neighborhood sizes, the scaled Polyak tends to have slightly reduced performance as complexity increases, and the decaying method demonstrates clear limitations in preserving discriminative features during dimensionality reduction. Overall, Polyak's step-size strategy stands out for its superior generalization capability in high-dimensional data scenarios, making it a robust choice for complex classification tasks.

\begin{table}[!htbp]
    \centering
    \begin{tabular}{|c|c|c|c|c|}
        \hline
        Step-size&Polyak& \text{Scaled Polyak} &Diminishing& Decaying \\ \hline
        Class 1& 0.925  & 0.913           &0.913        &   0.875\\ \hline
        Class 3& 0.862  & 0.850           &0.850        &   0.825 \\ \hline
        Class 5& 0.850  & 0.825           &0.850      &   0.788\\ \hline
    \end{tabular}
    \caption{Classification accuracy of K-Nearest Neighbors (KNN) for the Olivetti Faces dataset across three neighborhood sizes ($k=1,2,3$) after dimensionality reduction using NMF.}
    \label{tab:rmse_methods}
\end{table}

\begin{figure}[!htbp]
\centering
\subfloat[Polyak's step-size]{\includegraphics[width=7.5cm]{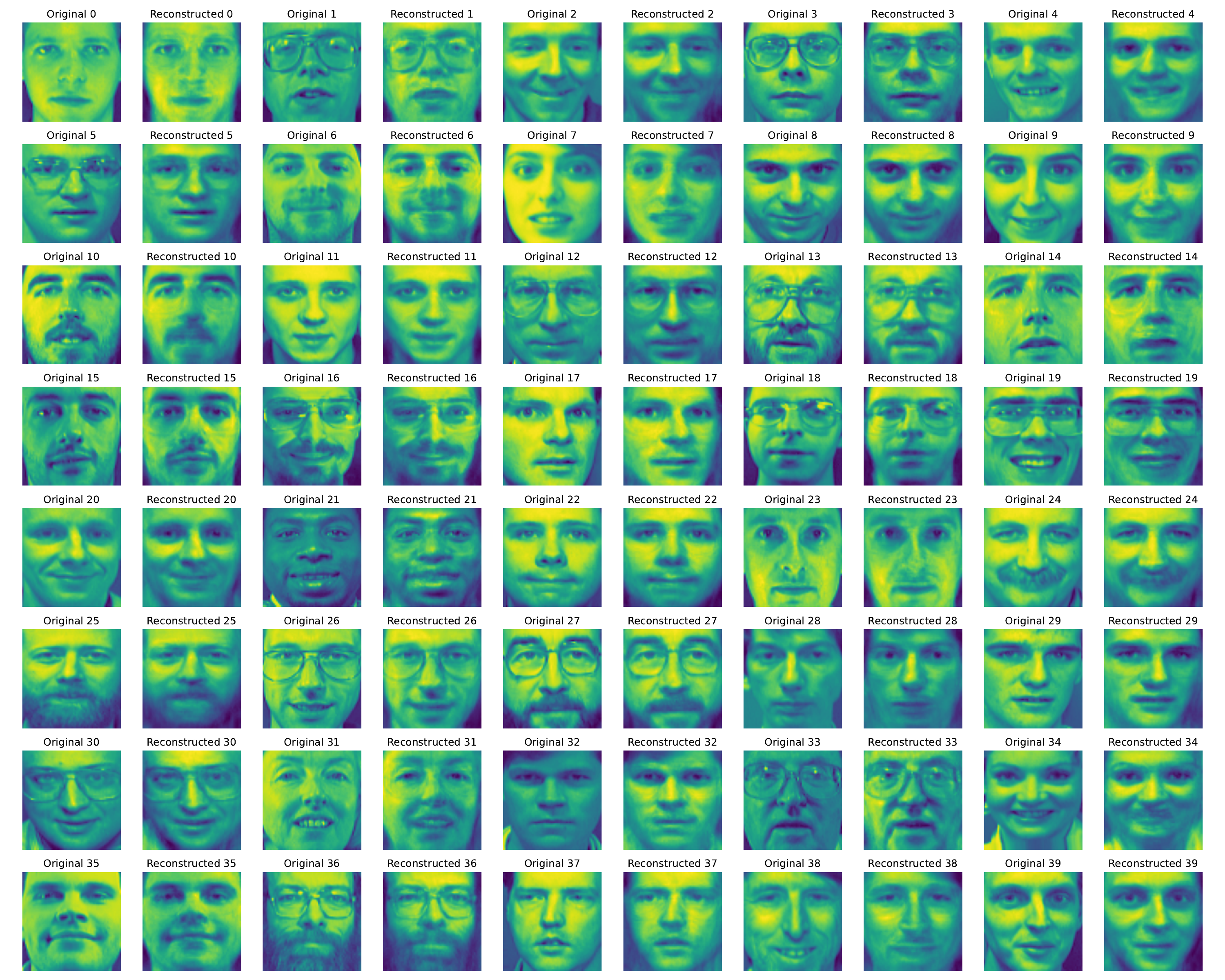}}%
\hspace{0.1cm}
\subfloat[Scaled Polyak's step-size]{\includegraphics[width=7.5cm]{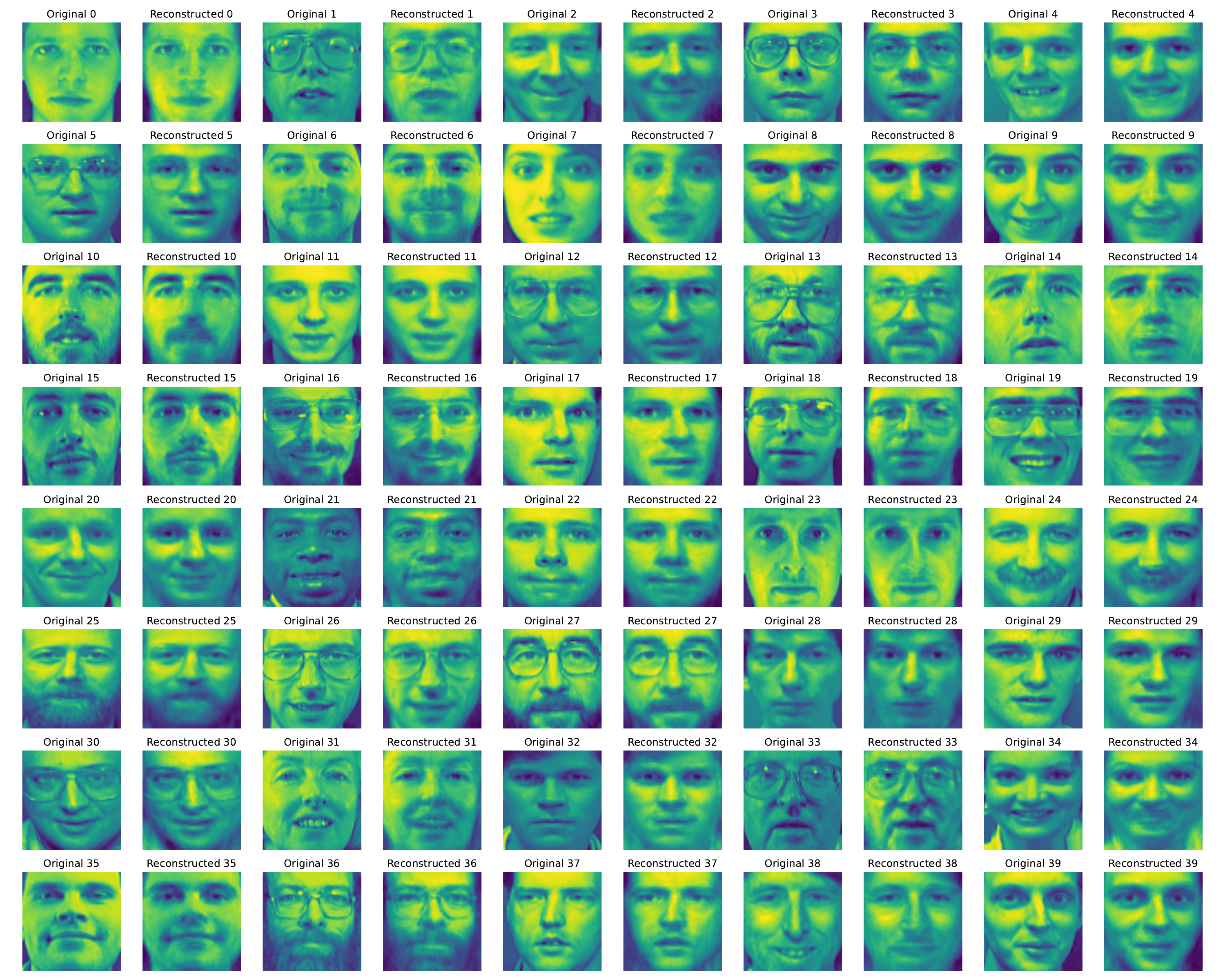}}%
\hspace{0.1cm}
\subfloat[Decaying step-size.]{\includegraphics[width=7.5cm]{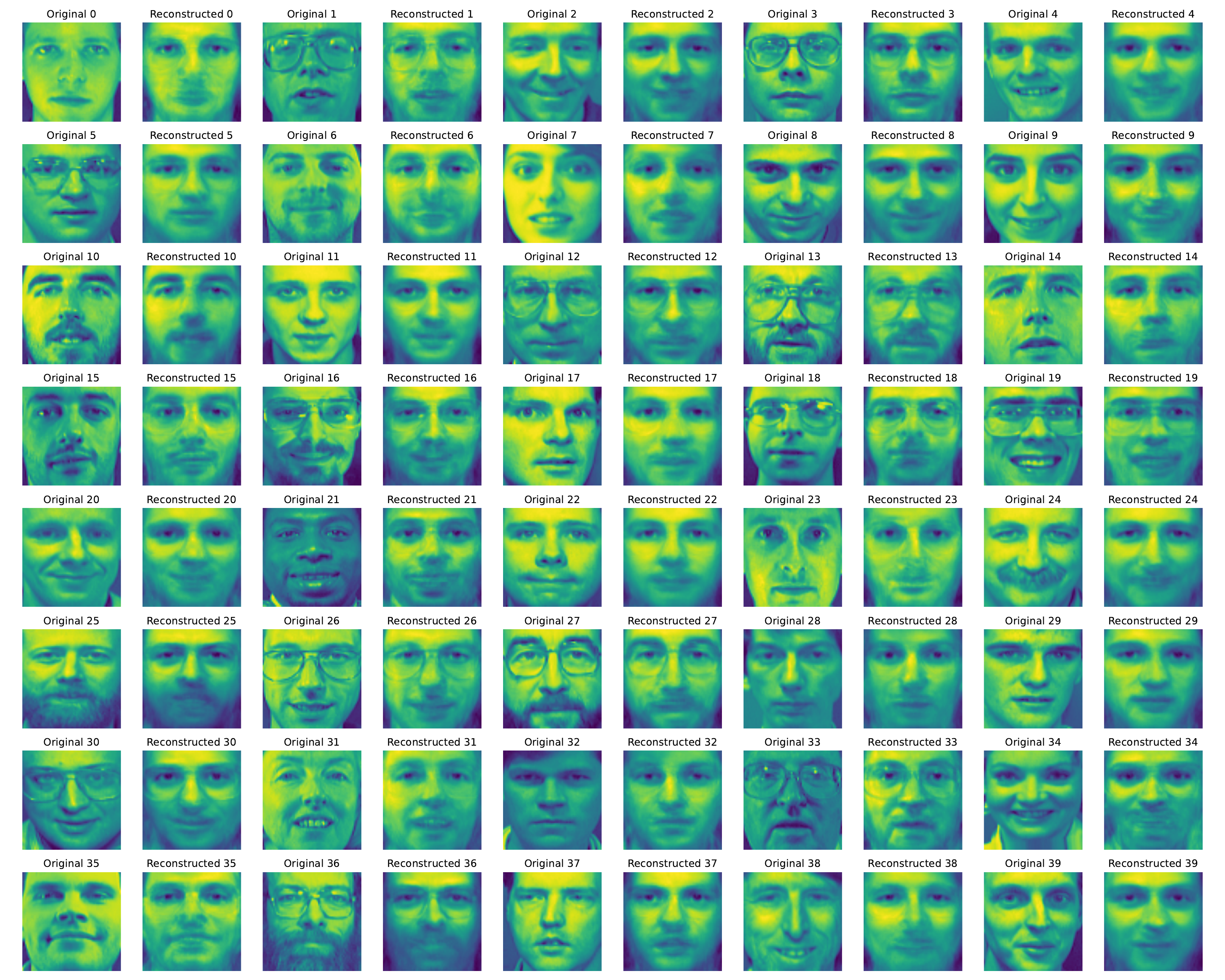}}%
\hspace{0.1cm}
\subfloat[Diminishing step-size.]{\includegraphics[width=7.5cm]{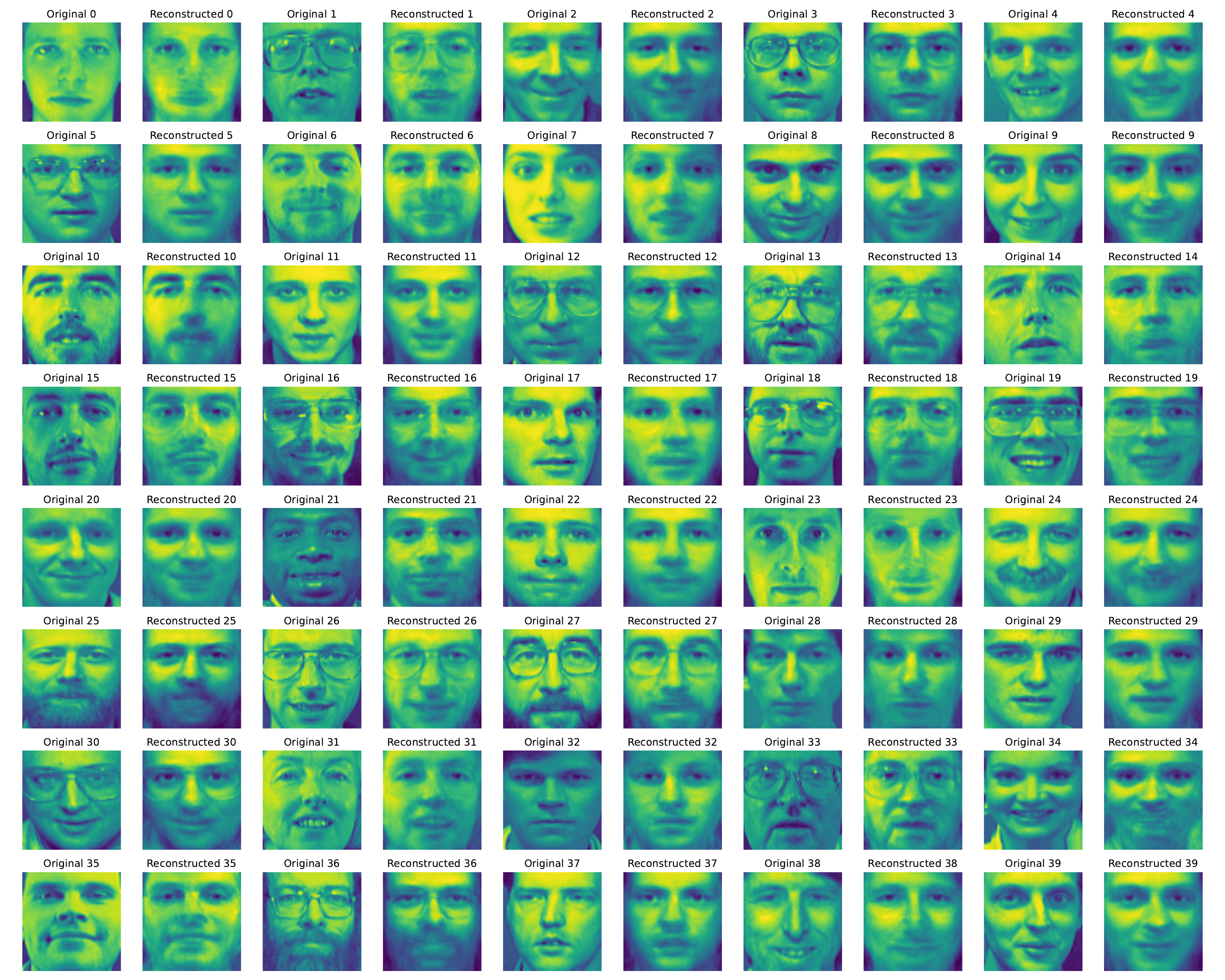}}%
\vspace{-2mm}
\caption{Visual comparison of original and reconstructed images from the Olivetti Faces dataset using four step-size strategies: Polyak, scaled Polyak, decaying, and diminishing. Each subfigure displays original images (left) alongside their reconstructions (right) for 40 selected individuals.}
\label{fig:contour}
\end{figure}

\begin{figure}[!htbp]
\centering
\subfloat[Nonnegative SVD starting point.]{\includegraphics[width=7.5cm]{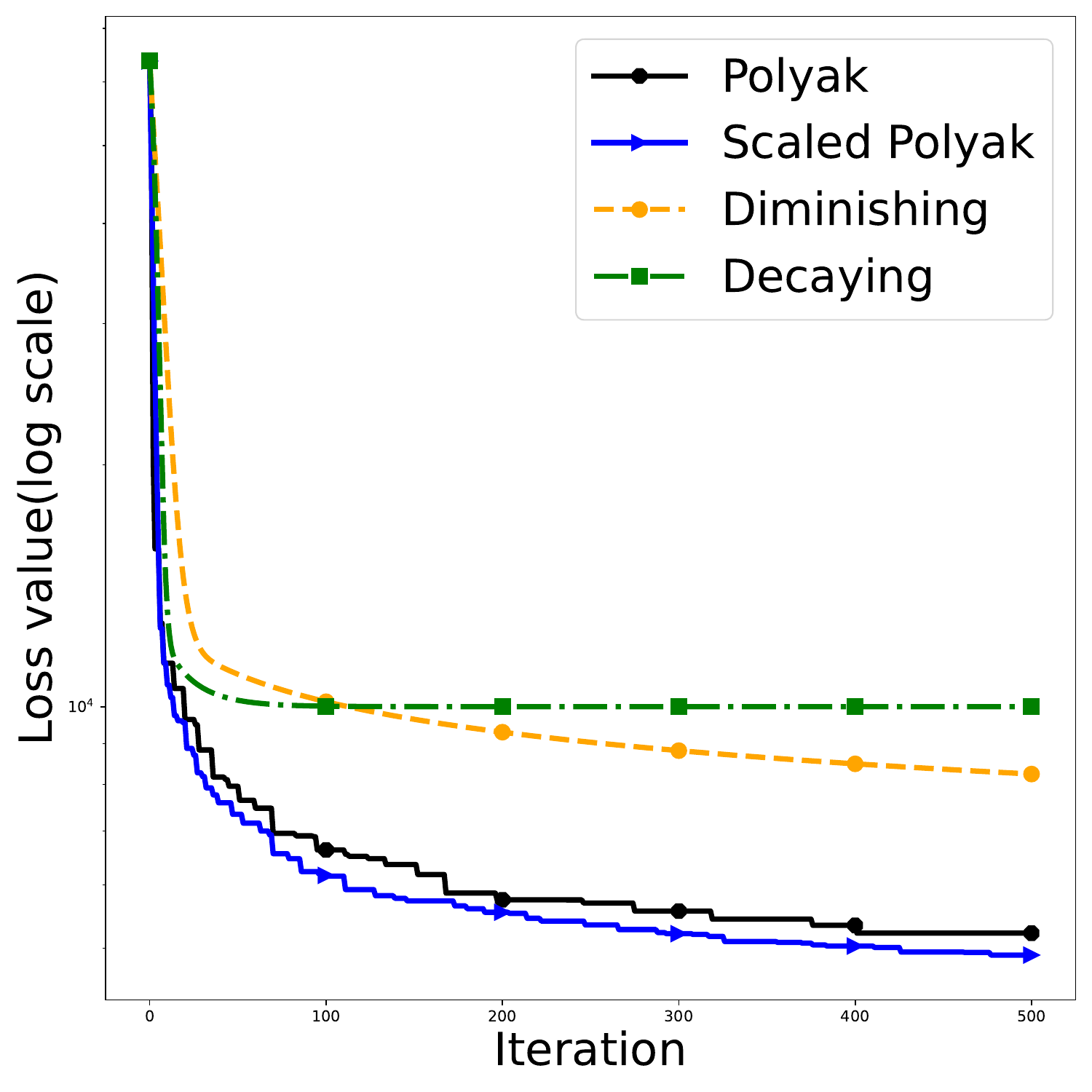}\label{fig:loss_FR-1}}%
\hspace{0.1cm}
\subfloat[Nonnegative random starting point.]{\includegraphics[width=7.5cm]{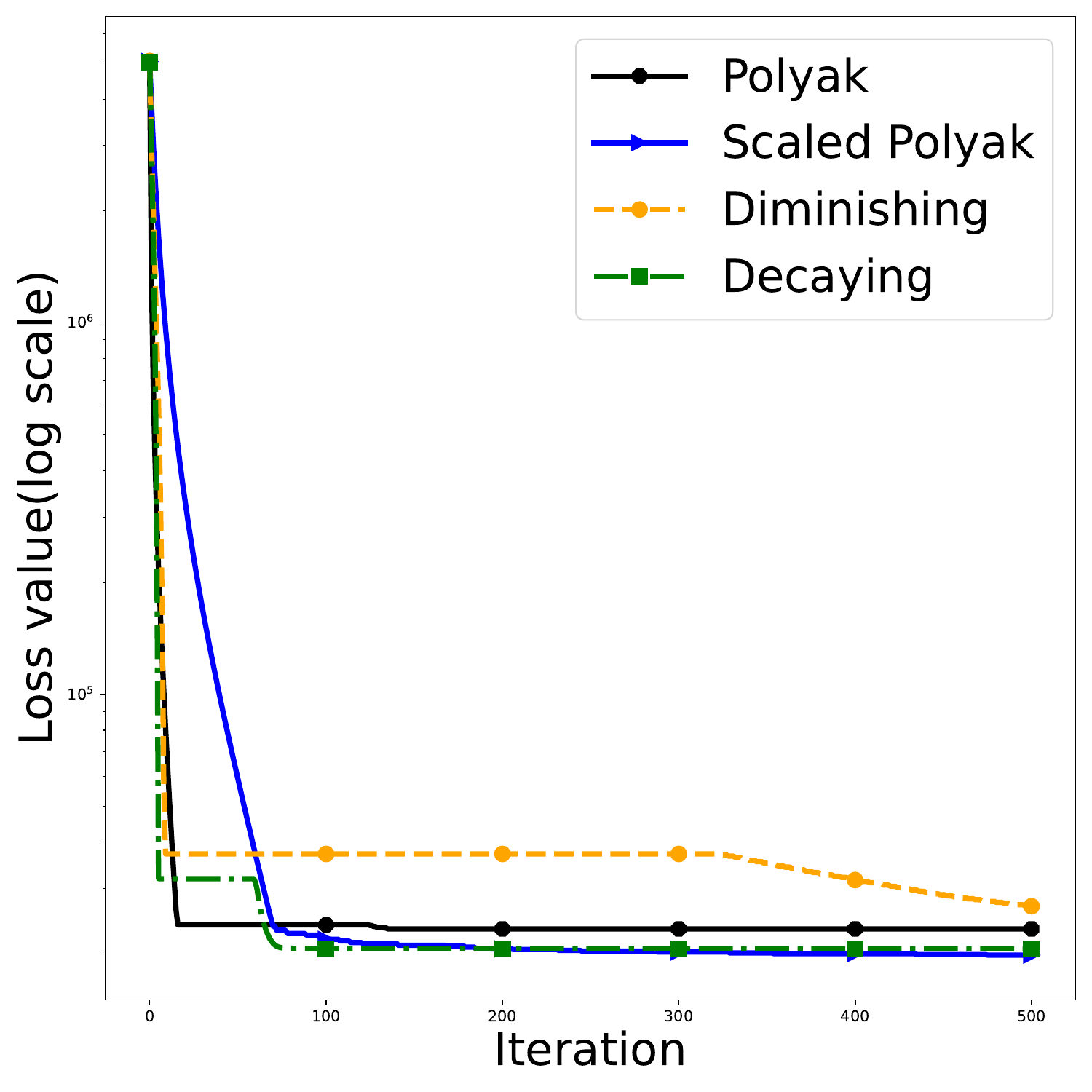}\label{fig:loss_FR-2}}%
\vspace{-2mm}
\caption{Loss convergence (log scale) for different step-size strategies on the Olivetti Faces dataset over 1000 iterations with two different starting points.}
\label{fig:loss_FR}
\end{figure}


\subsection{{\bf Robust image deblurring problem}}

We next illustrate the numerical behavior of the proposed projected subgradient methods on a robust image deblurring problem. Image deblurring can be formulated as the linear inverse problem
\[
y = Ax^\star + e,
\]
where $x^\star \in \mathbb{R}^n$ denotes the original image, $A \in \mathbb{R}^{n\times n}$ represents a blur operator, and $e \in \mathbb{R}^n$ corresponds to additive noise~\cite{jackie2007deblurring}.
To enhance robustness against noise and outliers, we consider
\begin{equation}\label{eq:robust_deblur}
\min_{x \in \mathbb{R}^n} \|Ax - y\|_1 + \lambda \sum_{i=1}^n r(x_i),
\end{equation}
where $r$ is the GMCP function introduced in \Cref{exa:para}~\ref{exa:para-5}, defined by
\[
r(t)=
\begin{cases}
(1+\nu)|t| - |t|^{1+\nu},\quad & \text{if } |t|\le 1,\\
\nu, & \text{if } |t|>1,
\end{cases}
\]
with $\nu \in (0,1]$.
By \cref{pro:relWeakConvCal}~\ref{pro:relWeakConvCal-1} and \cref{thm:S-chara}~\ref{thm:S-chara-2}, the objective function in~\cref{eq:robust_deblur} is $\nu$-paraconvex and satisfies a local H\"olderian error bound.

The forward operator $A$ models spatially invariant blur and is implemented as a two-dimensional convolution with a normalized $5\times5$ averaging kernel using FFTs under periodic boundary conditions. The original image $x^\star$ is the standard \emph{Cameraman} image normalized to $[0,1]$. The observations are generated with a blurred signal-to-noise ratio (BSNR) of $40$ dB.
The parameters are set as
\[
\lambda = 10^{-2}, \qquad \nu = \tfrac{1}{2},
\]
and all methods are run for $150$ iterations with initial point $x_0 = y$.
We compare four step-size strategies: Polyak, scaled Polyak, diminishing, and geometrically decaying.

\Cref{fig:loss_psnr_deblur-1} presents the running minimum of the objective values (log-scale). The Polyak and scaled Polyak strategies decrease the objective more rapidly during the initial iterations, while the diminishing rule exhibits slower convergence. The decaying strategy provides stable and competitive behavior.
\Cref{fig:loss_psnr_deblur} shows the evolution of the PSNR over the iterations. The Polyak-type strategies attain higher PSNR values, reaching approximately $29$ dB after $150$ iterations, whereas the diminishing rule yields lower reconstruction quality.
The reconstructed images obtained after $150$ iterations are displayed in \Cref{fig:im_deblur}. The Polyak and scaled Polyak strategies produce visually sharper reconstructions, while the diminishing rule results in noticeable over-smoothing.

\begin{figure}[!htbp]
\centering
\subfloat[Original image $x$.]{\includegraphics[width=5.5cm]{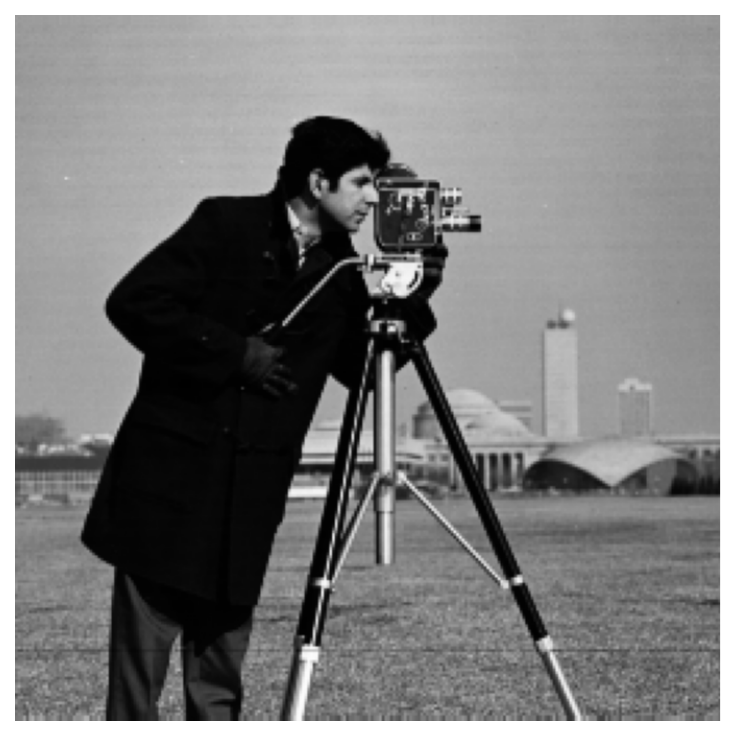}}%
\hspace{0.1cm}
\subfloat[Blurred image $y$.]{\includegraphics[width=5.5cm]{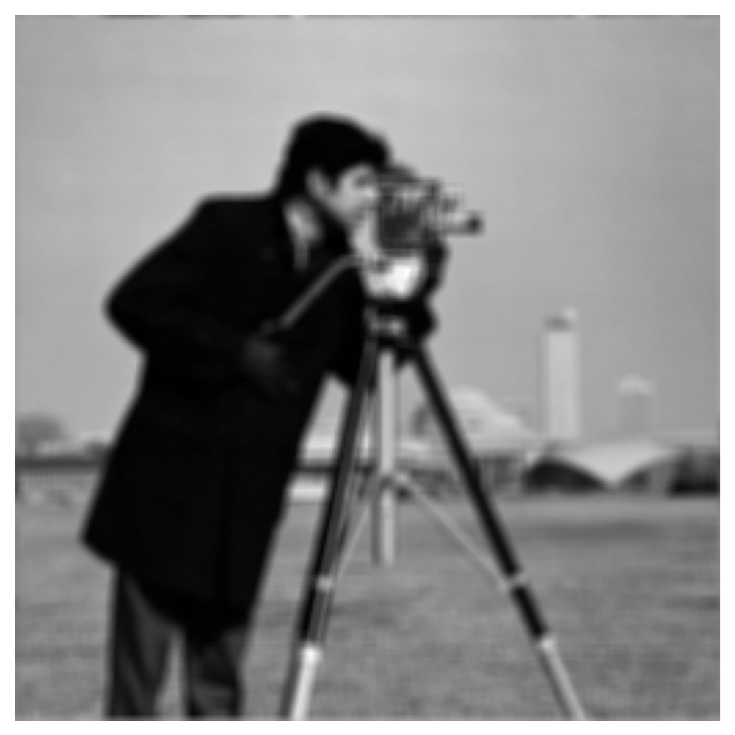}}%
\hspace{0.1cm}
\subfloat[Decaying (PSNR $=27.88$ dB).]{\includegraphics[width=5.5cm]{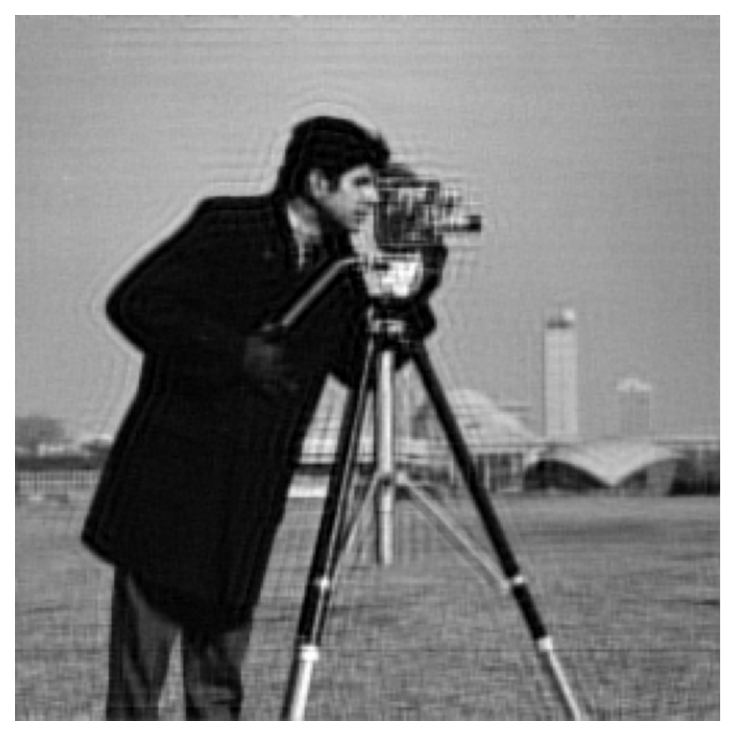}}%
\hspace{0.1cm}
\subfloat[Diminishing (PSNR $=24.19$ dB).]{\includegraphics[width=5.5cm]{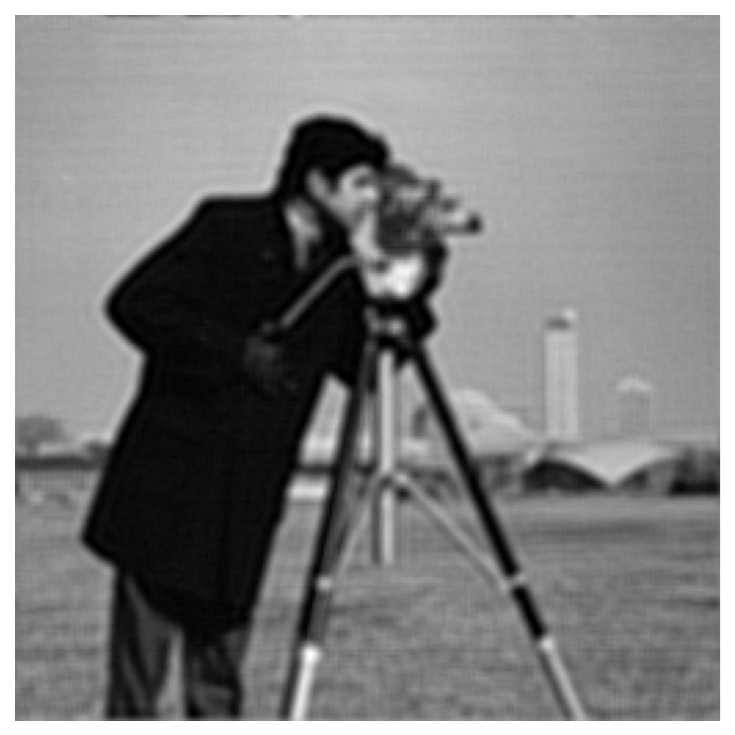}}%
\hspace{0.1cm}
\subfloat[Polyak (PSNR $=29.26$ dB).]{\includegraphics[width=5.5cm]{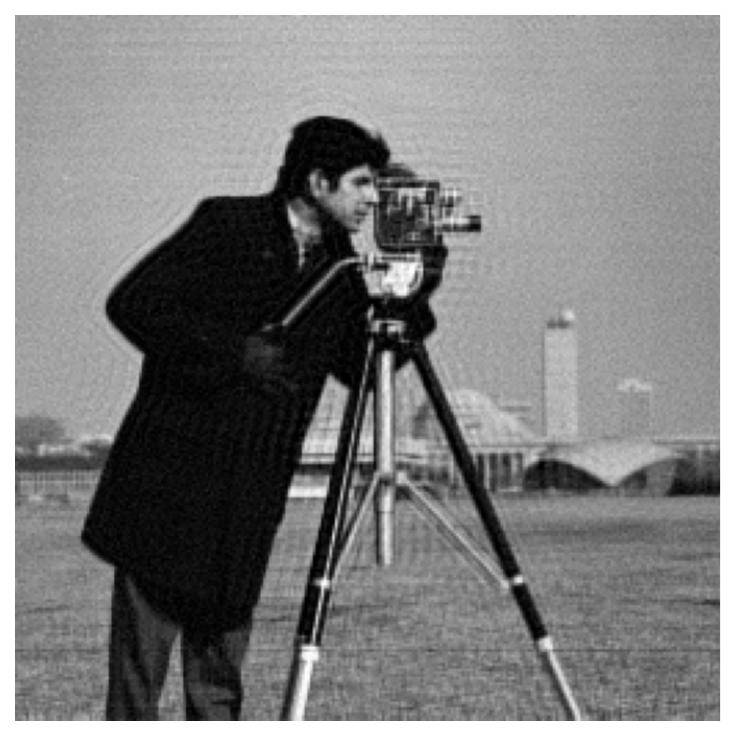}}%
\hspace{0.1cm}
\subfloat[Scaled Polyak  (PSNR $=28.86$ dB).]{\includegraphics[width=5.5cm]{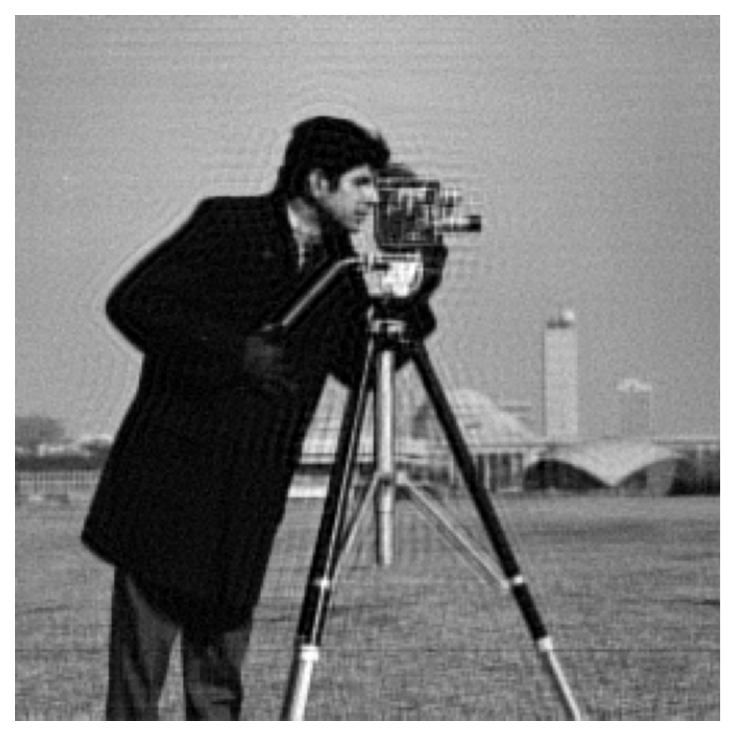}}%
\hspace{0.1cm}
\caption{Recovered images after 150 iterations for the robust image deblurring problem.}
\label{fig:im_deblur}
\end{figure}
%
\begin{figure}[!htbp]
\centering
\subfloat[Loss values (log-scale) over 150 iterations.]
{\includegraphics[width=8.1cm]{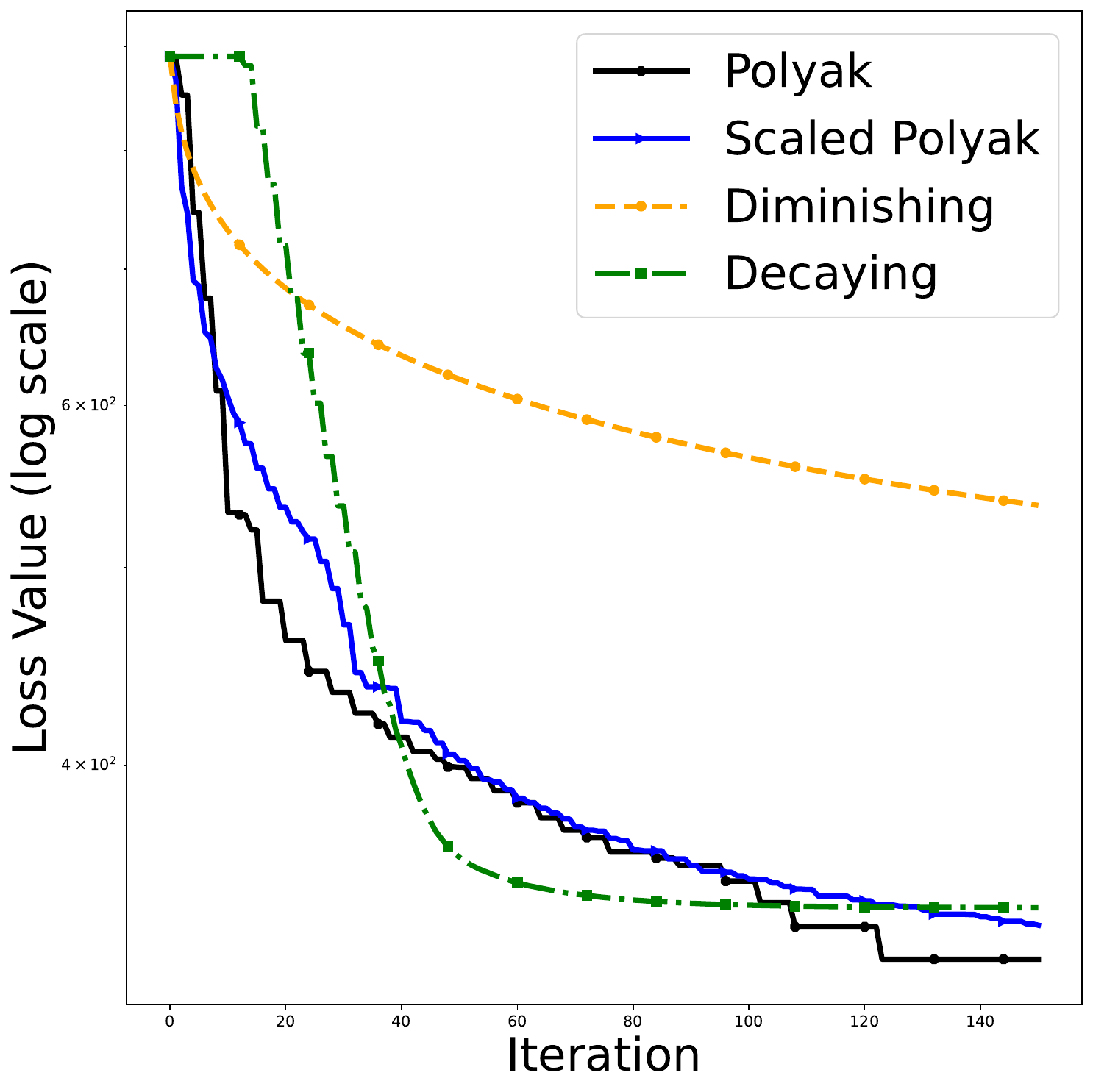}\label{fig:loss_psnr_deblur-1}}
\hspace{0.2cm}
\subfloat[PSNR values over 150 iterations.]
{\includegraphics[width=8.1cm]{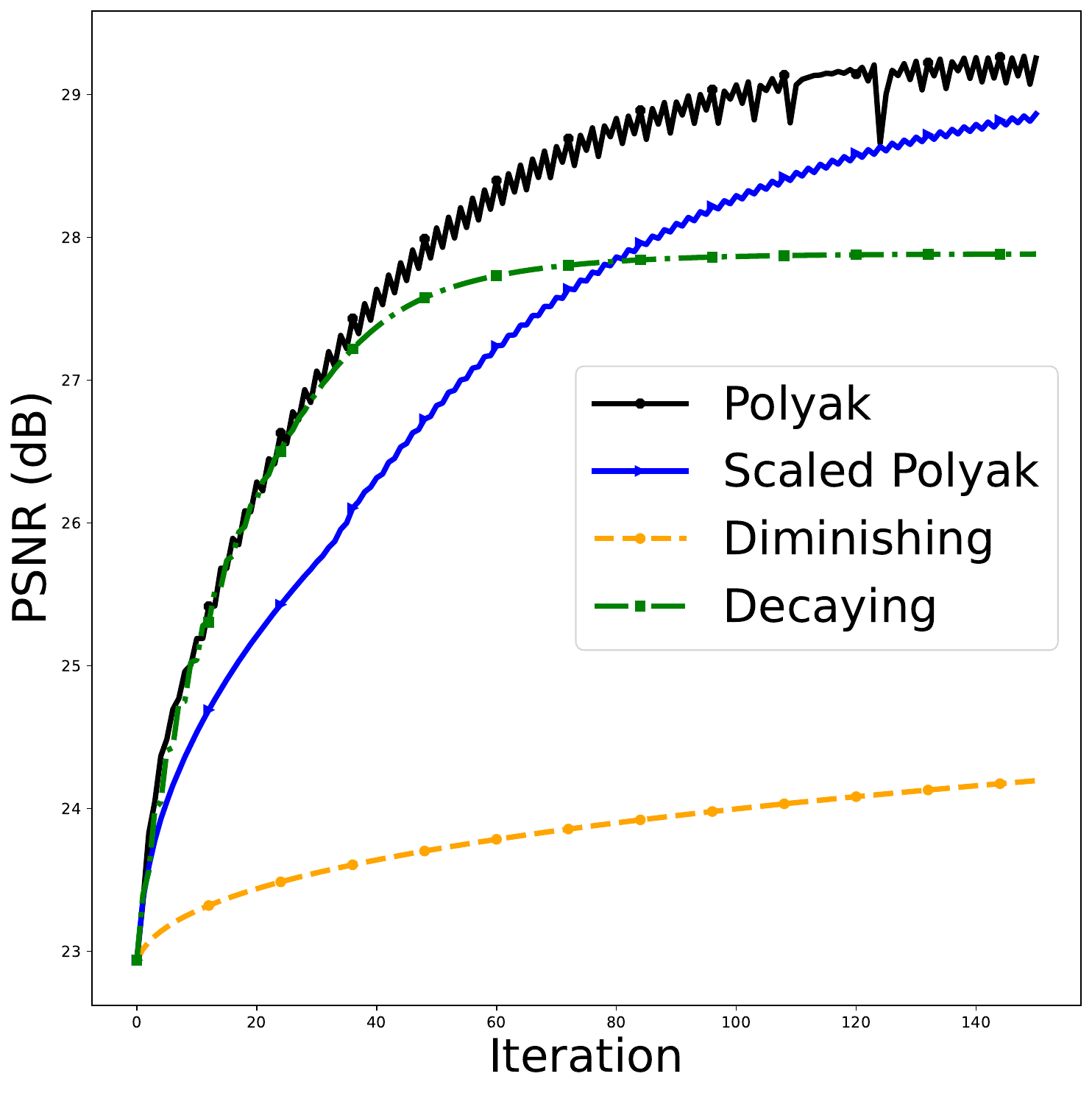}}
\caption{Cameraman deblurring.
(a) Objective value versus iteration (log scale).
(b) PSNR (dB) versus iteration.
Comparison of Polyak, scaled Polyak, diminishing, and decaying step-size strategies.}
\label{fig:loss_psnr_deblur}
\end{figure}

\section{Concluding remarks}\label{sec:Conclusion}

In this paper, we studied the projected subgradient methods for paraconvex functions under the H{\"o}lderian error bound condition. In the first step, we investigate the properties and characterizations of paraconvex functions, including their local Lipschitz continuity. Notably, we established that $\nu$-paraconvexity coincides with 
$\nu$-weak convexity on convex and compact sets. In addition, under the H{\"o}lderian error bound condition, we identified regions around the optimal set that include no saddle points. 

Next, we described the generic projected subgradient method with several step-sizes (i.e., constant, nonsummable diminishing, square-summable but not summable, geometrically decaying, and Scaled Polyak's step sizes). It was demonstrated that the projected subgradient method with a constant step-size achieves linear convergence up to a fixed threshold. Furthermore, we demonstrated subsequential and global convergence for nonsummable and square-summable but not summable step-sizes, respectively, while establishing linear convergence for geometrically decaying and Scaled Polyak's step-sizes.

Finally, we assessed the practical effectiveness of the proposed methods in robust low-rank matrix recovery problems, such as robust matrix completion, image inpainting, robust nonnegative matrix factorization, matrix compression, and robust nonnegative matrix factorization, and robust image deblurring.
Our preliminary numerical experiments highlighted the potential of our methods for nonsmooth and nonconvex problems that lack a suitable structure. Notably, the Scaled Polyak's step-size consistently outperformed the other considered algorithms, highlighting its practical value and potential for broader adoption in nonconvex optimization applications, e.g., solving high-order and non-Euclidean proximal-point subproblems; cf. \cite{ahookhosh2024high,ahookhosh2021bregman,kabgani2025moreau,kabgani2024itsopt}.


	

	\ifarxiv
		\bibliographystyle{plain}
	\else
		\phantomsection
		\addcontentsline{toc}{section}{References}
		\bibliographystyle{spmpsci}
	\fi
	\bibliography{Bibliography}

\end{document}